\documentclass{article}
\usepackage[utf8]{inputenc}
\usepackage{amsfonts}
\usepackage{appendix}
\usepackage{amsthm}
\usepackage{soul}

\usepackage{amsmath}
\usepackage{dsfont}
\usepackage{amssymb}
\usepackage{graphicx}
 \usepackage{color, xcolor}
\usepackage{array}
\usepackage{fullpage}
\usepackage{caption}
\usepackage{lineno}
\usepackage{subcaption}
\newcolumntype{R}[1]{>{\raggedleft\arraybackslash }b{#1}}
\newcolumntype{L}[1]{>{\raggedright\arraybackslash }b{#1}}
\newcolumntype{C}[1]{>{\centering\arraybackslash }b{#1}}
\newtheorem{Def}{Definition}
\newtheorem{thm}{Theorem}
\newtheorem{prop}{Proposition}
\newtheorem{lemma}{Lemma}
\newtheorem{corollary}{Corollary}
\newtheorem{rem}{Remark}

\newcommand{\Td}{\mathbb{T}^{d-1}}
\newcommand{\Tnd}{\mathbb{T}_{N}^{d-1}}
\newcommand{\R}{\mathbb{R}}
\newcommand{\Z}{\mathbb{Z}}
\newcommand{\N}{\mathbb{N}}

\newcommand{\whr}{\widehat{\rho}}
\newcommand{\whg}{\widehat{G}}

\newcommand{\intT}{\int_{0}^T}
\newcommand{\whsN}{\widehat{\Sigma}_{N}}

\newcommand{\xiw}{(\xi,\omega)}

\newcommand{\ddd}{2d(\lambda_1-\lambda_2)}
\newcommand{\e}{\eta}

\newcommand{\xek}{(x,x+e_{k})}

\newcommand{\nualph}{\nu_{\widehat{\alpha}}^N}

\newcommand{\xiwk}{\xi^{x,x+e_{k}},\omega^{x,x+e_{k}}}
\newcommand{\xit}{(\xi_t,\omega_t)}
\newcommand{\wht}{\widehat{\theta}}
\newcommand{\ttl}{\theta_{\ell}}
\newcommand{\ttr}{\theta_{r}}
\newcommand{\tr}{\text{Tr}}

\numberwithin{equation}{section}

\usepackage{geometry} 
\usepackage{soul}
\setstcolor{red}
\title{\textbf{ Hydrodynamic and hydrostatic limit for a generalized contact process with mixed boundary conditions
}}
\date{}

\usepackage{authblk}

\newsavebox\affbox
\author[1]{\text{M. Mourragui}}
\author[2]{\text{E. Saada}}
\author[3]{\text{S. Velasco}}
\affil[1] {LMRS, UMR 6085,\textit{Université de Rouen Normandie, Avenue de l'Université, BP.12, Technopôle du Madrillet, F76801 Saint-Etienne-du-Rouvray, France.}}
\affil[2]{ CNRS, UMR 8145, \textit{ Laboratoire MAP5, Université Paris Cité, 45 rue des Saints-Pères, 75270 Paris Cedex 06, France.}
}
\affil[3]{ Laboratoire MAP5, \textit{ Université Paris Cité, 45 rue des Saints-Pères, 75270 Paris Cedex 06, France}
}

\begin{document}

\maketitle

\begin{abstract} \noindent
		\textbf{Abstract }We consider an interacting particle system which models the sterile insect technique. It is the superposition of a  generalized contact process with exchanges of particles on a finite cylinder with open boundaries (see \textit{Kuoch et al., 2017}). We first show that when the system is in contact with reservoirs at different slow-down rates, the hydrodynamic limit is a set of coupled non linear reaction-diffusion equations with mixed boundary conditions. We then prove the hydrostatic limit when the macroscopic equations exhibit a unique attractor. \\ \\
		\let\thefootnote\relax\footnotetext{
			\small $^{*}$\textbf{Corresponding author.} \textit{
				\textit{E-mail addresses: mustapha.mourragui@univ-rouen.fr (M. Mourragui), ellen.saada@parisdescartes.fr (E. Saada), sonia.velasco@parisdescartes.fr (S. Velasco).  }}
			
		}
		\noindent \textbf{\textit{Keywords}}: \textit{Hydrodynamic limit; Hydrostatic limit; Random reservoirs; System of reaction-diffusion equations; Generalized contact process; Stirring process; Mixed boundary conditions. \\\\}
    \textbf{\textit{MSC}}: \noindent\textit{60K35; 82C22}
	\end{abstract}
\noindent\rule{15cm}{0.4pt}
\section{Introduction}
In this paper, we consider the interacting particle system introduced in \cite{kuoch:hal-01100145} to model the sterile insect technique. This technique was developed, among others, by E. Knipling (see \cite{knipling_possibilities_1955}) to eradicate New World screw worms in the 1950s, a serious pest for warm blooded animals. The method is still used today, for instance in France, to protect crops from the very invasive Mediterranean flies, and it is also being tested to fight mosquitoes which transmit dengue in countries like Panama or Brazil. We refer to \cite{TIS} for a detailed list of trials and programs regarding that method. The sterile insect technique works as follows: male insects are sterilized in captivity using gamma rays. They are then released in the wild population, where females mate only once, giving rise to no offspring if they mate with a sterile male. When enough sterile individuals are released, the wild population eventually becomes extinct. From a mathematical perspective, the sterile insect technique has mainly been modelled in a deterministic way through the study of partial differential equations (see \cite{Vauchelet}).

This technique was studied from a probabilistic perspective in \cite{kuoch:hal-01100145}  and \cite{KMS} using interacting particle systems.  In \cite{kuoch:hal-01100145}, a phase transition result was proved at the microscopic level. Recently, another probabilistic model was studied in \cite{Durrett}, also at the microscopic level. In \cite{KMS}, the study is carried out at the macroscopic level (hydrodynamic limit) in both finite and infinite volume with reservoirs, in order to account for the migration/immigration mechanism. 

Here, we aim at studying the hydrodynamic and hydrostatic limits of the interacting particle system in \cite{KMS}, under the effect of slow reservoirs in any dimension $d\geq 1$. The slow-down mechanism models the fact that beyond the boundary through which insects arrive into the system or leave it, there are very few insects (the exterior of the system might be a territory which is much less favorable to the development of these insects). 

In the perspective of interacting particle systems, the sterile insect technique is modelled as follows: individuals evolve on a $d$-dimensional finite set $B_{N}=\{-N,\cdots,N\}\times \Tnd$, where $N\geq 1$ and $\Tnd = (\Z/N\Z)^{d-1}.$ The evolution of the population is described by a continuous time Markov process $(\e_{t}^N)_{t\geq 0}$ with state space $E^{B_{N}}$ where $E$ is a countable set. In this model, the gender does not come into account so we refer to sterile individuals rather than sterile males. The quantity of interest here is not the number of insects per site but the type of insects present at a given site. Precisely, $E=\{0,1,2,3\}$ and for $x$ in $B_{N}$,
\[
\eta(x) = \left\{
    \begin{array}{ll}
        0 & \text{if there are no insects in}~ x, \\
        1 & \text{if there are only wild insects in}~ x,\\
        2 & \text{if there are only sterile insects in}~ x,\\
        3 & \text{if there is a combination of wild and sterile insects in}~ x.
    \end{array}
\right.
\]
The dynamics of the Markov process is the superposition of three Markovian jump processes:
\begin{itemize}
    \item [(i)] An exchange dynamics which models the fact that insects move in an isotropic way within the bulk $B_{N}$ and which is parameterised by a diffusivity constant $D>0$. Precisely, for a configuration $\e$ and $x,y$ two sites in $B_N$, the states of sites $x$ and $y$ in $\e$ are exchanged at rate $D$.
    \item[(ii)] A birth and death dynamics within $B_N$ which models births of individuals due to the mating of a wild individual with wild or sterile insects, as well as deaths of individuals. This is the dynamics introduced in \cite{kuoch:hal-01100145} that was referred to as a contact process with random slowdowns (CPRS).  It is parameterised by a release rate $r>0$ and growth rates $\lambda_{1},\lambda_{2}>0$. Sterile individuals are injected on a site at rate $r$ independently of everything else. The rate at which wild individuals give birth (to wild individuals) on neighbouring sites is $\lambda_{1}$ at sites in state 1, and $\lambda_{2}$ at sites in state 3. Sterile individuals do not give birth. We take $\lambda_{2}<\lambda_{1}$ to reflect the fact that fertility is reduced at sites in state $3$. Deaths for each type of insects occur independently and at rate 1. 
    \item[(iii)] A boundary dynamics which models the slow migration/immigration mechanism. The mechanism is parameterised by a function $\widehat{b}=(b_{1},b_2,b_3):\{-1,1\}\times \Td \rightarrow [0,1]^3$ for the rates, where $\Td = (\R/\Z)^{d-1}$, and the slowdown effect by two constants $\theta_{\ell}$ and $\theta_r$ in $\R^+$. 
\end{itemize}
Note that without the presence of sterile insects, the CPRS would be a basic contact process (as defined for instance in \cite{liggett_interacting_2005}) with parameter $\lambda_{1}$, and the presence of sterile insects can be interpreted as a random decrease of the fertility rate due to the presence of sites containing sterile and wild individuals. In \cite{kuoch:hal-01100145}, the microscopic study of the contact process with random slowdowns in dimension $d\geq 1$ leads to the following phase transition result: for certain values of $\lambda_1$ and $\lambda_2$, when $r$ is large enough, the healthy population almost surely becomes extinct, and survives otherwise. In this paper, the CPRS will be called generalized contact process. In \cite{KMS}, the hydrodynamic limit of the superposition of the three dynamics above, where the first and the third one are accelerated in the diffusive scaling $N^2$, and where $\ttl=\ttr=0$, is proven to be a system of non linear reaction-diffusion equations with Dirichlet boundary conditions in any dimension.

In this paper, we first prove the finite volume hydrodynamic limit of this particle system for any values of $\theta_{\ell},\theta_r\geq 0$ and in any dimension. The hydrodynamic equation obtained has mixed boundary conditions which depend on the values of $\theta_{\ell}$, resp. $\theta_r$. Precisely, for $\ttl\in [0,1)$, resp. $\ttr\in [0,1)$, we get a Dirichlet type boundary condition at the left-hand side, resp. right-hand side of the system. For $\ttl=1$, resp. $\ttr=1$, we get a Robin type boundary condition at the left-hand side, resp. right-hand side of the system. For $\ttl>1$, resp. $\ttr>1$, we get a Neumann type boundary condition at the left-hand side, resp. right-hand side of the system.

We then prove the finite volume hydrostatic limit of the particle system for a specific class of parameters regarding the dynamics. Within that class of parameters, the sequence of invariant measures of the interacting particle system is associated to a profile which is the stationary solution of the hydrodynamic equation with corresponding mixed boundary conditions. 

Our paper is, up to our knowledge, the first one regarding the effect of mixed reservoirs in and out of equilibrium (hydrodynamic and hydrostatic limit) for a multi species process in a bounded $d$-dimensional cylinder. Note that all our results can be extended to the $d$-dimensional hypercube, $[-1,1]^d$, following the method in \cite{LMS}. We believe that the analysis for general domains would require more effort, in particular, regarding the choice of a suitable discretization of the underlying macroscopic space. The discretization issue has been addressed for some conservative interacting particle systems evolving on a bounded Lipschitz domain (see \cite{sau} and references therein). For domains such as manifolds, we refer to \cite{Manifold} for the symmetric simple exclusion process with no reservoirs. Both papers \cite{sau} and \cite{Manifold} rely on duality techniques. The effect of reservoirs on a one dimensional conservative system has been widely studied in finite volume (see for instance \cite{Derrida}, \cite{Lebo}). Much is now known both at the microscopic and macroscopic levels. Recently, the effect of slow reservoirs has aroused considerable interest for the symmetric simple exclusion process in one dimension (see for instance \cite{baldasso_exclusion_2017}, \cite{franco_hydrodynamical_2013}, \cite{TT}, \cite{Patricia} and references therein). In \cite{farfan_hydrostatics_2011}, the authors proved a hydrostatic principle for a boundary driven gradient symmetric exclusion process using the fact that the stationary profile is a global attractor for the hydrodynamic equation. This method inspired our proof for the hydrostatic limit. However, the fact that we obtained coupled equations for the hydrodynamic limit, and that we work in any dimension make the analysis more subtle and require general analytical tools.

This paper is organized as follows. In Section 2 we introduce the notation and state our results. The proof of the hydrodynamic limit for each of these regimes is established in Section 3 via the Entropy Method. Among other things, as we work in an arbitrary dimension, some care must be taken to perform the "replacement lemma", and also to  define and characterize the solution of the hydrodynamic limit at the boundary, through the use of the "Trace Operator" (see subsection 3.4). The difficulties due to different boundary slowing exponents are purely analytical, and appear when proving uniqueness of the hydrodynamic equation with mixed boundary conditions. The proof of the hydrostatic limit, established in Section 4, relies on the use of a well chosen change of coordinates for the coupled equations. Under this change of coordinates (inspired by some simulations, see Appendix \ref{Simuu}), a comparison principle holds. It allows us to find a unique attractor when some conditions on the parameters are satisfied. Outside that class of parameters, although uniqueness of the invariant measure holds, we do not even know whether there is uniqueness of the stationary solution of the hydrodynamic equation, and simulations show that for Neumann type boundary conditions there are several stationary profiles. However, we believe that a more general hydrostatic principle in the spirit of the one proved in \cite{article} is valid. 

\section{Notation and results}
\subsection{The microscopic model}
The dynamics of our interacting particle system is given by three generators, one for the exchange dynamics, one for the generalized contact dynamics and one for the boundary dynamics. In order to explicit each one of those generators, let us give a few notations. Let $N\in \N$ and $d\geq 1$. For $p\geq 1$, we write $\mathbb{T}_N^p$,  resp.  $\mathbb{T}^p$, the discrete, resp. continuous, torus $(\Z/N\Z)^p$, $(\R/\Z)^p$. Denote by $B_{N}=\{-N,\cdots,N\}\times \Tnd$ the bulk and $\Gamma_{N}=\{-N,N\}\times \Tnd$, resp. $\Gamma_{N}^+=\{N\}\times \Tnd$, resp. $\Gamma_{N}^-=\{-N\}\times \Tnd $, the boundary, resp. right-hand side boundary, resp. left-hand side boundary, of the bulk. Denote $B=(-1,1)\times \Td$ the continuous counter part of the bulk, $\overline{B}=[-1,1]\times \Td$ its closure, $\Gamma=\{-1,1\}\times \Td$, $\Gamma^{-}=\{-1\}\times \Td$ and $\Gamma^{+}=\{1\}\times \Td $. 

The microscopic state space is denoted by ${\Omega}_{N}:=\{0,1,2,3\}^{B_{N}}$ and its elements, called configurations, are denoted by $\e$. Therefore, for $x\in B_{N}$, $\e(x)\in \{0,1,2,3\}$. To describe the dynamics of our model, we will use the correspondence introduced in \cite{KMS} between the state space $\Omega_{N}$ and $\widehat{\Sigma}_{N}:=(\{0,1\}\times \{0,1\})^{B_{N}}$ where the correspondence between an element $\xiw \in \widehat{\Sigma}_{N}$ and $\e \in \Omega_{N}$ is given as follows: for $x\in B_{N}$,
\begin{equation}\label{corresp}
    \begin{split}
        &\e(x)=0 ~ \Longleftrightarrow~  (1-\xi(x))(1-\omega(x))=1,  \\
        & \e(x)=1 ~ \Longleftrightarrow~  \xi(x)(1-\omega(x))=1,\\
        & \e(x) = 2 ~ \Longleftrightarrow~  (1-\xi(x))\omega(x) = 1,\\
        &\e(x)=3 ~ \Longleftrightarrow~  \xi(x)\omega(x) = 1.
    \end{split}
\end{equation}
In other words, $\omega\in \{0,1\}$ represents the presence of sterile insects, and $\xi \in \{0,1\}$ that of wild ones on a given site, i.e., $(\xi(x),\omega(x))=(0,0)$ if $x$ is in state $0$, $(1,0)$ if it is in state $1$, $(0,1)$ if it is in state $2$ and $(1,1)$ if it is in state $3$. Also, in order to describe the evolution of the density of sites in state $1$, $2$, $3$, resp $0$, we define for $x$ in $B_N$ and a configuration $\e\in \Omega_N$ with associated configuration $\xiw \in \whsN$,
\begin{equation}\label{fonctiondexiw}
    \left\{
    \begin{array}{ll}
         \e_1(x):=\mathds{1}_{\{\e(x)=1\}}=\xi(x)(1-\omega(x)) ,\\ 
        \e_2(x):=\mathds{1}_{\{\e(x)=2\}}=(1-\xi(x))\omega(x),\\
        \e_3(x):=\mathds{1}_{\{\e(x)=3\}}=\xi(x)\omega(x),\\
        \e_0(x):=\mathds{1}_{\{\e(x)=0\}}=(1-\xi(x))(1-\omega(x)).
    \end{array}
\right.
\end{equation}
Finally, we also express the correspondence \eqref{corresp} by the following application from $\whsN$ to $\Omega_N$: 
\begin{equation}\label{applicorresp}
\e=\e\xiw,~ ~ \text{where, for any}~ x\in B_N,~ ~ \e(x) = 2\omega(x)+\xi(x).
\end{equation}
\begin{itemize}
    \item [•] \textbf{Generator for the exchange mechanism:} it corresponds to the usual stirring mechanism where each site has an exponential clock with rate $D$ and independent from all the other clocks, where $D$ is a  fixed positive parameter. When the clock rings, a neighbouring site is chosen uniformly at random and the states of both sites are exchanged. The action of the generator on functions $f:\whsN \rightarrow \R$ is therefore given by:
    \begin{equation}\label{Exchange}
    \mathcal{L}_{N}f\xiw :=\sum_{k=1}^d\sum_{(x,x+e_{k})\in B_{N}^2}D\Big(f(\xi^{x,x+e_{k}},\omega^{x,x+e_{k}})-f\xiw \Big)
\end{equation}
where $(e_1,\cdots,e_d)$ is the canonical basis of $\Z^d$ and for $\zeta \in \{0,1\}^{B_N}$ and $x,y\in B_N$, $\zeta^{x,y}$ is the configuration obtained from $\zeta$ by exchanging the occupation variables $\zeta(x)$ and $\zeta(y)$, i.e,
 \[  \zeta^{x,y}(z)=\left\{
    \begin{array}{ll}
         \zeta(x)~ ~ \text{if}~ ~ z=y,\\ 
        \zeta(y)~ ~ \text{if}~ ~ z=x,\\
        \zeta(z)~ ~ \text{otherwise}.
    \end{array}
\right.\]
    \item[•] \textbf{Generator for the generalized contact process in the bulk:} following the description of the generalized contact process in the introduction, let us give its rates in the bulk. For $\e\in \Omega_{N}$, $x\in B_{N}$ and $i \in \{1,3\}$ denote by $n_i(x,\e)$ the number of neighbours of $x$ in state $i$, that is, $n_i(x,\e) = \underset{y\sim x}{\sum}\e_i(y)$, where $x\sim y$ means that $x$ and $y$ are neighbouring sites in $B_N$. Births and arrivals of sterile individuals at $x$ happen at the following rates:
\begin{equation}\label{CPRS}
\begin{split}
    & 0 \rightarrow 1 ~ \text{at rate}~ \lambda_{1}n_{1}(x,\e) + \lambda_{2}n_{3}(x,\e),~~~ 2 \rightarrow 3~ \text{at rate}~\lambda_{1}n_{1}(x,\e) + \lambda_{2}n_{3}(x,\e),\\
    & \text{and}~~~~~ 0 \rightarrow 2 ~ \text{at rate}~ r,~~~ 1\rightarrow 3~ \text{at rate}~ r,
\end{split}
\end{equation}
Deaths of individuals at $x$ happen at rate $1$:
\begin{equation}\label{CPRS2}
\begin{split}  
& 1\rightarrow0 ~ \text{at rate}~ 1,~ ~ 2\rightarrow 0~ \text{at rate}~1,~ ~  3 \rightarrow 1 ~ \text{at rate}~1 ,~ ~  3\rightarrow 2~ \text{at rate}~1,
\end{split}
\end{equation}
 Therefore, using the correspondences \eqref{corresp}, \eqref{fonctiondexiw} and \eqref{applicorresp}, the generator $\mathbb{L}_{N}= \mathbb{L}_{N,\lambda_{1},\lambda_{2},r} $ of the generalized contact process acts as follows on functions $f:\whsN \rightarrow \R$: for $\xiw$ in $\whsN$ and $\e=\e\xiw$, we have
\begin{equation}\label{genCPRS}
    \mathbb{L}_{N}f\xiw = \sum_{x\in B_{N}}\mathbb{L}_{B_{N}}^xf\xiw \, ,
\end{equation}
where for $x\in B_{N}$,
\begin{equation}\label{sousgenCPRS}
\begin{split}
    \mathbb{L}_{B_{N}}^xf\xiw &:= \Big(r(1-\omega(x))+\omega(x) \Big)\Big[f(\xi,\sigma^x\omega)-f\xiw \Big]\\
    &+ \Big(\beta_{B_{N}}(x,\xi,\omega)(1-\xi(x))+\xi(x) \Big)\Big[f(\sigma^x\xi,\omega)-f\xiw \Big],
\end{split}
\end{equation}
\begin{equation*}
    \beta_{B_{N}}(x,\e) := \lambda_{1}n_1(x,\e) + \lambda_{2}n_3(x,\e) 
\end{equation*}
where, for $\zeta\in \{0,1\}^{B_N}$, $\sigma^x\zeta$ is the configuration obtained from $\zeta$ by flipping the configuration at $x$, i.e.
 \[  \sigma^x\zeta(z)=\left\{
    \begin{array}{ll}
         1-\zeta(x)~ ~ \text{if}~ ~ z=x,\\ 
        \zeta(z)~ ~ \text{otherwise}.
    \end{array}
\right.\]

\item[•]\textbf{Generator for the boundary dynamics:} the generator of the dynamics at the boundary  is parametrised by $\wht=(\ttl,\ttr)$ with $\ttl,\ttr \geq 0$ and a positive function $\widehat{b}=(b_{1},b_{2},b_{3}):\Gamma \rightarrow \R_{+}^3$ satisfying the following conditions: there exists a neighbourhood $V$ of $\overline{B}$ in $\R\times \mathbb{T}^{d-1}$ and a smooth function $\widehat{g}=(g_{1},g_{2},g_{3}):V\rightarrow (0,1)^3$ in $\mathcal{C}^2(V,\R)$ (the space of twice differentiable functions), with
\begin{equation}\label{Condition1}
    \exists~ c^*,C^*>0,~ ~ 0<c^*<\underset{1\leq i \leq 3}{\min}~ |g_{i}|\leq \underset{1\leq i \leq 3}{\max}~ |g_{i}|\leq C^*<1,~~\text{and}~~g_1+g_2+g_3 < 1 
\end{equation}
and, the restriction of $\widehat{g}$ to $\Gamma$ is equal to $\widehat{b}$. The dynamics at the boundary can then be described as follows: a site $x\in \Gamma_N^-$, resp. $x\in \Gamma_N^+$,  flips from state $i\in \{0,1,2,3\}$ to state $j\in \{0,1,2,3\}\setminus \{i\}$ at rate $N^{-\ttl}b_j(x/N)$, resp. $N^{-\ttr}b_j(x/N)$ . In order to express the generator of the boundary dynamics, we make use of $\e_i=\e_i\xiw$ for $i\in \{0,1,2,3\}$ which is the configuration in $\{0,1\}^{B_N}$ obtained from $\xiw\in \whsN$ according to \eqref{fonctiondexiw}.
For $f:\whsN \rightarrow \R$, the boundary generator acts on $f$ as follows: for $\xiw$ in $\whsN$, we have
\begin{equation}\label{genbound}
\begin{split}
       L_{\widehat{b},\wht,N}f\xiw 
    &=N^{-\ttl}\sum_{i=0}^3\sum_{x\in \Gamma_{N}^-}b_i(x/N)\Big(f(\sigma_{i,x}\xiw)-f\xiw \Big)\\
    &+ N^{-\ttr}\sum_{i=0}^3\sum_{x\in \Gamma_{N}^+}b_i(x/N)\Big(f(\sigma_{i,x}\xiw)-f\xiw \Big),
\end{split}
\end{equation}
where $b_{0}(x/N) := 1- \sum_{i=1}^{3}b_{i}(x/N)$ and with $\sigma_{i,x}\xiw := \sigma_{i,x}\e\xiw$, the configuration in $\whsN$ associated to $\sigma_{i,x}\e$, where
\begin{equation*}
    \sigma_{i,x}\e(z):=\left\{
    \begin{array}{ll}
         i~ ~ \text{if}~ ~ z=x,\\ 
        \e\xiw(z)~ ~ \text{otherwise}
    \end{array}
\right.
\end{equation*}
with $\e\xiw$ as defined in \eqref{applicorresp}.
\end{itemize}
Fix a time horizon $T>0$ and denote by $\{(\xi_t^N, \omega_t^N),~ t\in[0,T]\}$ the Markov process associated to the generator
\begin{equation}\label{Generator}
    L_{N} :=N^2 \mathcal{L}_{N}+ N^2L_{\widehat{b},\wht,N} + \mathbb{L}_{N}. 
\end{equation}
Let $D_{\whsN}([0,T])$ be the path space of càdlàg trajectories with values in $\whsN$. Given a measure $\mu_{N}$ on $\whsN$, denote by $\mathbb{P}_{\mu_{N}}$ the probability measure on $D_{\whsN}([0,T])$ induced by $\mu_{N}$ and $(\xi_{t}\allowbreak,\omega_t)_{t\in [0,T]}$,
and denote by $\mathbb{E}_{\mu_{N}}$ the expectation with respect to $\mathbb{P}_{\mu_{N}}$.\\\\
\textbf{\textit{Invariant measures for the exchange and boundary dynamics:}}\\
Consider $\widehat{\alpha}=(\alpha_{1},\alpha_{2},\alpha_{3}):\overline{B}\longrightarrow (0,1)^3$ a smooth function satisfying, for $c^*,C^*>0$ given in \eqref{Condition1},
\begin{equation}\label{a)}
    0<c^*<\min_{1\leq i\leq 3}|\alpha_{i}|\leq  \max_{1\leq i\leq 3}|\alpha_{i}|\leq C^*<1,~~\alpha_1 + \alpha_2 + \alpha_3 < 1.
\end{equation}
Denote by $\nu_{\widehat{\alpha}}^N$ the Bernoulli compound product measure on $B_{N}$ with parameter $\widehat{\alpha}$: for $\xiw\in \whsN$,
\begin{equation}
    \nualph\xiw:= \frac{1}{Z_{\widehat{\alpha},N}}\exp\Big(\sum_{i=1}^3\sum_{x\in B_{N}}\Big(\log \frac{\alpha_{i}(x/N)}{\alpha_{0}(x/N)} \Big)\e_{i}(x) \Big), 
\end{equation}
where $\alpha_{0}=1-\alpha_{1}-\alpha_{2}-\alpha_{3}$ and
where $Z_{\widehat{\alpha},N}$ is the normalizing constant
\begin{equation*}
    Z_{\widehat{\alpha},N}= \prod_{x\in B_N}\frac{1}{\alpha_0(x/N)}.
\end{equation*}
Note that $\nualph$ is such that for every $1\leq i \leq 3$ and $x\in B_{N}$,
\begin{equation*}
    \mathbb{E}_{\nualph}[\e_{i}(x)] = \alpha_{i}(x/N).
\end{equation*}
One can verify the following statements:
    \begin{itemize}
        \item Consider $\widehat{\alpha}$ a smooth profile satisfying \eqref{a)} and
        
\begin{equation}\label{b)}
    \forall x\in \Gamma,~ ~ \widehat{\alpha}(x)=\widehat{b}(x).
\end{equation}
    Then, $\nualph$ is an invariant and reversible measure for the boundary dynamics: for any $f:\whsN\rightarrow \R$,
\begin{equation}\label{Invbord}
    \int_{\whsN}L_{\widehat{b},\wht,N}f\xiw d\nualph\xiw=0.
\end{equation}
        \item Consider $\widehat{\alpha}$ a constant profile. Then $\nualph$ is an invariant and reversible measure for the exchange dynamics so for any $f:\whsN\rightarrow \R$,
        \begin{equation}\label{Invech}
    \int_{\whsN}\mathcal{L}_{N}f\xiw d\nualph\xiw=0.
    \end{equation}
\end{itemize}
For any $\widehat{\theta}\in (\R^+)^2$, at fixed $N$, the dynamics defined by \eqref{Generator} is irreducible and the state space is finite. Therefore, the dynamics has a unique invariant probability measure that in the sequel we denote by $\mu_{N}^{ss}\allowbreak (\widehat{\theta}\allowbreak )$.\\\\
\textbf{\textit{Useful (in)equalities:}} For any $A,B>0$,
\begin{equation}\label{useful1}
    A(B-A) = -\frac{1}{2}(B-A)^2+ \frac{1}{2}(B^2-A^2).
\end{equation}
For any $a,b,A$ and $N\in \N$,
\begin{equation}\label{useful2}
    2ab \leq \frac{N}{A}a^2 + \frac{A}{N}b^2.
\end{equation}
For any sequences of positive numbers $(a_N)_{N\geq 1}$ and $(b_N)_{N\geq 1}$
\begin{equation}\label{Conv}
    \underset{N \rightarrow \infty }{\overline{\lim}}~ \frac{1}{N}\log(a_{N}+b_{N}) \leq \max\Big(\underset{N \rightarrow \infty }{\overline{\lim}}~ \frac{1}{N}\log  a_{N},\underset{N \rightarrow \infty }{\overline{\lim}}~ \frac{1}{N}\log b_{N} \Big).
\end{equation}

\subsection{The macroscopic equations}
Let us first introduce a few notations. We will write functions with values in $\R$ with Roman letters (for instance $G$) and the ones with values in $\R^3$ with letters with a hat (for instance $\whg$) . For $n,m\in \mathbb{N}$, denote by $\mathcal{C}^{n,m}([0,T]\times \overline{B})$ the space of functions that are $n$ times differentiable in time and $m$ times differentiable in space, $\mathcal{C}_{0}^{n,m}$, resp. $\mathcal{C}_{0,-}^{n,m} $, resp. $\mathcal{C}_{0,+}^{n,m}$, the ones in $\mathcal{C}^{n,m}([0,T]\times \overline{B})$ which are zero on $\Gamma$, resp. $\Gamma^-$, resp. $\Gamma^+$. Denote by $\mathcal{C}_{k}^{\infty}(B)$ the space of smooth functions with compact support in $B$, $\mathcal{C}^{m}(\overline{B}) $ the space of functions that are $m$ times differentiable in space with $\mathcal{C}_0^{m}(\overline{B})$, resp. $\mathcal{C}_{0,-}^{m}(\overline{B}) $, resp. $\mathcal{C}_{0,+}^{m}(\overline{B}) $, those which are zero on $\Gamma$, resp. $\Gamma^-$, resp. $\Gamma^+$, and $\mathcal{C}(\overline{B})$ the space of continuous functions on $\overline{B}$ . For $\wht=(\ttl,\ttr)$ in $(\R^+)^2$, we will use the following notation to denote these functional spaces:
\begin{equation}
    \mathcal{C}_{\wht}:=\left\{
    \begin{array}{ll}
         \mathcal{C}_{0}^{1,2}~ ~ \text{if}~ ~ \wht \in [0,1)^2,\\ 
        \mathcal{C}_{0,-}^{1,2} ~ ~ \text{if}~ ~ \ttl\in [0,1), \ttr \geq 1,\\
        \mathcal{C}_{0,+}^{1,2} ~ ~ \text{if}~ ~ \ttr\in [0,1), \ttl \geq 1,\\
        \mathcal{C}^{1,2} ~ ~ \text{if}~ ~ \ttr, \ttl \geq 1.
    \end{array}
\right.
\end{equation}
Let $\langle~ , ~ \rangle$ be the $L^2(\overline{B})$ inner product and $\langle~ , ~ \rangle_{\mu}$ the inner product with respect to a measure $\mu$. For $\widehat{f}=(f_1,f_2,f_3)$ and  $\widehat{g}=(g_1,g_2,g_3)$ in $\big(L^2(\overline{B})\big)^3$, $\langle\widehat{f},\widehat{g}\rangle=\sum_{i=1}^3\langle f_i,g_i\rangle $. Recall that $(e_1,\cdots, e_d)$ is the canonical basis of $\Z^d$. Introduce the Sobolev space $\mathcal{H}^1(B)$ which we recall to be the set of functions $g\in L^2(\overline{B})$ such that for any $1\le k\le d$, there exists an element denoted by $\partial_{e_k} g\in L^2(B)$ such that for any $\varphi$ in $\mathcal{C}_{k}^{\infty}(B)$,
\begin{equation*}
    \langle\partial_{e_{k}}\varphi,g\rangle = -\langle\varphi, \partial_{e_k} g\rangle,
\end{equation*}
where $\partial_{e_{k}}\varphi$ is the usual partial derivative. The $\mathcal{H}^1(B)$ norm is then defined as follows:
\begin{equation*}
    \|g\|_{\mathcal{H}^1(B)} = \Big(\|g\|_{L^2(B)}^2+ \sum_{k=1}^{d}\|\partial_{e_k} g\|_{L^2(B)}^2 \Big)^{1/2} = \Big(\|g\|_{2}^2+ \sum_{k=1}^{d}\|\partial_{e_k} g\|_{2}^2 \Big)^{1/2} .
\end{equation*}

We will write $\|g\|_{2}^2$ instead of $\|g\|_{L^2(B)}^2$ when no confusion is possible. Introduce $\mathcal{H}_0^1(B)$, the closure of $\mathcal{C}_{k}^{\infty}(B)$ in $\mathcal{H}^1(B)$ for that norm. Denote by $L^2\Big([0,T],\allowbreak\mathcal{H}^1(B)\Big)$ the space of functions $f:[0,T]\rightarrow\mathcal{H}^1(B)$ such that 
\[\int_{0}^T \|f(t,.)\|_{\mathcal{H}^1(B)}^2 dt < \infty.\]

In order to define the value of an element $G$ in $\mathcal{H}^1(B)$ at the boundary, we need to introduce the notion of trace of functions on such Sobolev spaces.
The trace operator in the Sobolev space \allowbreak $\mathcal{H}^1(B)$ can be defined as a bounded linear operator, \allowbreak $\tr : \mathcal{H}^1(B)\to L^2(\Gamma)$ such that $\tr$  extends the classical trace, that is
$\tr(G)=G_{|_\Gamma}$, for any $G\in \mathcal{H}^1(B)\cap {\mathcal C}(\overline{B})$.
We refer to \cite[Part II Section 5]{evans_partial_2010} for a detailed survey on the trace operator.

In the sequel, for $s,u\in [0,T] \times \Gamma$ and for any $f \in L^2\Big([0,T],\mathcal{H}^1(B)\Big)$, $f(s,u)$ stands for $\tr(f(s,.))(u)$. Notice that $\mathcal{H}_0^1(B)$ is the set of elements of $\mathcal{H}^1(B)$ with zero trace.

To lighten notations, for a function $\whg$ depending on time and space we will often write $\whg_s$ instead of $\whg(s,.)$. Finally, for $\wht\in (\R^+)^2$, introduce the following linear functional on $L^2\Big([0,T],\mathcal{H}^1(B)\Big)$ parameterised by a test function $\whg$ in $\mathcal{C}_{\wht}$ : for $t\in [0,T]$,
\begin{equation}\label{fonctionnelle}
\begin{split}
        I_{\whg}(\whr)(t)  &:= ~ \langle\whr_t,\whg_t\rangle -\langle\whr_0,\whg_0\rangle  -\int_{0}^t\langle\whr_s,\partial_{s}\whg_s\rangle ds-D\int_{0}^t\langle\whr_s,\Delta \whg_s\rangle ds
        -\int_{0}^t\langle\widehat{F}(\whr_{s}),\whg_s\rangle ds
\end{split}
\end{equation}
where $\widehat{F}=(F_1(\whr),F_2(\whr),F_3(\whr)):[0,1]^3 \rightarrow \R^3$ is defined by  
\begin{equation} \label{F}
\left\{
    \begin{array}{ll}
          F_1(\rho_1,\rho_2,\rho_3) = 2d(\lambda_1\rho_1+\lambda_2 \rho_3)\rho_0 + \rho_3-(r+1)\rho_1  \\
           F_2(\rho_1,\rho_2,\rho_3)=r\rho_0+\rho_3-2d(\lambda_1\rho_1+\lambda_2\rho_3)\rho_2-\rho_2\\
          F_3(\rho_1,\rho_2,\rho_3) =2d(\lambda_1\rho_1+\lambda_2 \rho_3)\rho_2+r\rho_1-2\rho_3,
           \end{array}
        \right.
\end{equation}
with $\rho_0=1-\rho_1-\rho_2-\rho_3$.

The hydrodynamic equation is a system of coupled reaction diffusion equations with mixed boundary conditions depending on $\wht$. If $\ttl$, resp. $\ttr$, belongs to $[0,1)$, the boundary conditions are of Dirichlet type on $\Gamma^-$, resp. $\Gamma^+$. If $\ttl=1$, resp $\ttr=1$, they are of Robin type on $\Gamma^-$, resp. $\Gamma^+$. If $\ttl>1$, resp. $\ttr >1$, they are of Neumann type on $\Gamma^-$, resp. $\Gamma^+$. We will focus on the cases where $\ttl\in [0,1),\ttr=1$ resp. $\ttl>1,\ttr=1$, corresponding to a Dirichlet boundary condition on $\Gamma^-$ and a Robin boundary condition on $\Gamma^+$, resp. a Neumann boundary condition on $\Gamma^-$ and a Robin boundary condition on $\Gamma^+$. All the other cases can be adapted from those ones (see Table \ref{Tab}).
\begin{Def} \label{Weakdyn}
Let $\widehat{\gamma}:B \rightarrow \R^3$ be a continuous function. 
\begin{itemize}
    \item [•] Hydrodynamic equation for $\ttl \in [0,1)$ and $\ttr=1$. \\
    A bounded function $\whr=(\rho_{1},\rho_{2},\rho_{3}):[0,T]\times B \rightarrow \R^3$ is a weak solution of the Dirichlet $+$ Robin mixed boundary problem
\begin{equation} \label{D+R}
    \left\{
    \begin{array}{ll}
        \partial_{t}\whr = D\Delta \whr + \widehat{F}(\whr)~ \text{in}~ B\times (0,T),\\ 
        \whr(0,.)=\widehat{\gamma}~ \text{in}~ B,\\
        \whr(t,.)_{|\Gamma^-}=\widehat{b}~ \text{for}~ 0<t\leq T,\\
        \partial_{e_{1}}\whr(t,.)_{|\Gamma^{+}}=\frac{1}{D}(\widehat{b}-\whr)_{|\Gamma^{+}}~ \text{for}~ 0<t\leq T,
    \end{array}
\right.
\end{equation}
if, for any $1\leq i\leq 3$, 
\begin{equation}\label{conditiona}
    \rho_{i}\in L^2\Big([0,T],\mathcal{H}^1(B)\Big),
\end{equation} 
for any function $\whg \in \mathcal{C}_{\wht}$, for any $t\in [0,T]$,
\begin{equation}
    \begin{split}
    &I_{\whg}(\whr)(t)+ D\sum_{i=1}^3\int_{0}^t\int_{\Gamma^-}b_{i}(r)(\partial_{e_{1}}G_{i,s})(r)n_{1}(r).dS(r)ds\\
    &+  D\sum_{i=1}^3\int_{0}^t \int_{\Gamma^+}\rho_{i}(s,r)(\partial_{e_{1}}G_{i,s})(r)n_{1}(r).dS(r)ds\\
    &-\sum_{i=1}^3\int_{0}^t\int_{\Gamma^+}G_{i}(r)(b_{i}(r)-\rho_{i}(s,r))n_{1}(r).dS(r)ds=0, \label{weakD+R}
    \end{split}
\end{equation}
where $n_{1}(r)$ is the outward normal unit vector to the boundary surface $\Gamma$ and $dS(r)$ is an element of surface on $\Gamma$. And,
\begin{equation} \label{CI}
    \whr(0,.)=\widehat{\gamma}(.) ~ ~ \text{almost surely.}
\end{equation}
    \item [•] Hydrodynamic equation for $\ttl>1$ and $\ttr =1$.\\
    A bounded function $\whr=(\rho_{1},\rho_{2},\rho_{3}):[0,T]\times B \rightarrow \R^3$ is a weak solution of the Neumann $+$ Robin mixed boundary problem
\begin{equation} \label{N+R}
    \left\{
    \begin{array}{ll}
        \partial_{t}\whr = D\Delta \whr + \widehat{F}(\whr)~ \text{in}~ B\times (0,T),\\ 
        \whr(0,.)=\widehat{\gamma}~ \text{in}~ B,\\
        \partial_{e_{1}}\whr(t,.)_{|\Gamma^-}=0~ \text{for}~ 0<t\leq T\\
        \partial_{e_{1}}\whr(t,.)_{|\Gamma^{+}}=\frac{1}{D}(\widehat{b}-\whr)_{|\Gamma^{+}}~ \text{for}~ 0<t\leq T,
    \end{array}
\right.
\end{equation}
if $\whr$ satisfies conditions \eqref{conditiona} and \eqref{CI} as well as the following condition \eqref{WeakN+R}: for any $\whg\in \mathcal{C}_{\wht}$, for any $t\in [0,T]$,
\begin{equation}\label{WeakN+R}
    \begin{split}
        &I_{\whg}(\whr)(t) + D\sum_{i=1}^3\int_{0}^t \int_{\Gamma^-}\rho_{i}(s,r)(\partial_{e_{1}}G_{i,s})(r)n_{1}(r).dS(r)ds\\
        &+D\sum_{i=1}^3\int_{0}^t \int_{\Gamma^+}\rho_{i}(s,r)(\partial_{e_{1}}G_{i,s})(r)n_{1}(r).dS(r)ds\\
        &-\int_{\Gamma^+}G_{i}(r)(b_{i}(r)-\rho_{i}(s,r))n_{1}(r).dS(r)ds=0.
    \end{split}
\end{equation}

\end{itemize}
\end{Def}
\begin{rem}
    In \eqref{weakD+R}, the integral over $\Gamma^-$ corresponds to the Dirichlet boundary condition. In \eqref{WeakN+R} the integral over $\Gamma^-$ comes from an integration by parts of the terms involved in the bulk. Both in \eqref{weakD+R} and \eqref{WeakN+R} the first integral over $\Gamma^+$ comes from an integration by parts of the terms involved in the bulk and the second integral over $\Gamma^+$ corresponds to the Robin boundary condition.
\end{rem}

\begin{figure}
    \centering
    \renewcommand{\arraystretch}
{2}
    \begin{tabular}{|R{1.5cm}|C{1.5cm}|L{1.5cm}|L{1.5cm}|}
  \hline
   \begin{center}
       $(\ttl,\ttr)$ 
   \end{center}   & \begin{center}
       $\ttr\in [0,1)$
   \end{center} & \begin{center}
       $\ttr=1$
   \end{center} & \begin{center}
       $\ttr>1$
   \end{center} \\
    \hline
    \begin{center}
        $\ttl\in [0,1)$
    \end{center} & \begin{center} (D ; D)\end{center} & \begin{center}
        (D ; R)
    \end{center} & \begin{center}
        (D ; $N_e$)
    \end{center} \\
   \hline
   \begin{center}
      $\ttl=1$  
   \end{center}
   &\begin{center}
       (R ; D)
   \end{center} & \begin{center}
       (R ; R)
   \end{center}& \begin{center}
       (R ; $N_e$)
   \end{center}\\
    \hline
   \begin{center}
       $\ttl>1$
   \end{center} & \begin{center}
       ($N_e$ ; D)
   \end{center} &\begin{center}
       ($N_e$ ; R)
   \end{center}& \begin{center}
       ($N_e$ ; $N_e$)
   \end{center}\\
    \hline
\end{tabular}
    \caption{Mixed boundary conditions depending on the values of $\ttl$ and $\ttr$. The letters D, resp. R, resp. $N_e$ denote a Dirichlet, resp. Robin, resp. Neumann boundary condition. For instance (D ; $N_e$) denotes a left-hand side Dirichlet boundary condition and a right-hand side Neumann boundary condition. }
    \label{Tab}
\end{figure}
\newpage
\begin{Def}[Stationary solution of the hydrodynamic equation]\

\begin{itemize}
    \item [•] A function $\overline{\rho} = (\overline{\rho}_{1},\overline{\rho}_{2},\overline{\rho}_{3}  )$ in $\big(\mathcal{H}^1(B)\big)^3$ is a stationary solution of \eqref{D+R} if for every function $\whg=(G_{1},G_{2},G_{3}) \in \mathcal{C}_{0,-}^2(\overline{B})^3$, for all $1\leq i \leq 3$,
\begin{equation}\label{stat D+R}
\begin{split}
    &D\langle\rho_i,\Delta G_i\rangle + \langle F_i(\whr),G_i\rangle
    =D\int_{\Gamma^-}b_{i}(r)(\partial_{e_{1}}G_{i})(r)n_{1}(r).dS(r)\\
    &+ D\int_{\Gamma^+}\rho_{i}(r)(\partial_{e_{1}}G_{i})(r)n_{1}(r).dS(r)-\int_{\Gamma^+}G_{i}(r)(b_{i}(r)-\rho_{i}(r))n_{1}(r).dS(r).
\end{split}
\end{equation}
In other words, $\overline{\rho}$ is a stationary solution of \eqref{D+R} if $\widehat{\rho}(t,u)\equiv \overline{\rho}(u)$ is a solution of \eqref{D+R}.
\item[•] A function $\overline{\rho} = (\overline{\rho}_{1},\overline{\rho}_{2},\overline{\rho}_{3}  )$ in $\big(\mathcal{H}^1(B)\big)^3$ is a stationary solution of \eqref{N+R} if for every function $\whg=(G_{1},G_{2},G_{3}) \in \mathcal{C}^2(\overline{B})^3$, for all $1\leq i \leq 3$,
    \begin{equation}\label{stat N+R}
\begin{split}
    &D\langle\rho_i,\Delta G_i\rangle + \langle F_i(\whr),G_i\rangle= D\int_{\Gamma^-}\rho_{i}(r)(\partial_{e_{1}}G_{i})(r)n_{1}(r).dS(r)\\
    &+D\int_{\Gamma^+}\rho_{i}(r)(\partial_{e_{1}}G_{i})(r)n_{1}(r).dS(r) -\int_{\Gamma^+}G_{i}(r)(b_{i}(r)-\rho_{i}(r))n_{1}(r).dS(r).
\end{split}
\end{equation}
In other words, $\overline{\rho}$ is a stationary solution of \eqref{D+R} if $\widehat{\rho}(t,u)\equiv \overline{\rho}(u)$ is a solution of \eqref{N+R}.

\end{itemize}

\end{Def}

\subsection{Hydrodynamic and hydrostatic results}
Let us state the main results proved in this paper. The first one (Theorem \ref{HL}) establishes the hydrodynamic limit of the dynamics defined above and the second one (Theorem \ref{T hydrostat}) establishes its hydrostatic limit.
Before stating Theorem \ref{HL}, let us first define the empirical measure  $\widehat{\pi}^N(\xi,\omega)=\widehat{\pi}^N$ associated to a given configuration $(\xi,\omega)$. Recall how in \eqref{fonctiondexiw}, we built $\e_i\in \{0,1\}^{B_N} $ from $\xiw \in \whsN$ for $0 \leq i \leq 3$. Then,
\begin{equation*}
\begin{split}
        \widehat{\pi}^N(\xi,\omega)&:=\Big(\frac{1}{N^d}\sum_{x\in B_N}\e_{1}(x)\delta_{x/N},\frac{1}{N^d}\sum_{x\in B_N}\e_{2}(x)\delta_{x/N},\frac{1}{N^d}\sum_{x\in B_N}\e_{3}(x)\delta_{x/N} \Big)\\
        &=: (\pi_{1}^N(\xi,\omega),\pi_{2}^N(\xi,\omega),\pi_{3}^N (\xi,\omega))
\end{split}
\end{equation*}
where $\delta_{x/N}$ is the point mass at $x/N$. For $\widehat{G}$ in  $\mathcal{C}^{1,2}([0,T]\times B)$ and $t\geq 0$, write

\begin{equation*}
    \langle\widehat{\pi}^N,\widehat{G}_t\rangle:=\sum_{i=1}^3\langle\pi_{i}^N,G_{i}(t,.)\rangle=\sum_{i=1}^3\frac{1}{N^d}\sum_{x\in B_{N}}\e_{i}(x)G_{i}(t,\frac{x}{N}).
\end{equation*}
The empirical measure is therefore the triplet of empirical measures associated to the density of sites in state $1$, resp. $2$, resp. $3$. Denote by $\mathcal{M}$ the set of positive measures on $B$ with total mass bounded by 3 (because for any configuration $\e$, $\pi^N(\e)(B) = \frac{|\Lambda_N|}{N^d}\leq 3$). The process 
$(\widehat{\pi}_{t}^N)_{t\geq 0} = (\widehat{\pi}^N(\xi_t,\omega_t))_{t\geq 0} $, is a Markov process with state space ${\mathcal{M}^3}$ and its trajectories belong to $D([0,T],{\mathcal{M}^3})$, the path space of càdlàg time trajectories with values in $\mathcal{M}^3$. We endow the path space with the Skorohod topology (we refer to \cite{billingsley_convergence_1999} for a detailed presentation on the Skorohod topology). For $\wht \in (\R^+)^2$ and $\mu_{N}$ a probability measure on $\whsN$, denote by $Q_{N}^{\wht}= \mathbb{P}_{\mu_{N}}(\widehat{\pi}^N)^{-1}$ the law of the process $(\widehat{\pi}^N\xit)_{t\geq 0}$ when $(\xi_0,\omega_0)\sim \mu_{N}$ and where $(\xi_t,\omega_t)_{t\geq 0}$ evolves according to the dynamics given by \eqref{Generator}, with parameter $\wht$ for the boundary reservoirs. The hydrodynamic result states as follows:

\begin{thm}[Hydrodynamic limit] \label{HL} For any sequence of initial probability measure $(\mu_{N})_{N\geq 1}$ on $\whsN$, the sequence of probability measures $(Q_{N}^{\wht})_{N\geq 1}$ is weakly relatively compact 
and  all its converging subsequences converge to   some limit  $Q^{\wht,*}$ concentrated on the set of weak solutions of the
hydrodynamic equation that are in $L^2([0,T];{\mathcal H}^1(B)) $, in the sense of Definition \ref{Weakdyn}. Furthermore, if there is an initial continuous profile $\widehat{\gamma}:B \rightarrow [0,1]^3$ such that for any $\delta>0$ and any $\whg \in \mathcal{C}_{k}^{\infty}(B)$,

\begin{equation*}
    \underset{N \rightarrow \infty}{\limsup}~ \mu_{N}\Big[\Big|\langle\widehat{\pi}^{N},\widehat{G}\rangle-\langle\widehat{\gamma},\widehat{G}\rangle \Big|>\delta \Big] = 0, 
\end{equation*}
then, $(Q_{N}^{\wht})_{N\geq 1}$ converges to the Dirac mass $Q^{\wht}$ concentrated on the unique weak solution $\whr$ of the boundary value problem associated to $\wht$ and with initial condition $\widehat{\gamma}$.
Therefore, for any $t\in [0,T]$, $\delta>0$ and any function $\whg \in \mathcal{C}_{c}^{1,2}([0,T]\times \overline{B})$,
\begin{equation*}
    \underset{N \rightarrow \infty}{\limsup}~\mathbb{P}_{\mu_{N}}\Big[\Big|\langle\widehat{\pi}_{t}^N,\whg_t \rangle - \langle\whr_t,\whg_t\rangle \Big| > \delta\Big] = 0.
\end{equation*}
\end{thm}
We prove Theorem \ref{HL} in Section 3.\medskip

Intuitively, the sterile insect technique is more effective when $r$ is large and $\lambda_2$ is small. This result has been made precise at the microscopic level for the generalized contact process, in \cite[Theorem 2.5]{kuoch:hal-01100145}, where the author proved a phase transition result: when $\lambda_1>\lambda_2$ are properly tuned, there is a critical value $r_c>0$ below which the wild population survives with strictly positive probability and above which the wild population dies out almost surely.

To establish the hydrostatic limit, depending on the parameters at the boundary, we will need one of the following sets of conditions to be satisfied:
\begin{equation*}
  (H_1): \; \left\{
    \begin{array}{ll}
        D\geq 1\\ 
       r+1>\ddd\\
       1> 2d\lambda_2
    \end{array}
\right.
\end{equation*}
\begin{equation*}
   (H_2): \; \left\{
    \begin{array}{ll}
      D\delta_1+ r+2>\ddd\\
       D\delta_1+1> 2d\lambda_2
    \end{array}
\right.
\end{equation*}
\begin{equation*}
  (H_3): \; \left\{
    \begin{array}{ll}
     r+2>\ddd\\
      1> 2d\lambda_2,
    \end{array}
\right.
\end{equation*}
where $\delta_1$ is the smallest eigenvalue of the Laplacian with Dirichlet boundary conditions (see \eqref{eigenprobDirich}). These conditions will appear technically in the proof of Theorem \ref{T hydrostat}. However, note that at fixed $D$ and $\lambda_1$, this corresponds to having $r$ large and $\lambda_2$ small enough, which is consistent with the effectiveness of the sterile insect technique.\medskip

Recall that $\mu_{N}^{ss}(\widehat{\theta})$ denotes the sequence of unique invariant measures for the irreducible dynamics defined by \eqref{Generator}. The hydrostatic result states as follows.
\begin{thm} [Hydrostatic limit] \label{T hydrostat} Suppose that conditions $(H_1)$ hold. There exists a unique stationary solution of \eqref{D+R} that we denote by $\overline{\rho}^{D,R}$, and a unique stationary solution of \eqref{N+R} that we denote by $\overline{\rho}^{N_e,R}$. Furthermore, the following statements hold.
\begin{itemize}
    \item [•]Consider $\widehat{\theta}=(\theta_{\ell},\theta_r)$ with $\theta_{\ell}\in [0,1)$ and $\theta_r=1$. For any continuous function $\whg:B \rightarrow [0,1]^3$,
\begin{equation}\label{Hydrostat D+R}
    \underset{N \rightarrow \infty}{\lim}~ \mathbb{E}_{\mu_{N}^{ss}(\widehat{\theta})}\Big(~ \Big|\sum_{i=1}^3\frac{1}{N^d}\sum_{x\in B_{N}}\e_{i}(x)G_{i}(x/N) - \sum_{i=1}^3\int_{\overline{B}}G_{i}(u)\overline{\rho}^{D,R}_{i}(u)du \Big|~   \Big) = 0. 
\end{equation}
In other words, the sequence $(\mu_{N}^{ss}(\widehat{\theta}))_{N\geq 1}$ is associated to the unique stationary profile $\overline{\rho}^{D,R}$.
\item[•]Consider $\widehat{\theta}=(\theta_{\ell},\theta_r)$ with $\theta_{\ell}>1$ and $\theta_r=1$. For any continuous function $\whg:B \rightarrow [0,1]^3$,
\begin{equation}\label{Hydrostat}
    \underset{N \rightarrow \infty}{\lim}~ \mathbb{E}_{\mu_{N}^{ss}(\widehat{\theta})}\Big(~ \Big|\sum_{i=1}^3\frac{1}{N^d}\sum_{x\in B_{N}}\e_{i}(x)G_{i}(x/N) - \sum_{i=1}^3\int_{\overline{B}}G_{i}(u)\overline{\rho}^{N,R}_{i}(u)du \Big|~   \Big) = 0. 
\end{equation}
In other words, the sequence $(\mu_{N}^{ss}(\widehat{\theta}))_{N\geq 1}$ is associated to the unique stationary profile $\overline{\rho}^{N_e,R}$.
\end{itemize}

\end{thm}

\begin{rem}
For all the other mixed boundary regimes corresponding to other values of $\widehat{\theta}$, the hydrostatic principle states in the same way, replacing $\overline{\rho}_i^{D,R}$ or $\overline{\rho}_i^{N_e,R}$ by the stationary solution of the associated hydrodynamic equation. In the cases where only Dirichlet and Robin boundary conditions are involved, we can slightly weaken the conditions $(H_1)$ by using conditions $(H_2)$ or $(H_3)$ instead. Precisely: in the (D;D), (D;R), (R;D) regimes, the hydrostatic principle holds under conditions $(H_2)$ and in the ($N_e$ ; $N_e$) regime, it holds under conditions $(H_3)$.
\end{rem}

The proof of Theorem \ref{T hydrostat} is done in Section 4. It essentially relies on an intermediate result stated in Theorem \ref{T3} regarding the convergence of solutions of the hydrodynamic equation towards the unique stationary state. This result is non standard as it involves a system of coupled equations and we prove it in the second Subsection of Section 4.

\section{Proof of the hydrodynamic limit}
As said before, we focus on the cases where $\ttl\in [0,1), \ttr=1$ and $\ttl>1, \ttr=1$.
We follow the entropy method introduced by Guo, Papanicolaou and Varadhan in \cite{cmp/1104161907} to prove the hydrodynamic limit. First, we prove tightness of the sequence of measures $(Q_{N}^{\wht})_{N\geq 1}$. Then, we show that any limit point of $(Q_{N}^{\wht})_{N\geq 1}$ is a Dirac mass concentrated on a weak solution of \eqref{D+R} if $\ttl\in [0,1), \ttr=1$, or \eqref{N+R} if $\ttl>1, \ttr=1$. Finally, we prove uniqueness of the solution of the hydrodynamic equations at fixed initial data. We do not give details for the standard steps but rather, insist on the specific difficulties arising in our case, namely, the $d$-dimensional replacement lemmas (subsections 3.2.3 and 3.2.4) and the uniqueness of the solution of the hydrodynamic equation (sections 3.5).
\subsection{The martingale property and tightness}
Recall from \eqref{Generator} the definition of the total generator $L_N$.
By Dynkin's formula (see \cite[Appendix A.1]{kipnis_scaling_1999}), for $1\leq i \leq 3$, $t\in [0,T]$ and $\whg \in \mathcal{C}_{c}^{1,2}([0,T]\times \overline{B})$,
\begin{equation}\label{Martingale}
    \begin{split}
    M_{i,t}^N(\whg) &:= \langle\pi_{i,t}^N,G_{i,t}\rangle - \langle\pi_{i,0}^N,G_{i,0}\rangle-\int_{0}^t\langle\pi_{i,s}^N,\partial_{s}G_{i,s}\rangle ds\\
    &- N^2\int_{0}^t\mathcal{L}_{N}\langle\pi_{i,s}^N,G_{i,s}\rangle ds - \int_{0}^t \mathbb{L}_{N}\langle\pi_{i,s}^N,G_{i,s}\rangle ds - N^2\int_{0}^t L_{\widehat{b},\wht,N}\langle\pi_{i,s}^N,G_{i,s}\rangle ds
    \end{split}
\end{equation}
is a martingale with respect to the natural filtration $\mathcal{F}_{t}=\sigma(\e_{s},~ s\leq t)$
and with quadratic variation given by:
\begin{equation} \label{Varquadr}
 \int_{0}^t {L}_{N}\left(\langle\pi_{i,s}^N,G_{i,s}\rangle ^2\right)ds  -2N^2\int_{0}^t \langle\pi_{i,s}^N,G_{i,s}\rangle 
 L_N \left( \langle \pi_{i,s}^N,G_{i,s}\rangle \right)ds.
\end{equation}

We then have that
\[\widehat{M}_{t}^N(\whg):=\sum_{i=1}^3 M_{i,t}^N(\whg)\]
is also a martingale whose quadratic variation is known. In order to develop the integral terms in \eqref{Martingale}, introduce the discrete second derivative in the direction $e_k$ (for $1\leq k \leq d$) in the bulk, the discrete Laplacian, and the discrete gradient in the direction $e_1$ at the boundary: for $x\in B_{N}\setminus \Gamma_{N}$,
\[(\partial_{e_{k}}^N)^2G(x/N) :=\sum_{k=1}^d N^2\Big(G\Big(\frac{x+e_{k}}{N}\Big)+G\Big(\frac{x-e_{k}}{N} \Big) - 2G\Big( \frac{x}{N}\Big) \Big),\]

\[\Delta_{N}G(x/N) := \sum_{k=1}^d (\partial_{e_{k}}^N)^2G(x/N),~ ~ (\partial_{e_{1}}^N)^{-}H(x/N) := N\Big(H \Big(\frac{x}{N}\Big) - H\Big(\frac{x-e_{1}}{N}\Big)\Big)  \]
and
\[(\partial_{e_{1}}^N)^{+}H(x/N) := N\Big(H \Big(\frac{x+e_{1}}{N}\Big) - H\Big(\frac{x}{N}\Big)\Big).\]
Computations yield
\begin{equation}\label{Integral2}
\begin{split}
    &M_{i,t}^N(\widehat{G}) = \langle\pi_{i,t}^N,G_{i,t}\rangle - \langle\pi_{i,0}^N,G_{i,0}\rangle-\int_{0}^t \langle\pi_{i,s}^N,\partial_{s}G_{i,s}\rangle ds\\
    &- \int_{0}^t\frac{D}{N^d}\sum_{x\in B_{N}\setminus \Gamma_{N}}\Delta_{N}G_{i,s}(x/N)\e_{i,s}(x)ds - \int_{0}^t\frac{D}{N^d}\sum_{x\in \Gamma_{N}}\sum_{k=2}^d (\partial_{e_{k}}^N)^2G_{i,s}(x/N)\e_{i,s}(x)ds\\
    &- \int_{0}^t\Big[\frac{D}{N^{d-1}}\sum_{x\in \Gamma_{N}^+}(\partial_{e_{1}}^N)^{-}G_{i,s}(x/N)\e_{i,s}(x) - \frac{D}{N^{d-1}}\sum_{x\in \Gamma_{N}^{-}}(\partial_{e_{1}}^N)^{+}G_{i,s}(x/N)\e_{i,s}(x) \Big]ds\\
    &- \int_{0}^t\frac{1}{N^d}\sum_{x\in B_{N}}G_{i,s}(x/N)\tau_{x}f_{i}(\e_s)ds\\
    & + \frac{N^2}{N^{d+\ttl}}\int_{0}^t\sum_{x\in \Gamma_{N}^-}G_{i,s}(x/N)\Big(\e_{i,s}(x)-b_{i}(x/N)\Big)ds\\
    &+ \frac{N^2}{N^{d+\ttr}}\int_{0}^t\sum_{x\in \Gamma_{N}^+}G_{i,s}(x/N)\Big(\e_{i,s}(x)-b_{i}(x/N)\Big)ds,
\end{split}
\end{equation}
where we used that 
\begin{equation}\label{IntegralCPRS}
    \mathbb{L}_{N}\langle\pi_{i,s}^N,G_{i,s}\rangle = \frac{1}{N^d}\sum_{x\in B_{N}}G_{i,s}(x/N)\tau_{x}f_{i}(\e_s),
\end{equation}
with
\begin{equation*}
    \begin{array}{ll}
         &\mathbb{L}_N \e_{1}(0) =  \beta_{B_{N}}(0,\e)\e_{0}(0) + \e_{3}(0)-(r+1)\e_{1}(0)=:f_{1}(\e),  \\
         & \mathbb{L}_N \e_{2}(0) = r\e_{0}(0)+\e_{3}(0)-\beta_{B_{N}}(0,\e)\e_{2}(0) -\e_{2}(0)=:f_{2}(\e),\\
         &\mathbb{L}_N \e_{3}(0)= \beta_{B_{N}}(0,\e)\e_{2}(0)+r\e_{1}(0)-2\e_{3}(0)=:f_{3}(\e).
    \end{array}
\end{equation*}
The second and third lines in \eqref{Integral2} correspond to the computation of the time integral associated to $ N^2\mathcal{L}_N$, the fourth line in \eqref{Integral2} corresponds to the time integral associated to $ \mathbb{L}_N$ and the last term, to the integral associated to $N^2L_{\widehat{b},\widehat{\theta},N}$.

For $i\in \{1,2,3\}$, a computation of the quadratic variation of the martingale $M_{i,t}^N (\widehat{G})$ shows that its expectation vanishes as $N\uparrow \infty$. Therefore, by Doob's inequality, for every $\delta >0$,
\begin{equation}\label{proba1}
    \limsup_{N \rightarrow \infty}{\mathbb P}_{\mu^N} \Big[~\underset{0\leq t\leq T}{\sup}~ \Big| M_{i,t}^N(\widehat{G})\Big|>\delta \Big]=0.
\end{equation}

\begin{prop}\label{tight}
The sequence of measures $(Q_{N}^{\wht})_{N\geq1}$ is tight in $D([0,T],{\mathcal{M}^3})$.
\end{prop}

We refer to \cite[Chapter 4]{kipnis_scaling_1999} for details regarding the proof of tightness of a sequence of probability measures. It is enough to show that for every $H$ in a dense subset of $\mathcal{C}(\overline{B})$ for the $L^2$ norm, for every $1\leq i\leq 3$,
\begin{equation} \label{pi}
    \underset{\delta \rightarrow 0}{\limsup}~ \underset{N \rightarrow \infty}{\limsup}~  \mathbb{E}_{\mu_{N}}\Big[~ \underset{|t-s|\leq \delta}{\sup}~ \Big|\langle\pi_{i,t},H\rangle - \langle\pi_{i,s},H \rangle\Big|~  \Big]=0.
\end{equation}
By density of $\mathcal{C}^2_{0}(\overline{B})$ in $\mathcal{C}^2(\overline{B})$ for the $L^1$ norm, it is enough to show \eqref{pi} with $H \in \mathcal{C}^2_{0}(\overline{B})$, so that $H$ vanishes at the boundary. To prove that, we use the martingale and its quadratic variation introduced in \eqref{Martingale} and \eqref{Varquadr}, and show that
\begin{equation}\label{A}
    \underset{\delta \rightarrow 0}{\limsup}~ \underset{N \rightarrow \infty}{\limsup}~  \mathbb{E}_{\mu_{N}}\Big[~\underset{|t-s|\leq \delta}{\sup}~ \Big|M_{i,t}^N(H)- M_{i,s}^N(H)\Big|~ \Big] = 0,
\end{equation}
and
\begin{equation}\label{B}
    \underset{\delta \rightarrow 0}{\limsup}~ \underset{N \rightarrow \infty}{\limsup}~  \mathbb{E}_{\mu_{N}}\Big[~\underset{|t-s|\leq \delta}{\sup}~ \Big|\int_{s}^t L_{N}\langle\pi_{i,r}^N,H\rangle dr\Big|~\Big] = 0.
\end{equation} We get \eqref{A} using the triangular inequality and \eqref{proba1}. To prove \eqref{B}, we show that there is a constant $C$ depending only on $H$ such that for every $r\in [0,T]$,
\begin{equation}\label{LN}
   \Big|L_{N}\langle\pi_{i,r}^N,H\rangle \Big|\leq C. 
\end{equation} For that,  we use the decomposition of $L_N$ and the fact that $H$ vanishes at the boundary as well as explicit computations and the fact that the $f_i$'s are uniformly bounded in $N$.

\subsection{Replacement Lemmas }
In order to characterize the limit points of a sequence $(Q_N^{\wht})_{N\geq 1}$, we need to close the equation \eqref{Integral2}. That means that we want to show that each term of the martingale converges to a term that appears in the weak formulation of the solution of the hydrodynamic equation, and that the martingale converges to zero. For that, we perform a replacement lemma in the bulk and one at the boundary. The replacement lemma in the bulk (Proposition \ref{Bulk}) is exactly the same as in \cite[Lemma 4.2]{KMS} and we refer to that article for a detailed proof. Here we focus on the replacement lemmas at the boundary and more specifically on the left-hand side boundary (the same statements hold on the right-hand side). There are two replacement lemmas: one for $\ttl\in [0,1)$ whose formulation coincides with the replacement lemma at the boundary in \cite[Proposition 4.3]{KMS} (corresponding to a Dirichlet condition), and one for $\ttr \geq 1$, whose formulation involves particle densities over small macroscopic boxes. 
\subsubsection{Dirichlet forms}
Let us recall the expressions introduced in \cite[Section 5]{KMS} of the Dirichlet forms associated to each dynamics. For that, recall the correspondences \eqref{corresp} and \eqref{fonctiondexiw}. For $f:\whsN\rightarrow \R$ and $\mu$ a measure on $\whsN$,
\begin{equation*}
    \mathcal{D}_{N}(f,\mu)=\sum_{k=1}^d\sum_{\xek\in B_{N}^2}\int_{\whsN}D\Big(\sqrt{f(\xiwk)}-\sqrt{f\xiw} \Big)^2d\mu\xiw,
\end{equation*}
\begin{equation*}
\begin{split} 
    D_{\widehat{b},\wht,N}(f,\mu) &:=\frac{1}{N^{\ttl}}\sum_{i=0}^3\sum_{x\in \Gamma_{N}^-}\int_{\whsN}b_{i}(x/N)(1-\e_i(x))\Big(\sqrt{f(\sigma_{i,x}\xiw)}-\sqrt{f\xiw} \Big)^2d\mu\xiw\\
    &+ \frac{1}{N^{\ttr}}\sum_{i=0}^3\sum_{x\in \Gamma_{N}^+}\int_{\whsN}b_{i}(x/N)(1-\e_i(x))\Big(\sqrt{f(\sigma_{i,x}\xiw)}-\sqrt{f\xiw} \Big)^2d\mu\xiw,
\end{split}
\end{equation*}
and
\begin{equation*}
    \begin{split}
            \mathbb{D}_{N}(f,\mu) &:= \sum_{x\in B_{N}}\int_{\whsN}\Big[r(1-\omega(x))+\omega(x) \Big]\Big(\sqrt{f(\xi,\sigma^x\omega)} - \sqrt{f\xiw} \Big)^2d\mu\xiw\\
    & +\sum_{x\in B_{N}}\int_{\whsN}\Big[\beta_{B_{N}}(x,\xi,\omega)(1-\xi(x))  +\xi(x) \Big]\Big(\sqrt{f(\sigma^x\xi,\omega)} - \sqrt{f\xiw} \Big)^2d\mu\xiw.
    \end{split}
\end{equation*}
In the proofs of the Replacement lemmas, we will widely make use of the following inequalities. 
\begin{lemma}\label{synthese}
    \begin{itemize}
        \item[(i)] Consider $\widehat{\alpha}$ a smooth profile which satisfies \eqref{a)}. 
        There is a constant $C_1>0$ such that for any density function $f:\whsN\rightarrow \R$ with respect to the measure $\nualph$,
         \begin{equation}\label{Echh}
        \langle\mathcal{L}_{N}\sqrt{f},\sqrt{f}\rangle_{\nualph} \leq -\frac{1}{4}\mathcal{D}_{N}(f,\nualph) + C_1N^{d-2}.
    \end{equation}
    \item[(ii)] Consider $\widehat{\alpha}$ a smooth which satisfies \eqref{a)} including constants. There is a constant $C_2>0$ such that for any density function $f:\whsN\rightarrow \R$ with respect to the measure $\nualph$,
    \begin{equation}\label{IneggenCPRS}
   \langle \mathbb{L}_{N}\sqrt{f},\sqrt{f}\rangle_{\nualph} \leq C_2N^d. 
\end{equation}
\item[(iii)] Consider $\widehat{\alpha}$ a smooth profile which satisfies \eqref{a)} and \eqref{b)}, then for any density function $f:\whsN\rightarrow \R$ with respect to the measure $\nualph$,
    \begin{equation} \label{bord}
        \langle L_{\widehat{b},\wht,N}\sqrt{f},\sqrt{f}\rangle_{\nualph} = -\frac{1}{2}D_{\widehat{b},\wht,N}(f,\nualph). 
    \end{equation}
    \end{itemize}
\end{lemma}
We refer to \cite{KMS}, Lemma 6.1 for the proof. The authors use some change of variable formulas in the same spirit as those given in \eqref{chtva1} and \eqref{chtva2}. For point $(iii)$, they use an alternative expression of the boundary generator expressed in terms of $\xiw$.   

\subsubsection{Replacement lemma in the bulk. }
Let us first introduce a few notations. Given a smooth profile $\widehat{\alpha}$, and a function $\phi:\whsN\rightarrow \R$, denote by $\overset{\sim}{\phi}(\widehat{\alpha})$ the expectation of $\phi$ under $\nualph$. For $\ell\in \N$, introduce
\begin{equation}\label{box}
    \Lambda_{x}^{\ell}= \{y\in B_N,~ \|y-x\|\leq \ell \},
\end{equation}
where $\|y-x\|=\max\{|y_i-x_i|,~1\leq i \leq d \}$, and denote by
$\e_{i}^{\ell}(x)$ the average of $\e$ in $\Lambda_{x}^{\ell}$, that is,
\begin{equation}\label{average}
    \e_{i}^{\ell}(x) = \frac{1}{|\Lambda_{x}^{\ell}|}\sum_{y\in \Lambda_{x}^{\ell}}\e_{i}(y),~ ~ \text{for}~ 1\leq i \leq 3.
\end{equation}
Introduce the vector
\[\widehat{\e}^{\ell}(x) = (\e^{\ell}_{1}(x),\e^{\ell}_{2}(x),\e^{\ell}_{3}(x))\]
and for $\varepsilon>0$,

\begin{equation*}
    V_{\varepsilon N}\xiw= \Big|\frac{1}{|\Lambda_{0}^{\lfloor \varepsilon N \rfloor}|}\sum_{y\in \Lambda_{0}^{\lfloor \varepsilon N \rfloor}}\tau_{y}\phi\xiw - \overset{\sim}{\phi}(\widehat{\e}^{\lfloor \varepsilon N \rfloor}(0))  \Big|.
\end{equation*}
In the sequel, we will write $\varepsilon N$ instead of $\lfloor \varepsilon N \rfloor$. The replacement lemma in the bulk stated and proved in \cite[Lemma 4.2]{KMS} is the following:
\begin{prop} \label{Bulk} For any $G\in \mathcal{C}_{0}^{1,2}$ and for any function $\phi:\whsN \rightarrow \R$
,
\begin{equation*} 
    \underset{\varepsilon \rightarrow 0}{\limsup} ~\underset{N \rightarrow \infty}{\limsup}~ \mathbb{E}_{\mu_{N}}\Big[\frac{1}{N^d}\sum_{x\in B_N}\intT\big|G_{s}(x/N)\big|\tau_{x}V_{\varepsilon N}(\xi_s,\omega_s) ds  \Big]  =0.
\end{equation*}
\end{prop}

\subsubsection{Replacement lemma at the left-hand side boundary for $ \ttl\in [0,1)$.}

Here we fix $\ttl$ in $[0,1)$ and prove the replacement lemma at the left-hand side boundary. It essentially states that when performing the macroscopic limit $N\rightarrow \infty$, we can replace $\e_i(x)$ by $b_i(x/N)$.  For $\ttr\in [0,1)$, the replacement lemma at the right-hand side boundary is exactly the same. Recall that this result has been proved for $\ttl=\ttr=0$ in \cite[Section 6]{KMS} and we generalize it here to the case where the left-hand side (or right-hand side) parameter $\ttl$ is allowed to vary in $[0,1)$.
\begin{prop} \label{Repdirich} For any sequence of measures $(\mu_{N})_{N\geq 0}$ on $\whsN$, for any $G\in \mathcal{C}^{1,2}([0,T]\times \overline{B})$ and any $i\in \{1,2,3\}$, for any $t\in [0,T]$, for all $\delta >0$,
\begin{equation}\label{ReplacementDirichlet}
    \underset{N \rightarrow \infty}{\limsup}~ \mathbb{P}_{\mu_{N}}\Big[\Big|\int_{0}^t\frac{1}{N^{d-1}}\sum_{x\in \Gamma_{N}^-}G(s,x/N)\big(\e_{i,s}(x)-b_{i}(x/N) \big)ds \Big| >\delta \Big] = 0.
\end{equation}
\end{prop}
\noindent Note that the replacement lemma at the right-hand side boundary for $\ttr\in [0,1)$ states as above, with the sum in $x$ carrying over $\Gamma_N^+$ rather than $\Gamma_N^-$.
\begin{proof}
    Fix an $i\in \{1,2,3\}$. It is enough to show that
\begin{equation*}
    \underset{N \rightarrow \infty}{\limsup}~\frac{1}{N^d}\log\Big( \mathbb{P}_{\mu_{N}}\Big[\Big|\int_{0}^t\frac{1}{N^{d-1}}\sum_{x\in \Gamma_{N}^-}G(s,x/N)\big(\e_{i,s}(x)-b_{i}(x/N) \big)ds \Big| >\delta \Big]\Big) = -\infty.
\end{equation*}
Consider $\widehat{\alpha}$ a smooth profile satisfying conditions \eqref{a)} and \eqref{b)}. For $a>0$,
\begin{equation*}
    \begin{split}
    &\mathbb{P}_{\mu_{N}}\Big[\Big|\int_{0}^t\frac{1}{N^{d-1}}\sum_{x\in \Gamma_{N}^-}G(s,x/N)\big(\e_{i,s}(x)-b_{i}(x/N) \big)ds \Big| >\delta \Big]\\
    &\leq \sup_{\xiw\in \whsN}\frac{d\mu_{N}}{d\nualph}\xiw\times \mathbb{P}_{\nualph}\Big[\Big|\int_{0}^t\frac{1}{N^{d-1}}\sum_{x\in \Gamma_{N}^-}G(s,x/N)\big(\e_{i,s}(x)-b_{i}(x/N) \big)ds \Big| >\delta \Big]
    \\
    &\leq \exp(K_{0}N^d-a\delta N^d)\mathbb{E}_{\nualph}\Big[\exp\Big(aN^d\Big|\int_{0}^t\frac{1}{N^{d-1}}\sum_{x\in \Gamma_{N}^-}G(s,x/N)\big(\e_{i,s}(x)-b_{i}(x/N) \big)ds \Big| \Big) \Big].
    \end{split}
\end{equation*}
We used, in the first inequality, that the Radon-Nikodym derivative of $\mu_{N}$ with respect to $\nualph$ is bounded by $\exp(K_{0}N^d)$ with $K_{0}$ a constant, and Chebychev's inequality in the second line. Therefore,
\begin{equation}\label{ineglog}
\begin{split}
    &\frac{1}{N^d}\log\Big( \mathbb{P}_{\mu_{N}}\Big[\Big|\int_{0}^t\frac{1}{N^{d-1}}\sum_{x\in \Gamma_{N}^-}G(s,x/N)\big(\e_{i,s}(x)-b_{i}(x/N) \big)ds \Big| >\delta \Big]\Big)\\
    &\leq -a\delta + K_{0}\\
    &+ \frac{1}{N^d}\log\Big\{\mathbb{E}_{\nualph}\Big[\exp\Big\{aN^d\Big|\int_{0}^t\frac{1}{N^{d-1}}\sum_{x\in \Gamma_{N}^-}G(s,x/N)\big(\e_{i,s}(x)-b_{i}(x/N) \big)ds \Big| \Big\} \Big] \Big\}.
\end{split}
\end{equation}
It is enough to show that the $\limsup$ of the last term is uniformly bounded in $a$ and then, take $a\rightarrow \infty$.
Since $e^{|x|} \leq e^{x} + e^{-x}$, using inequality \eqref{Conv}, we show that the last term in \eqref{ineglog}  without the absolute values, is uniformly bounded in $a$ and $N$. Applying Feynman-Kac's inequality (see \cite[Appendix A.1]{kipnis_scaling_1999}) with
\[V(s,(\xi_s,\omega_s)) = \frac{aN^d}{N^{d-1}} \sum_{x\in \Gamma_{N}^-}G(s,x/N)(\e_{i,s}(x)-b_{i}(x/N)),\]
we get that
\begin{equation*}
\begin{split}
    &\frac{1}{N^d}\log\Big(\mathbb{E}_{\nualph}\Big[\exp\Big(aN^d\int_{0}^t\frac{1}{N^{d-1}}\sum_{x\in \Gamma_{N}^-}G(s,x/N)\big(\e_{i,s}(x)-b_{i}(x/N) \big)ds  \Big) \Big] \Big)\\
    &\leq \int_0^t ds ~ \underset{f}{\sup}\Big\{ \int_{\whsN}\frac{a}{N^{d-1}}\sum_{x\in \Gamma_{N}^-}G(s,x/N)
    \big(b_{i}(x/N)-\e_{i}(x)\big)f\xiw d\nualph\xiw\\
    &+ \frac{1}{N^d}\langle L_{N}\sqrt{f},\sqrt{f}\rangle_{\nualph} \Big\},
\end{split}
\end{equation*} 
where the supremum is taken over densities with respect to $\nualph$. Note that for $x\in \Gamma_{N}^-$,
\begin{equation*}
    b_{i}(x/N)-\e_{i}(x) = \sum_{j\neq i}(b_{i}(x/N)\e_{j}(x)-b_{j}(x/N)\e_{i}(x)),
\end{equation*}
and, for $j\neq i$, by the change of variable presented in \eqref{chtva2},

\begin{equation}\label{chtva}
\begin{split}
    \int\e_{i}(x)b_{j}(x/N)f\xiw d\nualph\xiw&= \int\e_j(x) b_i(x/N) f(\sigma_{i,x}\xiw) d\nualph\xiw.
\end{split}
\end{equation}
Therefore, 
\begin{equation} \label{Vt}
    \begin{split}
    &G(s,x/N) \int\big(b_{i}(x/N)-\e_{i}(x)\big)f\xiw d\nualph\xiw\\
    &=G(s,x/N)\int b_{i}(x/N)\sum_{j\neq i}\e_{j}(x)f\xiw d\nualph \xiw\\
    &- G(s,x/N)\int b_{i}(x/N)\sum_{j\neq i}\e_{j}(x)f(\sigma_{i,x}\xiw)d\nualph\xiw \\
    &= -G(s,x/N)\int b_{i}(x/N)(1-\e_{i}(x))(f(\sigma_{i,x}\xiw)-f\xiw)d\nualph\xiw\\
    &\leq \frac{A}{2}\int b_{i}(x/N)(1-\e_{i}(x))\Big(\sqrt{f(\sigma_{i,x}\xiw)}-\sqrt{f\xiw} \Big)^2d\nualph\xiw\\
    & + \frac{1}{2A}\big( G(s,x/N)\big)^2\int b_{i}(x/N)(1-\e_{i}(x))\Big(\sqrt{f(\sigma_{i,x}\xiw)}+\sqrt{f\xiw} \Big)^2d\nualph\xiw,
    \end{split}
\end{equation}
\noindent where we used \eqref{useful2} in the last line replacing $A$ by $AN$, with $A>0$. Summing \eqref{Vt} over $\Gamma_N^-$ and multiplying by $\frac{a}{N^{d-1}}$ yields,
\begin{equation*}
    \begin{split}
    &\int_{\whsN}\frac{a}{N^{d-1}}\sum_{x\in \Gamma_{N}^-}G(s,x/N)(b_{i}(x/N)-\e_{i}(x))f\xiw d\nualph\xiw\\
    & \leq \frac{aAN^{\ttl}}{2N^{d-1}}D_{\widehat{b},\wht,N}(f,\nualph)+ \frac{a}{AN^{d-1}}\sum_{x\in \Gamma_{N}^-}\big(G(s,x/N)\big)^2\\
    &\leq \frac{aAN^{\ttl}}{2N^{d-1}}D_{\widehat{b},\wht,N}(f,\nualph) + \frac{a}{A}\|G^2\|_{\infty},
    \end{split}
\end{equation*}
where the second term in the first inequality comes from Cauchy-Schwarz's inequality, the fact that $f$ is a density, the change of variable formula \eqref{chtva} and the fact that each coordinate of $\widehat{b}$ is bounded by $1$. Therefore, using \eqref{Echh},  \eqref{IneggenCPRS} and \eqref{bord}  to bound $\langle L_N\sqrt{f},\sqrt{f}\rangle_{\nualph}$ and the fact that a Dirichlet form is positive we are left with 

\begin{equation}\label{ineglog2}
\begin{split}
         &\frac{1}{N^d}\log\Big(\mathbb{E}_{\nualph}\Big[\exp\Big(aN^d\int_{0}^t\frac{1}{N^{d-1}}\sum_{x\in \Gamma_{N}^-}G(s,x/N)\big(\e_{i,s}(x)-b_{i}(x/N) \big)ds  \Big) \Big] \Big)\\
     &\leq \int_0^t ds~  \underset{f}{\sup}\Big\{\frac{aAN^{\ttl}}{2N^{d-1}}D_{\widehat{b},\wht,N}(f,\nualph) + \frac{a}{A}\|G^2\|_{\infty} + \frac{1}{N^d}\langle L_N\sqrt{f},\sqrt{f}\rangle _{\nualph} \Big\} \\
     &\leq  T~ \underset{f}{\sup}\Big\{\frac{aAN^{\ttl}}{2N^{d-1}}D_{\widehat{b},\wht,N}(f,\nualph) + \frac{a}{A}\|G^2\|_{\infty} -\frac{N^2}{N^d}\big(\mathcal{D}_N(f,\nualph) + D_{\widehat{b},\wht,N}(f,\nualph)\big)+ C_1 +C_2 \Big\}\\
     & \leq T~ \underset{f}{\sup}\Big\{\Big(\frac{aAN^{\ttl}}{2N^{d-1}}-\frac{N^2}{N^d} \Big)D_{\widehat{b},\wht,N}(f,\nualph) \Big\} + \frac{a}{A}\|G^2\|_{\infty} + TC_1 + TC_2.
\end{split}
\end{equation}
Now, taking $A= \frac{2}{a}N^{1-\ttl}$, collecting \eqref{ineglog} and \eqref{ineglog2} we are left with
\begin{equation}
\begin{split}
    &\underset{N \rightarrow \infty}{\overline{\lim}}\frac{1}{N^d}\log\Big( \mathbb{P}_{\mu_{N}}\Big[\Big|\int_{0}^t\frac{1}{N^{d-1}}\sum_{x\in \Gamma_{N}^-}G(s,x/N)\big(\e_{i,s}(x)-b_{i}(x/N) \big)ds \Big| >\delta \Big]\Big)\\
    &\leq \underset{N \rightarrow \infty}{\overline{\lim}}\Big( -a\delta + K_0 + \frac{1}{2}a^2N^{\ttl-1}\|G\|_{\infty}^2 + TC_1 + TC_2\Big)\\
    &\leq  -a\delta + K_0+ TC_1 + TC_2
\end{split}
\end{equation}
and then, taking $a\rightarrow \infty$, the result follows.
\end{proof}

\subsubsection{Replacement lemma at the left-hand side boundary for $\ttl\geq 1$.}
For $\ttl\geq 1$, the replacement lemma at the boundary involves particle densities over small macroscopic boxes. Again, the same replacement lemma holds at the right-hand side boundary for $\ttr\geq 1$. In fact, we will see in the proof that the lemma holds for any positive value of $\ttl$, resp. $\ttr$, regardless of whether $\ttl$ resp. $\ttr \geq 1$. Here, as we are working in arbitrary dimension, some care must be taken in the proof when adapting the argument used for instance in \cite[Chapter 5]{kipnis_scaling_1999}.
\begin{prop} \label{lhsrep}
 For any sequence of measures $(\mu_N)_{N\geq 0}$ on $\whsN$, for any $G \in \mathcal{C}^{1,2}([0,T]\times \overline{B})$, for all $i\in\{1,2,3\}$ and any $t\in [0,T]$,
\begin{equation}\label{ReplDirichlet}
    \underset{\varepsilon \rightarrow 0}{\limsup}~  \underset{N \rightarrow \infty}{\limsup}~ \mathbb{E}_{\mu_{N}}\Big[~\Big|\frac{1}{N^{d-1}}\sum_{x\in \Gamma_{N}^-}\int_{0}^tG(s,x/N)(\e_{i,s}^{\varepsilon N}(x) -\e_{i,s}(x))ds \Big| ~\Big]=0.
\end{equation}
\end{prop}
\begin{proof}
For a vector $x=(x_1,\cdots,x_d)\in B_N$, write
$x=(x_1, \check{x})$, where $\check{x}=(x_2,\cdots,x_d)\in \Tnd$. First, consider the expression in the expectation without absolute value and the time integral, and rewrite it for any $s\in[0,t]$ as
\begin{equation}
 \begin{aligned}
  \mathfrak{I}_{N,\varepsilon}(G_s,\eta_s):= \sum_{j_1=0}^{\varepsilon N}~ 
  \sum_{\check{k}\in [-\varepsilon N,\varepsilon N]^{d-1}}
  \frac{1}{N^{d-1}} \sum_{x\in \Gamma_N^-  }\frac{1}{|\Lambda_x^{\varepsilon N}|}G(s,x/N)
  \Big(\eta_{i,s}(x+(j_1, \check{k})) -\eta_{i,s}(x)\Big),
 \end{aligned}
\end{equation}
where we recall that $\Lambda_x^{\varepsilon N}$ is defined in \eqref{box}. For $x\in \Gamma_N^- $,
\begin{equation}\label{Bis}
    |\Lambda_x^{\varepsilon N}|= (\varepsilon N +1)(2\varepsilon N +1)^{d-1}.
\end{equation}
Fix $(j_1,\check{k})\in \{0,\cdots,\varepsilon N\}\times [-\varepsilon N,\varepsilon N]^{d-1}$ . The sum over $x\in \Gamma_N^-$ can be handled in the following way
\begin{equation}\label{hand}
 \begin{aligned}
  & \frac{1}{N^{d-1}}\sum_{x\in \Gamma_N^-}G(s,x/N)
  \eta_{i,s}(x+(j_1, \check{k})) \\
&=  \frac{1}{N^{d-1}}\sum_{x\in \Gamma_N^-  }
  \Big(G(s,x/N) - G(s,(x+(j_1, \check{k}))/N)\Big)
  \eta_{i,s}(x+(j_1, \check{k})) \\
&+ \frac{1}{N^{d-1}}   \sum_{x\in \Gamma_N^-  }
   G(s,(x+(j_1, \check{k}))/N)
  \eta_{i,s}(x+(j_1, \check{k})). \,
 \end{aligned}
\end{equation}
Since $G$ is twice differentiable in space, a Taylor expansion allows us to bound the first term on the right-hand side of \eqref{hand} by $d\varepsilon C_G(N)$ where $C_G(N)$ is uniformly bounded in $N$ by a constant $C_G$, depending only on $G$. Now, rewrite the last term in \eqref{hand} as follows:
\begin{equation}\label{hannd}
 \begin{aligned}
  &\frac{1}{N^{d-1}}\sum_{x\in \Gamma_N^-  }
   G(s,(x+(j_1, \check{k}))/N)
  \eta_{i,s}(x+(j_1, \check{k}))\\
  &= \frac{1}{N^{d-1}} \sum_{\check{x}\in \Tnd}G(s, (-N+j_1,\check{x}+ \check{k})/N) \e_{i,s}((-N+j_1,\check{x}+ \check{k}))\\
  &= \frac{1}{N^{d-1}} \sum_{\check{x}\in \Tnd}G(s, (-N+j_1,\check{x})/N) \e_{i,s}((-N+j_1,\check{x}))\\
&  = \frac{1}{N^{d-1}}\sum_{x\in \Gamma^-_N  }
   G(s,x/N)
  \eta_{i,s}(x+j_1e_1)) + \varepsilon C'_G(N),
 \end{aligned}
\end{equation}
where again, to get the last line, we used a Taylor expansion of $G$ and with $C'_G(N)$ uniformly bounded in $N$ by a constant $C'_G$ depending only on $G$. Using \eqref{hand}, \eqref{Bis} and \eqref{hannd}, we get
\begin{equation*}
 \begin{aligned}
\mathfrak{I}_{N,\varepsilon}(G_s,\eta_s)& = \varepsilon C''_G(N) +  \frac{1}{\varepsilon N +1} \sum_{j_1=0}^{\varepsilon N}~ 
  \frac{1}{N^{d-1}} \sum_{x \in \Gamma_N^-}G(s,x/N) \big[\e_{i,s}(x+j_1e_1)-\e_{i,s}(x) \big]\\
 & := \varepsilon C''_G(N) \, +\mathfrak{R}_{N,\varepsilon}(G_s,\eta_s),
 \end{aligned}
\end{equation*}
with $C''_G(N)$, uniformly bounded by $C''_G$, a constant that only depends on $G$.
Therefore, we are left to prove that
\[
 \underset{\varepsilon \rightarrow 0}{\limsup}~  \underset{N \rightarrow \infty} {\limsup}~ 
\mathbb{E}_{\mu_{N}}\Big[\Big|
\int_0^t  \mathfrak{R}_{N,\varepsilon}(G_s,\eta_s) ds
 \Big|\Big]=0\, .
\]
Consider $\widehat{\alpha}$ a smooth profile satisfying conditions \eqref{a)} and \eqref{b)} . By the entropy inequality (see \cite[Appendix 1]{kipnis_scaling_1999}), for any $A>0$,
\begin{equation}\label{RL1}
\begin{split}
    &\mathbb{E}_{\mu_{N}}\Big[\Big|
\int_0^t  \mathfrak{R}_{N,\varepsilon}(G_s,\eta_s) ds
 \Big|\Big]\\
    &\leq \frac{1}{AN^d}H(\mu_{N}|\nualph) + \frac{1}{AN^d}\log \mathbb{E}_{\nualph}\Big[\exp\Big(AN^d\Big|\int_0^t  \mathfrak{R}_{N,\varepsilon}(G_s,\eta_s) ds\Big|\Big) \Big].
\end{split}
\end{equation}
As $B_N$ is finite, there is a constant $K_{0}>0$ such that $H(\mu_{N}|\nualph) \leq K_{0}N^d $ so the first term in \eqref{RL1} is bounded by $K_{0}/A$. Let us show that the second term tends to zero when $N\rightarrow \infty$ and $\varepsilon \rightarrow 0$ and then take $A$ arbitrarily big. Again, by \eqref{Conv}, it is enough to show that the second term in \eqref{RL1} without the absolute values in the exponential, tends to zero.  By Feynman-Kac's inequality, 
\begin{equation}\label{Sup}
\begin{split}
    &\frac{1}{AN^d}\log \mathbb{E}_{\nualph}\Big[\exp\Big(AN^d\int_0^t  \mathfrak{R}_{N,\varepsilon}(G_s,\eta_s) ds\Big) \Big]\\
    &\leq \int_0^t ds \; \underset{f}{\sup}\Big[ \int
    \mathfrak{R}_{N,\varepsilon}(G_s,\eta)
    \big)f\xiw d\nualph\xiw + \frac{1}{AN^d}\langle L_{N}\sqrt{f},\sqrt{f}\rangle_{\nualph}\Big]
\end{split}
\end{equation}
where the supremum is taken over densities with respect to $\nualph$. Now, rewrite $\mathfrak{R}_{N,\varepsilon}(G_s,\eta)$ using a telescopic sum to write the differences $\e_{i}(x+j_1e_1)-\e_{i}(x) $: 
\[\mathfrak{R}_{N,\varepsilon}(G_s,\eta) = \frac{1}{\varepsilon N +1} \sum_{j_1=0}^{\varepsilon N}
  \frac{1}{N^{d-1}} \sum_{x \in \Gamma_N^-} \sum_{\ell =0}^{j_1-1}
  G(s,x/N) \Big[\eta_{i}((\ell +1)e_1 +x) -\eta_{i}(\ell e_1 + x) \Big] \]
Fix $0\le \ell\le j_1\le \varepsilon N$. Performing the change of variable
\[\xiw \rightarrow (\xi^{\ell e_1 +x,(\ell +1)e_1 +x},\omega^{\ell e_1 +x,(\ell +1)e_1 +x}):= \xiw^{\ell,x} \]
and using \eqref{chtva1},
\begin{equation}\label{Up1}
 \begin{aligned}
    &\int \big( \eta_{i}((\ell+1) e_1 + x)-\eta_{i}(\ell e_1 + x)\big)f\xiw  d\nualph\xiw\\
    &=   \int \eta_{i}(\ell e_1 + x)  \Big[
    f(\xiw^{\ell,x})
    - f\xiw \Big] d\nualph\xiw\\
    &+  \int \eta_{i}((\ell+1) e_1 + x) f\xiw  
   \Big(1- \frac{\nualph(\xiw^{\ell,x})}{\nualph\xiw}  \Big) d\nualph\xiw.
   \end{aligned}
\end{equation}
To deal with the first term on the right-hand side of \eqref{Up1} using inequality \eqref{useful2}, this term is bounded by
\begin{equation*}
\begin{split}
     & \frac{B}{2}\int \eta_{i}(\ell e_1 + x)  \Big[
    \sqrt{f(\xiw^{\ell,x})}
    - \sqrt{f\xiw }\Big]^2 d\nualph\xiw\\
    &+ \frac{1}{2B}\int \eta_{i}(\ell e_1 + x)  \Big[
    \sqrt{f(\xiw^{\ell,x})}
    + \sqrt{f\xiw }\Big]^2 d\nualph\xiw\\
    &\le \frac{B}{2}\int \eta_{i}(\ell e_1 + x)  \Big[
    \sqrt{f(\xiw^{\ell,x})}
    - \sqrt{f\xiw }\Big]^2 d\nualph\xiw\\
    &+  \frac{1}{B}\int \eta_{i}(\ell e_1 + x)  \Big[
    f(\xiw^{\ell,x})
    + f\xiw \Big] d\nualph\xiw\\
    &\leq \frac{B}{2}\int \eta_{i}(\ell e_1 + x)  \Big[
    \sqrt{f(\xiw^{\ell,x})}
    - \sqrt{f\xiw }\Big]^2 d\nualph\xiw \\
    &+  \frac{1}{B} \Big[ 1+ \int \eta_{i}(\ell e_1 + x) 
    f(\xiw^{\ell,x}) d\nualph\xiw \Big]
\end{split}
\end{equation*}
where $B>0$ will be chosen later and where we used that $f$ is a density with respect to $\nualph$ in the last line. Note that $\xiw \mapsto f(\xiw^{\ell,x})$ is not a density but we can deal with the last integral term as follows:
\begin{equation}
    \begin{split}
    &\int \eta_{i}(\ell e_1 + x) 
    f(\xiw^{\ell,x}) d\nualph\xiw \\
    &=\sum_{j\neq i} \int \eta_{i}(\ell e_1 + x) \eta_{j}((\ell+1) e_1 + x) 
    f(\xiw^{\ell,x}) d\nualph\xiw\\
    &= \sum_{j\neq i} \int \eta_{j}(\ell e_1 + x) \eta_{i}((\ell+1) e_1 + x) \Big(1+ R_{i,j}^{\ell e_1 + x,(\ell+1) e_1 + x}(\widehat{\alpha}) \Big)
    f\xiw d\nualph\xiw\\
    &\leq 1+ \frac{C}{N},
    \end{split}
\end{equation}
where in the second line we used the change of variable formula \eqref{chtva1} and, in the last line, the fact that $f$ is a density with respect to $\nualph$ and that $R_{i,j}^{\ell e_1 + x,(\ell+1) e_1 + x}(\widehat{\alpha})=O(N^{-1})$ so bounded by $C/N$ where $C$ is a constant.
Therefore, for $N$ large enough,
for all $0 \le \ell\le j_1\le \varepsilon N$,
\begin{equation*}\label{Up2}
 \begin{aligned}
    &\int \big( \eta_{i}((\ell+1) e_1 + x)-\eta_{i}(\ell e_1 + x)\big)f\xiw  d\nualph\xiw \\
    & \le \frac{B}{2}\int \eta_{i}(\ell e_1 + x)  \Big[
    \sqrt{f(\xiw^{x,\ell})}
    - \sqrt{f\xiw }\Big]^2 d\nualph\xiw + \frac{2}{B} + \frac{C}{BN}.
   \end{aligned}
\end{equation*}
Now, let us deal with the second term in \eqref{Up1}. Using the explicit expression of the product measure $\nualph$, one  has that for $(x,x+ \ell e_1) \in B_{N}^2$
$$\frac{\nualph(\xi^{x,x+e_{\ell}},\omega^{x,x+e_{\ell}})}{\nualph\xiw} = \prod_{i=1}^3\Big(\frac{\alpha_{i}(x/N)}{\alpha_{0}(x/N)} \Big)^{\e_{i}(x+e_{\ell})-\e_{i}(x)}\Big(\frac{\alpha_{i}((x+e_{\ell})/N)}{\alpha_{0}((x+e_{\ell})/N)}  \Big). $$
Now using that $\alpha_{i}\big(\frac{x+e_{\ell}}{N}\big)= \alpha_{i}\big(\frac{x}{N}\big) + O\big(\frac{1}{N}\big)$
we have the following inequality: there is a constant $\Tilde{C}>0$ such that
\begin{equation} \label{Ineg}
    \Big|1- \frac{\nualph(\xi^{x,x+e_{\ell}},\omega^{x,x+e_{\ell}})}{\nualph\xiw}  \Big|\leq \frac{\Tilde{C}}{N}.
\end{equation}
Therefore, the second term in \eqref{Up1} is bounded by $\Tilde{C}/N$.  

We are left with
\begin{equation}
    \begin{split}
        &\int \mathfrak{R}_{N,\varepsilon}(G_s,\eta) f\xiw d\nualph\xiw\\
    &\leq \frac{1}{\varepsilon N +1}\sum_{j_1=0}^{\varepsilon N}  \frac{1}{N^{d-1}}\sum_{x\in \Gamma_N^{-}}\sum_{\ell=0}^{j_1-1}G_s\Big(\frac{x}{N}\Big)\\
    &\times \Big[ \frac{B}{2}\int \eta_{i}(\ell e_1 + x)  \Big[
    \sqrt{f(\xiw^{x,\ell})}
    - \sqrt{f\xiw }\Big]^2 d\nualph\xiw+ \frac{2}{B} + \frac{C}{BN} + \frac{\Tilde{C}}{N}\Big]\\
    &\leq \frac{1}{\varepsilon N +1}\sum_{j_1=0}^{\varepsilon N}\frac{B}{2N^{d-1}} \|G\|_{\infty}\mathcal{D}_N(f,\nualph) + \frac{\|G\|_{\infty} 2\varepsilon N}{B} + \frac{\|G\|_{\infty}C\varepsilon }{B} + \frac{\|G\|_{\infty}\Tilde{C}}{N}\\
    &\leq \frac{B}{2N^{d-1}} \|G\|_{\infty}\mathcal{D}_N(f,\nualph) + \frac{2\|G\|_{\infty}\varepsilon N}{B} + \frac{\|G\|_{\infty}C\varepsilon }{B} + \frac{\|G\|_{\infty}\Tilde{C}}{N}.
    \end{split}
\end{equation}
This, combined with \eqref{Sup} as well as Lemma \ref{synthese} yields:
\begin{equation}\label{telescopd}
\begin{aligned}
            &\frac{1}{AN^d}\log ~ \mathbb{E}_{\nualph}\Big[\exp\Big(AN^d\int_0^t\mathfrak{R}_{N,\varepsilon}(G_s,\eta_s) ds\Big) \Big]\\
&\leq T~ \underset{f}{\sup}\Big[\Big( 
            \frac{\|G\|_{\infty} B}{2} N^{1-d}- \frac{N^{2-d}}{4A}\Big)\mathcal{D}_{N}(f,\nualph)\Big]+  T\varepsilon\|G\|_{\infty}\Big(\frac{2N}{B}+ \frac{C}{B} \Big)+ \frac{TC_4}{A}+ \frac{T\|G\|_{\infty}\Tilde{C}}{N}
\end{aligned}
\end{equation}
with $C_{4}>0$, a constant that is uniform in $N$ and $\varepsilon$.
Taking $B=\frac{N}{4A\|G\|_{\infty}}$ and putting together \eqref{Sup}, \eqref{RL1} and \eqref{telescopd} yields
\begin{equation*}
\begin{aligned}
        \mathbb{E}_{\mu_{N}}\Big[\int_0^t\mathfrak{R}_{N,\varepsilon}(G_s,\eta_s) ds \Big] \leq T\varepsilon\|G\|_{\infty}\Big(8A\|G\|_{\infty}+\frac{4A\|G\|_{\infty}C}{N}\Big) + \frac{K_0+TC_4}{A}+ \frac{T\|G\|_{\infty}\Tilde{C}}{N},
\end{aligned}        
\end{equation*}
so
\begin{equation*}
\begin{aligned}
        \underset{N \rightarrow \infty} {\limsup}~\mathbb{E}_{\mu_{N}}\Big[\Big|\int_0^t\mathfrak{R}_{N,\varepsilon}(G_s,\eta_s) ds \Big|\Big] \leq 8\varepsilon T\|G\|_{\infty}^2A + \frac{K_0+TC_4}{A},
\end{aligned}        
\end{equation*}
taking $\varepsilon \rightarrow 0$, and then $A \rightarrow \infty$, the result follows.
\end{proof}
\subsection{Energy estimates}
In view of the proof of uniqueness of the limit of the sequence of probability measures $(Q_N^{\wht})_{N\geq 1}$, we state that any limiting measure $Q^{\wht}$ is concentrated on a trajectory belonging to a specific functional space. This allows to define the hydrodynamic limit at the boundary.
\begin{prop}\label{Sobo}
Let $\wht \in (\R^+)^2$ and $Q^{\wht}$ be a limit point of the sequence of probability measures $(Q_{N}^{\wht})_{N\geq 1}$. Then, the probability measure $Q^{\wht}$ is concentrated on paths $\widehat{\rho}(t,u)du$ such that for every $1\leq i \leq 3$, $\rho_{i}$ belongs to $L^2((0,T);\mathcal{H}^1(B))$.
\end{prop}
This follows from the Lemma below and the Riesz Representation Theorem. 

\begin{lemma} \label{Esp} For any $\wht \in (\R^+)^2$, there is a constant $K_{\wht}>0$ such that for every $1\leq i \leq 3$,
\begin{equation}\label{EESP4}
    \mathbb{E}_{Q^{\wht}}\Big[ \underset{H}{\sup} ~\Big( \int_{0}^T \int_{B}\sum_{k=1}^d \partial_{e_k}H(s,u)\rho_{i}(s,u)duds - K_{\wht}\int_{0}^T\int_{B}H(s,u)^2duds\Big)\Big]<\infty, 
\end{equation}
where the supremum is carried over functions $H\in \mathcal{C}_{c}^{0,2}([0,T]\times B)$.
\end{lemma}
For the proof of Lemma \ref{Esp}, which we do not detail here, one can follow the arguments in \cite[ Lemma 7.2 Chapter 5]{kipnis_scaling_1999}. First prove \eqref{EESP4} for a dense and countable set of elements of $\mathcal{C}_{c}^{0,2}([0,T]\times B)$ thanks to Feynmann-Kac's inequality. Then, use an integration by parts to deal with the spatial derivatives in $H$, as well as the change of variable \eqref{chtva1}. To recover Proposition \ref{Sobo} from Lemma \ref{Esp} and the Riesz Representation Theorem, we also refer to \cite[Chapter 5, Theorem 7.1]{kipnis_scaling_1999}.

\subsection{Characterization of the limit point in the (Dirichlet ; Robin) mixed regime}
In order to show that the limit point of the sequence of probability measures $(Q_N^{\wht})_{N\geq 1}$ lies on the trajectory with density profile the unique solution of the hydrodynamic equation associated to $\wht$ and $\widehat{\gamma}$, we give a characterization result (see Proposition \ref{Caract}). We will focus on the (Dirichlet ; Robin) mixed regime since the (Neumann ; Robin) mixed regime can be proved following the same strategy. Therefore, take $\ttl\in [0,1)$ and $\ttr= 1$. 

As mentioned in the introduction, in one dimension, the macroscopic trajectories are continuous in space and their values at the boundaries are defined in the classical sense. This is no longer valid in higher dimension. To deal with this difficulty we use the regularity of the trajectories proved in Proposition \ref{Sobo}: the trajectories lie in $L^2([0,T],\mathcal{H}^1(B))$ so their values at the boundary are defined via the trace operator (see Lemma \ref{Trace}).

\begin{prop}\label{Caract}
If $Q^{\wht}$ is a limit point of the sequence of probability measures $(Q_{N}^{\wht})_{N\geq 1}$, then
\begin{equation}\label{caracdirichlet}
\begin{split}
    Q^{\theta}\Big[&\widehat{\pi},~ \Big|I_{\whg}(\whr)(t)\\
    &+ D\sum_{i=1}^3\int_{0}^t\Big[\int_{\Gamma^-}b_{i}(r)(\partial_{e_{1}}G_{i,s})(r)n_{1}(r).dS(r)+ \int_{\Gamma^+}\rho_{i}(s,r)(\partial_{e_{1}}G_{i,s})(r)n_{1}(r).dS(r)\Big]ds\\
    &   -\sum_{i=1}^3\int_{0}^t\int_{\Gamma^+}G_{i}(r)(b_{i}(r)-\rho_{i}(s,r))n_{1}(r).dS(r)ds\Big|=0
    ,~\forall t\in [0,T],~\forall \whg \in \mathcal{C}_{\wht}\Big]=1,
\end{split}
\end{equation}
where $I_{\whg}(\whr)$ was defined in \eqref{fonctionnelle}.
\end{prop}
\begin{proof}
    The fact that any limit point is concentrated on trajectories which are absolutely continuous with respect to the Lebesgue measure comes from Proposition \ref{Sobo}. Let $Q^{\wht}$ be a limit point of the sequence of probability measures $(Q_{N}^{\wht})_{N\geq 1}$. To prove \eqref{caracdirichlet},
it is enough to show that for any fixed $\delta >0$ and $\whg \in \mathcal{C}_{0,-}^{1,2}$,
\begin{equation*}
\begin{split}
    &Q^{\wht}\Big[\widehat{\pi},\underset{0\leq t\leq T}{\sup} \Big|I_{\whg}(\whr)(t)+ D\sum_{i=1}^3\int_{0}^t\Big[\int_{\Gamma^-}b_{i}(r)(\partial_{e_{1}}G_{i,s})(r)n_{1}(r).dS(r)\\
    &+ \int_{\Gamma^+}\rho_{i}(s,r)(\partial_{e_{1}}G_{i,s})(r)n_{1}(r).dS(r)\Big]ds\\
    &-\sum_{i=1}^3\int_{0}^t\int_{\Gamma^+}G_{i}(r)(b_{i}(r)-\rho_{i}(s,r))n_{1}(r).dS(r)ds\Big|>\delta\Big]=0.
\end{split}
\end{equation*}
Here, note that for $s\in [0,T]$ and $r\in \Gamma$, $\rho_{i}(s,r)$ stands for $\tr(\rho)(s,r)$ which is well defined since  
$\rho$ is in $L^2([0,T],\mathcal H^1(B))$. By the triangular inequality, it suffices to prove that for any $1\leq i \leq 3$,
\begin{equation}
    \begin{split}
        &Q^{\wht}\Big[\widehat{\pi},~\underset{0\leq t\leq T}{\sup} \Big| I_{G_i}(\rho_i)(t)+ D\int_{0}^t\Big[\int_{\Gamma^-}b_{i}(r)(\partial_{e_{1}}G_{i,s})(r)n_{1}(r).dS(r)\\
        &+ \int_{\Gamma^+}\rho_{i}(s,r)(\partial_{e_{1}}G_{i,s})(r)n_{1}(r).dS(r)\Big]ds-\int_{0}^t\int_{\Gamma^+}G_{i}(r)(b_{i}(r)\\
        &-\rho_{i}(s,r))n_{1}(r).dS(r)ds\Big|>\delta\Big]=0. \label{caracdirichlet1m}
    \end{split}
\end{equation}
As usual, we would like to approximate $\rho$ by a convolution of its associated empirical measure with an approximation of the identity. Indeed, that convolution product can then be written in terms of the mean value of the configuration in a microscopic box. This is straightforward in the bulk, however, for the boundary terms, we need to justify that such an approximation works (see \eqref{arg2}). Without loss of generality, let us deal with $i=1$. We turn to our martingales \eqref{Martingale} $M_{1,t}^N(\whg)$ and recall that we have proved that 
its quadratic variation vanishes as $N\uparrow \infty$. For $\varepsilon>0$, introduce the set
\[B_{N,\varepsilon}=\{-N(1-\varepsilon),\cdots,N(1-\varepsilon)\}\times \mathbb T_N^{d-1}.\]
We now use Proposition \ref{Bulk} to replace the local functions of $\e$ by functions of the particle density: 
\begin{equation}\label{Qnmar}
\begin{split}
  &M_{1,t}^N (\widehat{G}) =  \langle\pi_{1,t}^N,G_{1,t}\rangle - \langle\pi_{1,0}^N,G_{1,0}\rangle-\int_{0}^t\langle \pi_{1,s}^N,\partial_{s}G_{1,s}\rangle ds\\
   & - \int_{0}^t\frac{D}{N^d}\sum_{x\in B_{N}\setminus \Gamma_{N}}\Delta G_{1,s}\big( \frac{x}{N}\big)\e_{1,s}(x) ds\\
      &  +\frac{D}{N^{d-1}}\Big[\int_{0}^t\sum_{x\in \Gamma_{N}^-}b_1(x/N)\partial_{e_{1}}G_{1,s}\big( \frac{x}{N}\big) ds
       + \int_{0}^t\sum_{x\in \Gamma_{N}^-}\partial_{e_{1}}G_{1,s}\big( \frac{x}{N}\big)
       \big(\e_{1,s}(x)  - b_{1}(x/N)\big)ds\Big] \\
       &  -\frac{D}{N^{d-1}}\Big[\int_{0}^t\sum_{x\in \Gamma_{N}^+}\partial_{e_{1}}G_{1,s}\big( \frac{x}{N}\big)\e_{1,s}^{\varepsilon N}(x)ds
     - \int_{0}^t\sum_{x\in \Gamma_{N}^+}\partial_{e_{1}}G_{1,s}\big( \frac{x}{N}\big)\big(\e_{1,s}^{\varepsilon N}(x)-\e_{1,s}(x) \big)ds\Big]\\
       &  + \int_{0}^t\frac{D}{N^{d-1}}\sum_{x\in \Gamma_{N}^+}G_{1,s}(x/N)\big(\e_{1,s}^{\varepsilon N}(x)-
     b_1(x/N) \big)ds\\
       & - \int_{0}^t\frac{1}{N^d}\sum_{x\in B_{N,\varepsilon}}G_{1,s}\Big(\frac{x}{N}\Big)
     \Big\{2d(\lambda_1\e_{1,s}^{\varepsilon N}(x)+\lambda_2\e_{3,s}^{\varepsilon N}(x))\e_{0,s}^{\varepsilon N}(x)+\e_{3,s}^{\varepsilon N}(x)-(r+1)\e_{1,s}^{\varepsilon N}(x) \Big\}ds\\
      & +R\big(N,\varepsilon,G_1,(\eta_t)_{t\in[0,T]}\big)\, ,
   \end{split}
\end{equation}
where $R\big(N,\varepsilon,G_1,(\eta_t)_{t\in[0,T]}\big)$ is a random variable satisfying
\begin{equation*}
    \lim_{\varepsilon \to 0}\lim_{N \to \infty}
{\mathbb E}_{\mu^N} \Big[R\big(N,\varepsilon,G_1,(\eta_t)_{t\in[0,T]}\big)\Big]=0 .
\end{equation*}
From Proposition \ref{Repdirich} and Proposition \ref{lhsrep}, the martingale $M_{1,t}^N(\whg)$ can be rewritten as
\begin{equation}\label{Qnmar2}
\begin{split}
  &  M_{1,t}^N (\widehat{G}) =  \langle\pi_{1,t}^N,G_{1,t}\rangle - \langle\pi_{1,0}^N,G_{1,0}\rangle-\int_{0}^t\langle\pi_{1,s}^N,\partial_{s}G_{1,s}\rangle ds\\
   &- \int_{0}^t\frac{D}{N^d}\sum_{x\in B_{N}\setminus \Gamma_{N}}\Delta G_{1,s}\Big(\frac{x}{N}\Big)\e_{1,s}(x) ds
       +\int_{0}^t\frac{D}{N^{d-1}}\sum_{x\in \Gamma_{N}^-}b_1(x/N)\partial_{e_{1}}G_{1,s}\Big(\frac{x}{N}\Big) ds \\
       & -\int_{0}^t\frac{D}{N^{d-1}}\sum_{x\in \Gamma_{N}^+}\partial_{e_{1}}G_{1,s}\Big(\frac{x}{N}\Big)\e_{1,s}^{\varepsilon N}(x)ds
     + \int_{0}^t\frac{D}{N^{d-1}}\sum_{x\in \Gamma_{N}^+}G_{1,s}\Big(\frac{x}{N}\Big)\big(\e_{1,s}^{\varepsilon N}(x)-
     b_1(x/N) \big)ds\\
       & - \int_{0}^t\frac{1}{N^d}\sum_{x\in B_{N,\varepsilon}}G_{1,s}\Big(\frac{x}{N}\Big)
     \Big(2d(\lambda_1\e_1^{\varepsilon N}(x)+\lambda_2\e_3^{\varepsilon N}(x))\e_0^{\varepsilon N}(x)+\e_3^{\varepsilon N}(x)-(r+1)\e_1^{\varepsilon N}(x) \Big)ds\\
      & +R'\big(N,\varepsilon,G_1,(\eta_t)_{t\in[0,T]}\big)\, ,
   \end{split}
\end{equation}
where $R'\big(N,\varepsilon,G_1,(\eta_t)_{t\in[0,T]}\big)$ is a random variable satisfying
$$\lim_{\varepsilon \to 0} \lim_{N \to \infty}
{\mathbb E}_{\mu^N}\Big[R'\big(N,\varepsilon,G_1,(\eta_t)_{t\in[0,T]}\big)\Big]=0. $$
On the other hand, by \eqref{proba1} recall that 
\begin{equation*}
    \limsup_{N \rightarrow \infty}{\mathbb P}_{\mu^N} \Big[\underset{0\leq t\leq T}{\sup}~ \Big| M_{1,t}^N(\widehat{G})\Big|>\delta \Big]=0.
\end{equation*}
Now, introduce the following approximations of the identity on $B$:
\begin{equation} \label{uepsilon} 
 u_{\varepsilon}(x) =
    \frac{1}{(2\varepsilon)^{d}}\mathds{1}_{[-\varepsilon,\varepsilon]^d}(x), 
\end{equation}
\begin{equation}
   u_{\varepsilon}^{right}(x)=
    \frac{1}{\varepsilon(2\varepsilon)^{d-1}}\mathds{1}_{[0,\varepsilon]\times[-\varepsilon,\varepsilon]^{d-1}}
    (x), ~ ~~~\text{and}~ ~  ~u_{\varepsilon}^{left}(x)=
    \frac{1}{\varepsilon(2\varepsilon)^{d-1}}\mathds{1}_{[-\varepsilon,0]\times[-\varepsilon,\varepsilon]^{d-1}}
    (x).
\end{equation}
Note that for $\varepsilon>0$, $1\leq i \leq 3$, $x\in B_{N,\varepsilon}$, $y\in \Gamma_N^+$, and $z\in \Gamma_N^-$,

\begin{equation}\label{empir}
    \e^{\varepsilon N}_{i}(x)\, =\, \frac{(2\varepsilon N)^d}{(2\varepsilon N+1)^d}\big(\pi_i^N\ast u_{\varepsilon}\big)(x/N),  
\end{equation}
\begin{equation}\label{empirbord}
    \e^{\varepsilon N}_{i}(y)\, =\, \frac{(2\varepsilon N)^{d-1}}{(2\varepsilon N+1)^{d-1}}\big(\pi_i^N\ast u^{right}_{\varepsilon}\big)\Big(\frac{y}{N}\Big),~ ~ \text{and}~ ~ \e^{\varepsilon N}_{i}(z)\, =\, \frac{(2\varepsilon N)^{d-1}}{(2\varepsilon N+1)^{d-1}}\big(\pi_i^N\ast u^{left}_{\varepsilon}\big)\Big(\frac{z}{N}\Big).
\end{equation}
Here we will only make use of \eqref{empir} and the first relation in \eqref{empirbord} since we need to replace elements in the bulk and the right-hand side boundary of the system to recover the weak formulation of the equation in the (Dirichlet; Robin) regime. For regimes where a replacement is needed on the left-hand side boundary, we use the second relation in \eqref{empirbord} in the same way. 

We may thus replace in \eqref{Qnmar2} and \eqref{proba1}, $\e^{\varepsilon N}_{i}$ by $\pi_i^N\ast u_{\varepsilon}$ in the bulk and
$\e^{\varepsilon N}_{i}$ by $\pi_i^N\ast u_{\varepsilon}^{right}$ at the right boundary.
Therefore, for any $\delta>0$.
\[
\limsup_{\varepsilon \rightarrow 0} \limsup_{N \rightarrow \infty}Q_N^{\wht}\Big[ 
\underset{0\leq t\leq T}{\sup}~\Big| {\mathcal F}_{1,N,\epsilon}^{\widehat{G},t}\big( \widehat{\pi}\big)\Big| \ge \delta\Big]=0,
\]
where for any trajectory $\widehat{\pi}$ and for any $t\in[0,T]$,
\begin{equation}\label{Qnmar3}
\begin{split}
& {\mathcal F}_{1,N,\epsilon}^{\widehat{G},t}\big( \widehat{\pi}\big)=
\langle\pi_{1,t},G_{1,t}\rangle - \langle\pi_{1,0},G_{1,0}\rangle-\int_{0}^t\langle\pi_{1,s},\partial_{s}G_{1,s}\rangle ds\\
   &\quad - \int_{0}^tD \big\langle\pi_{1,s},\Delta G_{1,s}\big\rangle ds
       +\int_{0}^t\frac{D}{N^{d-1}}\sum_{x\in \Gamma_{N}^-}b_1(x/N)\partial_{e_{1}}G_{1,s}(x/N) ds \\
       & \quad  -\int_{0}^t\frac{D}{N^{d-1}}\sum_{x\in \Gamma_{N}^+}\partial_{e_{1}}G_{1,s}(x/N)
       \big(\pi_{1,s} *u_\varepsilon^{right}\big) (x)ds\\
   & \quad + \int_{0}^t\frac{D}{N^{d-1}}\sum_{x\in \Gamma_{N}^+}G_{1,s}(x/N)\big(
   \big(\pi_{1,s} *u_\varepsilon^{right}\big) (x)-
     b_1(x/N) \big)ds\\
       & \quad  - \int_{0}^t\frac{1}{N^d}\sum_{x\in B_{N,\varepsilon}}
       G_{1,s}(x/N), F_1\Big(\pi_{1,s} *u_\varepsilon (x/N),\pi_{2,s} *u_\varepsilon(x/N),
       \pi_{3,s} *u_\varepsilon(x/N)\Big)
     ds\, ,
\end{split}
\end{equation}
where functions $F_i$, $i=1,2,3$ are defined in \eqref{F}.
By approximating Lebesgue integrals by Riemann sums, on the bulk and at the boundary, we obtain
\[
\limsup_{\varepsilon \rightarrow 0} \limsup_{N \rightarrow \infty}Q_N^{\wht}\Big[ ~ 
\underset{0\leq t\leq T}{\sup}~\Big| {\mathcal F}_{1,\epsilon}^{\widehat{G},t}\big( \widehat{\pi}\big)\Big| \ge \delta\Big]=0\, ,
\]
where for any trajectory $\widehat{\pi}$ and for any $t\in[0,T]$,
\begin{equation}
\begin{split}
& {\mathcal F}_{1,\epsilon}^{\widehat{G},t}\big( \widehat{\pi}\big)=
\langle \pi_{1,t},G_{1,t}\rangle - \langle\pi_{1,0},G_{1,0}\rangle-\int_{0}^t\langle\pi_{1,s},\partial_{s}G_{1,s}\rangle ds\\
   &- D\int_{0}^t  \big<\pi_{1,s},\Delta G_{1,s}\big> ds
       +D\int_{0}^t\int_{\Gamma^-}  b_1(r)\partial_{e_{1}}G_{1,s}(r) \, dr ds \\
       & -D\int_{0}^t \int_{\Gamma^+}\partial_{e_{1}}G_{1,s}(r)
       \big(\pi_{1,s} *u_\varepsilon^{right}\big) (r) \, drds\\
       &+ \int_{0}^t \int_{\Gamma^+}G_{1,s}(r)\big(
   \big(\pi_{1,s} *u_\varepsilon^{right}\big) (r)-
     b_1(r) \big)\, drds\\
       &  - \int_{0}^t\int_{B_{\varepsilon}}
       G_{1,s}(r), F_1\Big(\pi_{1,s} *u_\varepsilon(r),\pi_{2,s} *u_\varepsilon(r),
       \pi_{3,s} *u_\varepsilon(r)\Big)\, dr ds\, ,
\end{split}
\end{equation}
with $B_{\varepsilon}=[-1+\varepsilon,1+\varepsilon]\times{\mathbb T}^{d-1}$.
By the continuity of the function 
$\widehat{\pi} \to {\mathcal F}_{1,\epsilon}^{\widehat{G},t}\big( \widehat{\pi}\big)$, for each $\varepsilon >0$,
we get for any limit point $Q^{\wht}$  of the sequence of probability measures $(Q_{N}^{\wht})_{N\geq 1}$,
\begin{equation}\label{arg0}
   \limsup_{\varepsilon \rightarrow 0} Q^{\wht}\Big[ 
\underset{0\leq t\leq T}{\sup}~\Big| {\mathcal F}_{1,\epsilon}^{\widehat{G},t}\big( \widehat{\pi}\big)\Big| \ge \delta\Big]=0\, . 
\end{equation}
To conclude the proof, it remains to prove that we may replace the convolutions appearing in the functional
${\mathcal F}_{1,\epsilon}^{\widehat{G},t}$ by the associated density of the trajectory. By Proposition \ref{Sobo}, $Q^{\wht}$ is concentrated on paths $(\widehat{\pi}(t,dr))_{t\in[0,T]}=(\widehat{\rho}(t,r)dr)_{t\in[0,T]}$ which are absolutely continuous with respect to the Lebesgue measure and such that for every $1\leq i \leq 3$, $\rho_{i}$ belongs to $L^2([0,T],\mathcal{H}^1(B))$. For the replacement of the convolution with the density in the bulk, since $u_\varepsilon$ is an approximation of the identity in $L^1(B)$ and the functions $F_i$ are Lipschitz, the random variables 
\[
\int_{0}^t\int_{B_{\varepsilon}}
       G_{1,s}(r) F_1\Big(\pi_{1,s} *u_\varepsilon(r),\pi_{2,s} *u_\varepsilon(r),
       \pi_{3,s} *u_\varepsilon (r)\Big)\, dr ds
\]
converge $Q^{\wht}$ almost surely to 
\begin{equation}\label{arg1}
    \int_{0}^t\int_{B}
       G_{1,s}(r) F_1\Big(\rho_{1,s} (r),\rho_{2,s}(r),
       \rho_{3,s}(r)\Big)\, dr ds\, .
\end{equation}
For the replacement of the convolution at the boundary we use the following result which follows from \cite[Section 5.3]{Fine}: for any $H\in {\mathcal H}^1(B)$ 
\begin{equation}\label{arg2}
 \lim_{\varepsilon\to 0}  H * u_\varepsilon^{right}
= \tr (H)~~\text{a.s in}~ ~ \Gamma^+.
\end{equation}
For the other terms in ${\mathcal F}_{1,\epsilon}^{\widehat{G},t}$, by the dominated convergence Theorem, for almost every trajectory $(\widehat{\pi}(t,dr))_{t\in[0,T]}=(\widehat{\rho}(t,r)dr)_{t\in[0,T]}$ with $\rho_{1}\in L^2([0,T],\mathcal{H}^1(B))$,
\begin{equation}\label{arg3}
 \begin{split}
  &\lim_{\varepsilon\to 0}~ 
  D\int_{0}^t \int_{\Gamma^+}\partial_{e_{1}}G_{1,s}(r)
       \big(\pi_{1,s} *u_\varepsilon^r\big) (r) \, drds
   \, -\, \int_{0}^t \int_{\Gamma^+}G_{1,s}(r)
   \Big(\pi_{1,s} *u_\varepsilon^r (r) \,-\, b_1(r) \Big)\, drds \\
&=\, D\int_{0}^t \int_{\Gamma^+}\partial_{e_{1}}G_{1,s}(r)
       \tr(\rho_{1,s})(r) \, drds
   \, -\, \int_{0}^t \int_{\Gamma^+}G_{1,s}(r)\Big(
   \tr(\rho_{1,s})(r) \, -\, b_1(r) \Big) drds.
    \end{split}
\end{equation}
Collecting \eqref{arg0}, \eqref{arg1}, \eqref{arg2} and \eqref{arg3}, we obtain \eqref{caracdirichlet1m} and conclude the proof.
\end{proof}

\subsection{Uniqueness of the limit points}
In order to finish the proof of the hydrodynamic limit specific to each regime we are left to show that each boundary valued problem \eqref{D+R} and \eqref{N+R} with fixed initial data admits a unique solution. For that, we use the standard method which consists in decomposing the difference of two solutions on the orthonormal basis of a well chosen eigenvectors of the Laplacian. The choice of the family of eigenvectors is not necessarily intuitive and depends on the boundary conditions of the mixed regime considered. We thus give details for both the (Neumann; Robin) and (Dirichlet; Robin) mixed regimes, for which the family of eigenvectors are different. As we are working in dimension $d\geq 1$, we will need to control integral terms on the boundary. Therefore, we will make use of the following result regarding the continuity of the trace operator. We refer to \cite[Part II Chapter 5]{evans_partial_2010} for a detailed survey of the trace operator.
\begin{thm}[Trace Theorem, see \cite{evans_partial_2010}\label{Trace}]
Fix $1\leq p<\infty$ and $\Omega$ an open bounded subspace of $\R^d$ with smooth boundary $\partial \Omega$. There is a constant $C_{tr}>0$ depending only on $\Omega$ and $p$ such that for any $\varphi\in \mathcal{C}^{\infty}(\overline{\Omega})$,
\[\|\varphi\|_{L^p(\partial \Omega)} \leq C_{tr}\|\varphi\|_{W^{1,p}}, \]
where $\|.\|_{L^p(\partial \Omega)} $ denotes the $L^p$ norm on $\partial \Omega$ and $\|.\|_{W^{1,p}} $ the Sobolev norm on $\Omega$ given by
\[\|\varphi\|_{W^{1,p}} = \Big(\|\varphi\|_{L^p(\Omega)}^p+\|\nabla\varphi\|_{L^p(\Omega)}^p\Big)^{1/p}, \]
where
\[\|\nabla\varphi\|_{L^p(\Omega)}^p = \sum_{i=1}^d \|\partial_{e_i}\varphi\|_{L^p(\Omega)}^p. \]
\end{thm}
\begin{rem}
    For $p=2$ and $\Omega=B$,
    \begin{equation}\label{Trace2}
        \|\varphi\|_{L^2(\partial\Omega)}^2 \leq \|\varphi\|_{L^2(\Omega)}^2+ \|\nabla \varphi\|_{L^2(\Omega)}^2 
    \end{equation}
    In particular, $C_{tr}=1$.
\end{rem} In the sequel we only make use of \eqref{Trace2} but we stated Theorem \ref{Trace} for the sake of completeness.
\subsubsection{Uniqueness of the solution in the (Neumann ; Robin) mixed regime}

Here, we choose a basis of eigenvectors satisfying Neumann conditions on both boundaries \eqref{eigenprobneum}.
\begin{thm}\label{Uniqueness1}
There exists a unique solution to the Neumann + Robin boundary problem \eqref{N+R}.
\end{thm}
\begin{proof}
By Liouville's Theorem stated for instance in \cite{evans_partial_2010}, there is a countable system $\{V_n,~ \alpha_n,~ n\geq 1\}$ of eingensolutions for the problem
\begin{equation}\label{eigenprobneum}
\left\{
    \begin{array}{ll}
         -\Delta \phi = \alpha \phi  \\
          \partial_{e_1}\phi_{|\Gamma}=0
    \end{array}
    \right.
\end{equation}
in $\mathcal{H}^1(B)$ and containing all possible eigenvalues. The set $\{V_n,~ n\geq 1\}$ forms a complete, orthonormal system in the Hilbert space $L^2(B)$ and the eigenvalues
\begin{equation} \label{ordrevp}
    0\leq \alpha_1<\alpha_2<\cdots<\alpha_n \underset{n\rightarrow \infty}{\longrightarrow} \infty
\end{equation}
have finite multiplicity. Note that for any $U,W\in \mathcal{H}^1(B)$, 
\begin{align}
    \langle U,W\rangle_2=\underset{n\rightarrow \infty}{\lim}\sum_{k=1}^n\langle U,V_k\rangle\langle W,V_k\rangle,\label{Norme2}\\
     \langle\nabla U, \nabla W\rangle_2=\underset{n\rightarrow \infty}{\lim}\sum_{k=1}^n\alpha_k \langle U,V_k\rangle \langle W,V_k\rangle.\label{Norme3}
\end{align}
One can check that since we are working on $(-1,1)\times \mathbb{T}^{d-1}$, for $k=(k_1,\cdots,k_d)\in \mathbb{N}\times(\mathbb{N}\setminus\{0\})^{d-1} $,
\[V_k(x_1,\cdots x_d) =2^{\frac{d-1}{2}}\cos\Big(\frac{k_1\pi x_1}{2}+ \frac{\pi}{2}\Big)\prod_{i=2}^{d}\sin(k_i\pi x_i)~ ~ ~ \text{and}~ ~ ~ ~ ~~\alpha_k = \frac{(k_1\pi)^2}{4}+\sum_{i=2}^d k_i^2\pi^2.\]
Furthermore, define
\begin{equation}\label{check}
    \Check{V}_k(x_2,\cdots,x_d)=2^{\frac{d-1}{2}}\prod_{i=2}^{d}\sin(k_i\pi x_i)~ ~ ~ \text{and}~ ~ ~ ~ ~~\Check{\alpha}_k = \sum_{i=2}^d k_i^2\pi^2.
\end{equation}
We have
\begin{equation}
          \|U\|_{L^2(\Gamma^+)}^2=\underset{n\rightarrow \infty}{\lim}\sum_{k=1}^n\Big(\int_{\Gamma^+}U(r)\Check{V}_k(r)n_1(r)dS(r) \Big)^2.\label{Norme4}
\end{equation}
Note that by abuse of notations we indexed the family $V_k$ by $\mathbb{N}\setminus\{0\}$ instead of $\mathbb{N}\times(\mathbb{N}\setminus\{0\})^{d-1}$ but this is not a problem because we can give an order to elements of $\mathbb{N}\times(\mathbb{N}\setminus\{0\})^{d-1}$.

Consider $\whr^1$ and $\whr^2$ two solutions of \eqref{N+R} associated to the same initial profile and for $n \in \N$ and $t>0$, introduce
\begin{equation}\label{Decompose1}
    G_n(t)= \sum_{i=1}^3\sum_{k=1}^n|\langle\rho_i^1-\rho_i^2,V_k\rangle|^2.
\end{equation}
Let us show that $\underset{n\rightarrow \infty}{\lim}G_n(t)=\|\whr^1-\whr^2\|_2^2=:G(t)=0$. For that, apply the weak formulation \eqref{N+R} with $V_k$: for any $1\leq i \leq 3$
\begin{equation}
\begin{split}
        \langle(\rho_i^1-\rho_i^2)(t,.),V_k\rangle
        &= -D\alpha_k \int_0^t\langle(\rho_i^1-\rho_i^2)(s,.),V_k\rangle ds + \int_0^t\langle(F_i(\whr^1)-F_i(\whr^2))(s,.),V_k\rangle ds\\
        &-\int_{0}^t\int_{\Gamma^+}(\rho_i^1-\rho_i^2)(s,r)V_k(r)n_1(r).dS(r)ds.
\end{split}
\end{equation}
Therefore $\langle\rho_{i}^1(t,.)-\rho_{i}^2(t,.),V_{k}\rangle  $ is time differentiable with derivative:
\begin{equation}
\begin{split}
  \partial_{t}\langle \rho_{i}^1(t,.)-\rho_{i}^2(t,.),V_{k}\rangle
  &= - D\alpha_{k}\langle\rho_{i}^1(t,.)-\rho_{i}^2(t,.), V_{k}\rangle + \langle F_{i}(\whr^1(t,.))-F_{i}(\whr^2(t,.)),V_{k}\rangle\\
  &-\int_{\Gamma^+}(\rho_i^1-\rho_i^2)(t,r)V_k(r)n_1(r).dS(r)
\end{split}
\end{equation}
and so is $G_n$, with
\begin{equation}\label{H'n}
    \begin{split}
        G'_{n}(t) &= -2D\sum_{i=1}^3\sum_{k=1}^n\alpha_{k}\big|\langle\rho_{i,t}^1-\rho_{i,t}^2,V_{k}\rangle\big|^2 + 2\sum_{i=1}^3\sum_{k=1}^n\langle F_{i}(\whr_t^1)-F_{i}(\whr_t^2),V_{k}\rangle \langle \rho_{i,t}^1-\rho_{i,t}^2,V_{k}\rangle\\
        & - 2\sum_{i=1}^3\sum_{k=1}^n\int_{\Gamma^+}(\rho_i^1-\rho_i^2)(t,r)\Check{V}_k(r)n_1(r).dS(r)\langle\rho_{i,t}^1-\rho_{i,t}^2,V_{k}\rangle\\
        &\leq -2D\sum_{i=1}^3\sum_{k=1}^n\alpha_{k}\big|\langle\rho_{i,t}^1-\rho_{i,t}^2,V_{k}\rangle\big|^2 + \sum_{i=1}^3\sum_{k=1}^n\langle F_{i}(\whr_t^1)-F_{i}(\whr_t^2),V_{k}\rangle^2+ G_{n}(t)\\
        &+ \frac{1}{A}\sum_{i=1}^3\sum_{k=1}^n\Big(\int_{\Gamma^+}(\rho_i^1-\rho_i^2)(t,r)\Check{V}_k(r)n_1(r).dS(r)\Big)^2 + A G_n(t),
    \end{split}
\end{equation}
for any $A>0$, where we used both the Cauchy-Schwarz and \eqref{useful2} inequalities in the last line. By \eqref{Norme2}, \eqref{Norme3} and \eqref{Norme4}, the right-hand side of \eqref{H'n} converges to
\begin{equation}\label{leftH'n}
    -2D\|\nabla(\whr^1-\whr^2)\|_2^2+ \sum_{i=1}^3\|F_{i}(\whr^1)-F_{i}(\whr^2) \|_{2}^2+(1+A)\|\whr^1-\whr^2\|_2^2+ \frac{1}{A}\|\whr^1-\whr^2\|_{L^2(\Gamma^+)}^2.
\end{equation}
By the trace inequality \eqref{Trace2},
\begin{equation}\label{Ttrace}
    \|\whr^1-\whr^2\|_{L^2(\Gamma)}^2\leq \|\whr^1-\whr^2\|_{2}^2 + \|\nabla(\whr^1-\whr^2)\|_2^2.
\end{equation}
Furthermore, using that $\whr^1$ and $\whr^2$ take their values in $[0,1]^3$, there is a constant $C:=C(\lambda_{1},\lambda_{2},r,d)>0$ such that for any $\whr^{a},\whr^{b} \in[0,1]^3$ and $1\leq i \leq 3$,
\begin{equation*}
    \big|F_{i}(\whr^{a})-F_{i}(\whr^b) \big|\leq C\sum_{j=1}^3|\rho_{j}^{a}-\rho_{j}^{b}|.
\end{equation*}
Then, by Cauchy-Schwarz's inequality, there is a constant $C'>0$ such that for any $1\leq i \leq 3$,
\begin{equation}\label{lipsch2}
    \|F_{i}(\whr^{a})-F_{i}(\whr^b)\|_{2}^2\leq C'\sum_{j=1}^3\|\rho_{i}^{a}-\rho_{i}^{b}\|_{2}^2 .
\end{equation}
Putting together \eqref{leftH'n}, \eqref{Ttrace}, \eqref{lipsch2}, taking $A>\frac{1}{D}$ and applying the dominated convergence theorem, we are left with
\begin{equation}
    G'(t) \leq (C'+2+A) G(t).
\end{equation}
Grönwall's inequality and the fact that $G(0)=0$ yields $G(t)=0$ at any time.
\end{proof}

\subsubsection{Uniqueness of the solution in the (Dirichlet ; Robin) mixed regime}
Here, we choose a basis of eigenvectors satisfying a Dirichlet boundary condition on the left and Neumann boundary condition on the right \eqref{eigenprobD+N}.
\begin{thm}\label{TD}
There exists a unique solution to the Dirichlet + Robin boundary problem \eqref{D+R}.
\end{thm}
\begin{proof}
     The proof follows the same lines as the previous one except that we consider another family of eigenfunctions of the Laplacian. Indeed, consider the following boundary-eigenvalue problem for the Laplacian:
\begin{equation}\label{eigenprobD+N}
\left\{
    \begin{array}{ll}
         -\Delta \phi = \gamma \phi  \\
         \phi(x)=0~ ~ \text{for}~ ~ x\in\Gamma^-\times \mathbb{T}^{d-1}\\
         \partial_{e_1}\phi(x)=0~ ~ \text{for}~ ~ x\in\Gamma^+\times \mathbb{T}^{d-1}\\
          \phi \in \mathcal{H}^{1}(B).
    \end{array}
    \right.
\end{equation}
Again, one can check that the countable system of eigensolutions $\{W_{n},~\gamma_{n},~ n\geq 1\}$ given below (in \eqref{BBAASE}) for the problem \eqref{eigenprobD+N} contains all possible eigenvalues and is a complete, orthonormal system in the Hilbert space $L^2(B)$, that the eigenvalues $\gamma_{n}$ have finite multiplicity and that
\begin{equation}\label{ordrevp2}
    0<\gamma_{1}< \gamma_{2}\cdots< \gamma_{n} \rightarrow \infty.
\end{equation}
Furthermore, \eqref{Norme2}, \eqref{Norme3} and \eqref{Norme4} stay valid when we replace $V_k$ by $W_k$, where, for $k=(k_1,\cdots,k_d)\in \mathbb{N}\times(\mathbb{N}^*)^{d-1} $,
\begin{equation}\label{BBAASE}
    W_k(x) = 2^{\frac{d-1}{2}}\Big[ (-1)^{k_1}\cos\Big(\big(\frac{\pi}{4}+ \frac{k_1\pi}{2} \big)x \Big)+ \sin\Big(\big(\frac{\pi}{4}+ \frac{k_1\pi}{2} \big)x \Big)\Big]\prod_{i=2}^{d}\sin(k_i\pi x_i) 
\end{equation}
with
\begin{equation} \label{gama}
    \gamma_k =\Big(\frac{\pi}{4}+ \frac{k_1\pi}{2} \Big)^2 + \sum_{i=2}^d k_i^2\pi^2.
\end{equation}
Again, by abuse of notation we have indexed the $W_k$'s by $\mathbb{N}^*$ instead of $(\mathbb{N}^*)^d$.

As before, take $\whr^1$ and $\whr^2$ two solutions of \eqref{D+R} with same initial data and introduce\begin{equation}\label{Hn}
    H_{n}(t) = \sum_{i=1}^3\sum_{k=1}^n\big|\langle \rho_{i}^1(t,.)-\rho_{i}^2(t,.),W_{k}\rangle\big|^2
\end{equation}
and
\begin{equation}
    H(t) = \|(\whr^1-\whr^2)(t,.)\|_2^2.
\end{equation}
Using the weak formulation \eqref{weakD+R} with $W_k$, we get that for any $1\leq i \leq 3$,
\begin{equation}
\begin{split}
        \langle(\rho_i^1-\rho_i^2)(t,.),W_k\rangle
        &= -D\gamma_k \int_0^t\langle(\rho_i^1-\rho_i^2)(s,.),W_k\rangle ds + \int_0^t\langle(F_i(\whr^1)-F_i(\whr^2))(s,.),W_k\rangle ds\\
        &-\int_{0}^t\int_{\Gamma^+}(\rho_i^1-\rho_i^2)(s,r)\Check{W}_k(r)n_1(r).dS(r)ds,
\end{split}
\end{equation}
where the $\Check{W}_k=\Check{V}_k$ are defined in \eqref{check}.
Then, we conclude following exactly the same lines as the proof of Theorem \ref{Uniqueness1}.
\end{proof}

\subsubsection{Uniqueness of the solution in the other regimes}
In order to prove uniqueness in the other regimes, one can follow the same classic method used above. The orthonormal basis used to decompose the difference of two solutions as in \eqref{Decompose1} or \eqref{Hn} then depends on the boundary conditions. For the (Dirichlet ; Dirichlet) regime, the decomposition is carried out on the eigenvectors of the following boundary-eigenvalue problem for the Laplacian:
\begin{equation}\label{eigenprobDirich} 
\left\{
    \begin{array}{ll}
         -\Delta \phi = \delta \phi  \\
          \phi \in \mathcal{H}_0^{1}(B),
    \end{array}
    \right.
\end{equation}
for which the associated family of eigenvectors is
\[U_k(x_1,\cdots x_d) = 2^{\frac{d-1}{2}}\prod_{i=1}^d\sin(k_i\pi x_i),\]
with eigenvalues given by
\[\delta_k= \sum_{i=1}^d k_i^2\pi^2\]
for $k=(k_1,\cdots,k_d)\in (\mathbb{N}^*)^d$. As before, for $V,W\in L^2(B)$,
\begin{align}
    \langle V,W\rangle _{2}=\underset{n\rightarrow \infty}{\lim}\sum_{k=1}^n\langle V,U_{k}\rangle_{2}\langle W,U_{k}\rangle_{2},\label{PSn}\\
    \langle \nabla V, \nabla W\rangle_2=\underset{n\rightarrow \infty}{\lim}\sum_{k=1}^n\delta_k \langle V,U_k\rangle \langle W,U_k\rangle, \label{norme2dirich}\\
    \|V\|_{L^2(\Gamma)}^2=\underset{n\rightarrow \infty}{\lim}\sum_{k=1}^n\Big(\int_{\Gamma}V(r)\Check{U}_k(r)n_1(r)dS(r) \Big)^2\label{norme3dirich}
\end{align}
where the $\Check{U}_k=\Check{V}_k$ are defined in \eqref{check}.
\section{Hydrostatic limit}
In this section, we prove Theorem \ref{T hydrostat} which states that when the parameters  $r,\lambda_1,\lambda_2,d, D$ satisfy certain conditions, starting from an invariant measure, the system converges to the stationary profile of the corresponding hydrodynamic equation. Precisely, recall that in Section 2, for $\widehat{\theta}\in (\R^+)^2$  we defined $\mu_{N}^{ss}(\widehat{\theta})$ as the sequence of unique invariant measures for the irreducible dynamics defined by \eqref{Generator}. The hydrostatic principle states that this sequence is associated to the unique stationary solution of the hydrodynamic equation, if existence and uniqueness of such a solution hold. For the proof, we were inspired by \cite{farfan_hydrostatics_2011} and the key argument relies on the convergence of all the trajectories satisfying the hydrodynamic equation to the unique stationary profile of these equations. In \cite{farfan_hydrostatics_2011}, the convergence of trajectories is established thanks to a comparison principle. The difficulty here is that we are dealing with a system of coupled equations and we need to define a specific order for which such a comparison principle holds. Now in \cite[Theorem 4.1]{kuoch:hal-01100145}, it has been proved that at the microscopic level, the generalized contact process is attractive only for the following order:
\begin{equation}\label{order}
    2<0<3<1. 
\end{equation}
Note that in the corresponding state space $\whsN$, the order above translates into
\begin{equation*}
    (0,1)<(0,0)<(1,1)<(1,0).
\end{equation*}
Attractiveness for the order \eqref{order} means that given two configurations $\e \leq \overset{\sim}{\e}$, it is possible to build a coupling between $(\eta_t)_{t\geq 0}$ and $(\overset{\sim}{\e}_t)_{t\geq 0}$ where both these processes evolve according to the dynamics given by \eqref{Generator}, such that $ \e_0\leq  \overset{\sim}{\e}_0$  and almost surely, for all $t\geq 0$, $\e_t \leq \overset{\sim}{\e}_t$ pointwise in the sense of \eqref{order}. Note that using \cite[Theorem 2.4]{Borrello}, one can show that the system remains attractive when adding an exchange and reservoir dynamics. It is then natural to think that attractiveness also holds at the macroscopic level through a comparison principle. A comparison principle means that if two profiles are such that at a certain time, one is smaller than the other almost everywhere, then the same is true at any later time. Considering the microscopic order \eqref{order} it is intuitive to consider that the largest state at the macroscopic level corresponds to $(\rho_1=1,\rho_2=0,\rho_3=0)$ and the smallest state to  $(\rho_2=1, \rho_1=\rho_3=0)$. We will work under the following change of coordinates:
\begin{equation} \label{Changecoord}
    \left\{
    \begin{array}{ll}
        \rho_1\\ 
        T:=\rho_1+\rho_3\\
        R:=1-(\rho_2+\rho_3)
    \end{array}
\right.
\end{equation}
which is consistent with the fact that $(1,1,1)$ corresponds to the largest profile $(\rho_1=1,\rho_2=0,\rho_3=0)$ and $(0,0,0)$ with the lowest one $(\rho_2=1, \rho_1=\rho_3=0)$. In the sequel, we will say that given two profiles $\whr$ and $\widehat{\phi}$, $\whr\leq \widehat{\phi} $ if: \begin{equation}
        \left\{
    \begin{array}{ll}
        \rho_1\leq \phi_1\\ 
        \rho_1+\rho_3\leq \phi_1+\phi_3\\
        1-(\rho_2+\rho_3)\leq 1-(\phi_2+\phi_3) 
    \end{array}
\right.
\end{equation}
almost everywhere. Note that this new order adapted at the microscopic level, i.e
\begin{equation}\label{Newworder}
    \Tilde{\e} \leq \e ~ \Leftrightarrow~ \forall x\in B_N,~  \left\{
    \begin{array}{ll}
        \Tilde{\e}_1(x) \leq \e_1(x)\\ 
        \Tilde{\e}_1(x) + \Tilde{\e}_3(x) \leq \e_1(x) + \e_3(x)\\
        \e_2(x) + \e_3(x) \leq \Tilde{\e}_2(x)+\Tilde{\e}_3(x),
    \end{array}
\right.
\end{equation}
is not equivalent to the one given in \eqref{order}. Indeed, consider the configuration $\e$ full of $3$'s and $\Tilde{\e}$ full of $0$'s. Then $\Tilde{\e} \leq \e$ for the order \eqref{order} but not for the order \eqref{Newworder}. However, one can check that if $\Tilde{\e} \leq \e$ for the order \eqref{Newworder}, then $\Tilde{\e} \leq \e$ for the order \eqref{order}, so the macroscopic order is consistent with the microscopic one but it is weaker, so we can compare fewer profiles.

However, the notable fact, which we will prove, is that we have monotonicity under this new order, i.e. a comparison principle holds under the change of coordinates \eqref{Changecoord}. To prove that, as previously, since we are working in any dimension $d\geq 1$ with mixed boundary conditions, some care must be taken to deal with the integral terms on $\Gamma$. For that, we strongly rely on analytical tools stated in \cite{Roubi}.

Under the change of coordinates \eqref{Changecoord}, the coupled equations in the bulk become, :
\begin{equation} \label{Dirichletmodifie}
    \left\{
    \begin{array}{ll}
        \partial_t\rho_1=D\Delta \rho_1 + F_1(\rho_1,T,R)\\ 
        \partial_t T= D\Delta T + H(\rho_1,T,R)\\
        \partial_t R= D\Delta R + J(R)
    \end{array}
\right.
\end{equation}
with
\begin{equation} 
    \left\{
    \begin{array}{ll}
         F_1(\rho_1,T,R) = 2d\big[(\lambda_1-\lambda_2)\rho_1 + \lambda_2T \big](R-\rho_1) + T - (r+2)\rho_1\\ 
        H(\rho_1,T,R) = 2d\big[(\lambda_1-\lambda_2)\rho_1 + \lambda_2T\big](1-T)-T\\
        J(R)= -(r+1)R+1.
    \end{array}
\right.
\end{equation}
We will see that the comparison principle stated and proved in Lemma \ref{L Att} yields the following Theorem which is used to prove Theorem \ref{T hydrostat}.

\begin{thm} \label{T3}Suppose that conditions $(H_1)$ hold. Then, there exists a unique stationary solution $\overline{\rho}^{D,R}$, resp. $\overline{\rho}^{N_e,R}$, of \eqref{D+R}, resp. \eqref{N+R}. Furthermore, for any solution $\whr^{D,R}$, resp. $\whr^{N_e,R}$, to the boundary value problem \eqref{D+R}, resp. \eqref{N+R},
\begin{equation}\label{P1}
    \underset{t \rightarrow \infty}{\lim}~ \sum_{i=1}^3\|\rho_{i}^{D,R}(t,.) - \overline{\rho}^{D,R}_{i}(.)\|_{1} = 0,
\end{equation}
resp.
\begin{equation}
    \underset{t \rightarrow \infty}{\lim}~ \sum_{i=1}^3\|\rho_{i}^{N_e,R}(t,.) - \overline{\rho}^{N_e,R}_{i}(.)\|_{1} = 0.
\end{equation}
\end{thm}
Note that this result can be equivalently formulated in the change of coordinates \eqref{Changecoord} and we will prove it in that setting in the next subsection.

\begin{rem} \label{Necessary}
One could ask if conditions on the parameters are necessary to establish existence and uniqueness of the stationary solution of the hydrodynamic equation. Could we not generalize the result to all parameters? In order to answer that, we simulated the solutions to the equation in the (Neumann ; Neumann) regime for which the constant profile $(\rho_1=0, \rho_2=\frac{r}{r+1},\rho_3=0)$ is stationary. Indeed,
\[F_{1}\Big(0,\frac{r}{r+1},0\Big) = F_{2}\Big(0,\frac{r}{r+1},0\Big)=F_{3}\Big(0,\frac{r}{r+1},0\Big)=0\]
 and it corresponds to the extinction regime, that is, there are no more wild insects. We observed (see below in the Appendix \ref{Simuu}) that in dimensions $1$, for parameters $\lambda_1=1$, $\lambda_2 = 0.75$ and $D=r=1$, for which conditions $(H_1)$ are not satisfied, the solution of the hydrodynamic equation starting from $\rho_1=1, \rho_2=\rho_3=0$ converges to a constant profile which is not $(0,\frac{r}{r+1},0)$ so uniqueness does not hold. Simulations confirm that Theorem \ref{T3} does not hold in all generality and that conditions on the parameters are necessary, although conditions $(H_1)$ might not be the optimal ones.
\end{rem}

\subsection{Proof of the hydrostatic limit}
Let us prove Theorem \ref{T hydrostat}. We prove the first point, the second one follows in the same way. Denote by $\mathcal{A}_{T} \subset D([0,T],\big(\mathcal{M}^+\big)^3)$ the set of trajectories $\{\widehat{\rho}(t,u)du,~ 0 \leq t \leq T \}$ whose density $\widehat{\rho}=(\rho_{1},\rho_{2},\rho_{3})$ satisfies conditions \eqref{conditiona} and \eqref{CI} of the definition of a weak solution of \eqref{D+R} for some initial profile $\widehat{\rho}_{0}$. Consider $Q^*_{ss}(\widehat{\theta})$ a limit point of the sequence $(Q_{\mu_{N}^{ss}(\widehat{\theta})}^N)_{N\geq 1}$ associated to the invariant measures. By Theorem \ref{HL}, 
 \begin{equation}
    Q^*_{ss}(\widehat{\theta})\big(\mathcal{A}_{T} \big)=1. 
 \end{equation}
Now consider $Q_{ss}^{N_{k}}(\widehat{\theta})$ a subconverging sequence of $(Q_{\mu_{N}^{ss}}^N(\widehat{\theta}))_{N\geq 1}$. By stationarity of $\mu_{N}^{ss}(\widehat{\theta})$
\begin{equation}
\begin{split}
    \mathbb{E}_{Q_{ss}^{N_{k}}( \widehat{\theta})}\Big(~ \Big|\langle\widehat{\pi}^N, \whg\rangle - \langle\overline{\rho},\whg\rangle \Big| ~ \Big) &= \mathbb{E}_{Q_{ss}^{N_{k}}(\widehat{\theta})}\Big(~\Big|\langle\widehat{\pi}_{T}^N, \whg\rangle - \langle\overline{\rho},\whg\rangle\Big|~\Big)
\end{split}
\end{equation}
and
\begin{equation}\label{Debproofhydro}
\begin{split}
    \underset{k \rightarrow \infty}{\lim}~ \mathbb{E}_{Q_{ss}^{N_{k}}(\widehat{\theta})}\Big(~\Big|\langle\widehat{\pi}_{T}^N, \whg\rangle - \langle\overline{\rho},\whg\rangle \Big|~\Big) &= \mathbb{E}_{Q_{ss}^*(\widehat{\theta})}\Big(~\Big|\langle\widehat{\pi}_{T},\widehat{G}\rangle - \langle\overline{\rho},\whg\rangle \Big|\mathds{1}_{\mathcal{A}_{T}} \Big)\\
    & \leq 
    \sum_{i=1}^3 \|G_{i}\|_{\infty}~ \mathbb{E}_{Q_{ss}^*}
\Big[ \| \rho_{i}(T,.) - \overline{\rho}_{i}(.) \|_{1} \Big] .
\end{split}
\end{equation}
Then, one concludes thanks to \eqref{P1} in Theorem \ref{T3} and dominated convergence theorem.

\subsection{Proof of Theorem \ref{T3}}
In order to prove Theorem \ref{T3} we first establish a comparison principle (Lemma \ref{L Att}). Then, we show that the difference between the largest solution and the smallest solution vanishes (Lemma \ref{L2}).
Using an integration by parts, it is useful to rewrite  
 the weak formulations \eqref{weakD+R} and \eqref{weakD+R}, in the following suitable forms: for any $0\le \tau\le t\le T$, for any 
 $G\in {\mathcal C}^2([0,T]\times B)$,
\begin{equation}\label{weakD+R-2}
\begin{split}
    \langle\whr_t,\whg_t\rangle -\langle\whr_\tau,\whg_\tau\rangle  = &\int_{\tau}^t\langle\whr_s,\partial_{s}\whg_s\rangle ds - D\int_{\tau}^t\int_B \big(\nabla \whr_s \cdot\, \nabla \whg_s\big) (r)dr ds \\
\ &  -\int_{\tau}^t\langle\widehat{F}(\whr_{s}),\whg_s\rangle ds
  - D\sum_{i=1}^3\int_{\tau}^t\int_{\Gamma^-}b_{i}(r)(\partial_{e_{1}}G_{i,s})(r)n_{1}(r).dS(r)ds\\
\ &   +\sum_{i=1}^3\int_{\tau}^t\int_{\Gamma^+}G_{i}(r)(b_{i}(r)-\rho_{i}(s,r))n_{1}(r).dS(r)ds=0,
\end{split}
\end{equation}
and
\begin{equation}\label{WeakN+R-2}
    \begin{split}
        \langle \whr_t,\whg_t\rangle -\langle\whr_\tau,\whg_\tau\rangle  = &\int_{\tau}^t\langle\whr_s,\partial_{s}\whg_s\rangle ds - D\int_{\tau}^t\int_B \big(\nabla \whr_s \cdot\, \nabla \whg_s\big) (r)dr ds \\
 \ & -\int_{\tau}^t\langle \widehat{F}(\whr_{s}),\whg_s\rangle ds
 -\int_{\Gamma^+}G_{i}(r)(b_{i}(r)-\rho_{i}(s,r))n_{1}(r).dS(r)ds=0.
    \end{split}
\end{equation}

\begin{lemma}\label{L Att} Consider $\whr_{0}^1$ and $\whr_{0}^2$ two initial profiles.
\begin{itemize}
    \item [•]Denote by $\whr_{t}^1$ resp. $\whr_{t}^2$, the solutions to the (Dirichlet ; Robin) boundary problem \eqref{D+R} associated to each of those initial profiles. Assume that there is an $s\geq 0$ such that almost surely (in the Lebesgue measure sense), $\rho_1^1(s,u)\leq \rho_1^2(s,u)$, $T^1(s,u)\leq T^2(s,u)$ and $R^1(s,u)\leq R^2(s,u)$. Then, for all $s\geq t$, $\rho_1^1(t,u)\leq \rho_1^2(t,u)$, $T^1(t,u)\leq T^2(t,u)$ and $R^1(t,u)\leq R^2(t,u)$ almost surely.
    \item [•] The same result holds when $\whr_{t}^1$ resp. $\whr_{t}^2$, are two solutions to the (Neumann ; Robin) boundary problem \eqref{N+R}.
\end{itemize}

\end{lemma}
Note that Lemma \ref{L Att} holds for all parameters $r,\lambda_1,\lambda_2,d$ and $D$, regardless of conditions $(H_1)$.
\begin{proof} We prove the first point and the proof of the second one follows in the same way. Introduce
\begin{equation}
\begin{split}
    A(t)  &= \int_B \big(\rho_1^1-\rho_1^2 \big)_+^2(t,u){\mathrm{d}u}+ \int_B \big(T^1-T^2 \big)_+^2(t,u){\mathrm{d}u} + \int_B \big(R^1-R^2 \big)_+^2(t,u){\mathrm{d}u}
 \end{split}
\end{equation}
where $x_+$ denotes $\max(x,0)$, the positive part of $x$. We show that $A(t)=0$ for all $t\geq s$. 
Using the weak formulation \eqref{weakD+R} of the solution of the (Dirichlet ; Robin) boundary problem and using Lemma 7.3 and Remark 7.5 in \cite{Roubi}, we get:
\begin{equation}
    \begin{split}
        \frac12\frac{d}{dt}\int_B \big(\rho_1^1-\rho_1^2 \big)_+^2&(t,u){\mathrm{d}u}=\frac12\frac{d}{dt}\int_B \big(\rho_1^1-\rho_1^2 \big)_+^2(t,u){\mathrm{d}u}\\
        &= - D\int_B \nabla(\rho_1^1-\rho_1^2) \nabla\big(\rho_1^1-\rho_1^2 \big)_+(t,u){\mathrm{d}u} \\
        &+ \int_B\big(F_1(\whr^1)-F_2(\whr^2)\big)\big(\rho_1^1-\rho_1^2 \big)_+ \big)(t,u){\mathrm{d}u} - \int_{\Gamma^+}(\rho_1^1-\rho_1^2)_+^2(t,u) {\mathrm{d}u}.
    \end{split}
\end{equation}
Using that $\nabla\big((\rho_1^1-\rho_1^2 )_+\big) = \mathds{1}_{(\rho_1^1-\rho_1^2 )\geq 0} \nabla\big(\rho_1^1-\rho_1^2 \big) $ and that $\int_B \big(\rho_1^1-\rho_1^2 \big)_+^2(0,u){\mathrm{d}u}=0 $,
we are left with:
\begin{equation}
    \begin{split}
        \frac{1}{2}\int_B \big(\rho_1^1-\rho_1^2 \big)_+^2(t,u){\mathrm{d}u}&\leq -\int_0^t\int_BD\mathds{1}_{(\rho_1^1-\rho_1^2 )\geq 0} \|\nabla\big(\rho_1^1-\rho_1^2 \big)\|_2^2(r,u){\mathrm{d}u}{\mathrm{d}r}\\
        &+ \int_0^t\int_B\big(F_1(\whr^1)-F_2(\whr^2)\big)\big((\rho_1^1-\rho_1^2 \big)_+ \big)(r,u){\mathrm{d}u}{\mathrm{d}r}.
    \end{split}
\end{equation}
Proceeding in the same way for $\int_B \big(T^1-T^2 \big)_+^2(t,u){\mathrm{d}u}  $ and $\int_B \big(R^1-R^2 \big)_+^2(t,u){\mathrm{d}u} $ we get:
\begin{equation}
\begin{split}
        \frac{1}{2}\int_B \big(T^1-T^2 \big)_+^2(t,u){\mathrm{d}u} &\leq -\int_0^t\int_BD\mathds{1}_{(T^1-T^2 )\geq 0} \|\nabla\big(T^1-T^2 \big)\|_2^2(r,u){\mathrm{d}u}{\mathrm{d}r}\\
        &+ \int_0^t\int_B\big(H(\whr^1)-H(\whr^2)\big)(T^1-T^2 \big)_+(r,u){\mathrm{d}u}{\mathrm{d}r}
\end{split}
\end{equation}
and
\begin{equation}
\begin{split}
        \frac{1}{2}\int_B \big(R^1-R^2 \big)_+^2(t,u)du &\leq-\int_0^t\int_BD\mathds{1}_{(R^1-R^2 )\geq 0} \|\nabla\big(R^1-R^2 \big)\|_2^2(r,u){\mathrm{d}u}{\mathrm{d}r}\\
        &+ \int_0^t\int_B\big(J(R^1)-J(R^2)\big)(R^1-R^2 \big)_+(r,u){\mathrm{d}u}{\mathrm{d}r}.
        \end{split}
\end{equation}
Therefore,
\begin{equation}\label{A(t)}
    \begin{split}
        \frac{1}{2}A(t)&\leq \int_0^t\int_B\big(F_1(\whr^1)-F_2(\whr^2)\big)\big(\rho_1^1-\rho_1^2 \big)_+ (r,u){\mathrm{d}u}{\mathrm{d}r}\\
        &+ \int_0^t\int_B\big(H(\whr^1)-H(\whr^2)\big)(T^1-T^2 \big)_+(r,u){\mathrm{d}u}{\mathrm{d}r}\\
        &+  \int_0^t\int_B\big(J(R^1)-J(R^2)\big)(R^1-R^2 \big)_+(r,u){\mathrm{d}u}{\mathrm{d}r}.
    \end{split}
\end{equation}
Now let us use the explicit expressions of $F_1$, $J$ and $H$. We also use the following inequality: for any $C\geq 0$ $x,y\in \R$,
\begin{equation}\label{partiepos}
    Cxy_+\leq Cx_+ y_+.
\end{equation}
In order to avoid confusions, a squared term will always be put between brackets, while, for instance $\rho_1^2$ refers to the first coordinate of $\whr^2$. We will denote by $C$ a positive constant which depends on $\lambda_1,\lambda_2, r, d$ with values possibly changing from one line to the next.
\begin{equation}\label{F1-F2}
\begin{split}
       \big(F_1(\whr^1)-F_2(\whr^2)\big)\big(\rho_1^1-\rho_1^2 \big)_+ 
       &= \big[\ddd (R^2-\rho_1^1-\rho_1^2) -2d\lambda_1T^1-(r+2) \big](\rho_1^1-\rho_1^2)_+^2\\
       &+\big[ 2d\lambda_1 \rho_1^1 + 2d\lambda_2 \rho_3^1 \big](R^1-R^2)(\rho_1^1-\rho_1^2)_+ \\
       &+ \big[ 1+2d\lambda_2(1-\rho_2^2-\rho_3^3-\rho_1^2)\big](T^1-T^2)(\rho_1^1-\rho_1^2)_+\\
       &\leq C(\rho_1^1-\rho_1^2)_+^2 + \big[ 2d\lambda_1 \rho_1^1 + 2d\lambda_2 \rho_3^1 \big](R^1-R^2)_+(\rho_1^1-\rho_1^2)_+ \\
       &+ \big[ 1+2d\lambda_2(1-\rho_2^2-\rho_3^3-\rho_1^2)\big](T^1-T^2)_+(\rho_1^1-\rho_1^2)_+,
\end{split}
\end{equation}
where we used \eqref{partiepos} and the fact that $ 2d\lambda_1 \rho_1^1 + 2d\lambda_2 \rho_3^1\geq 0 $ and $ 1+2d\lambda_2(1-\rho_2^2-\rho_3^3-\rho_1^2)\geq 0 $ in the last line.
\begin{equation}\label{H1-H2}
\begin{split}
        &\big(H(\whr^1)-H(\whr^2)\big)(T^1-T^2)_+ = \big[2d\lambda_2-\ddd \rho_1^1-2d\lambda_2(T^2+T^1)-1 \big](T^1-T^2)_+^2\\
        & + \big[\ddd(1-\rho_1^2-\rho_3^2) \big](\rho_1^1-\rho_1^2)(T^1-T^2)_+\\
        &\leq C(T^1-T^2)_+^2 + \big[\ddd(1-\rho_1^2-\rho_3^2) \big](\rho_1^1-\rho_1^2)_+(T^1-T^2)_+
\end{split}
\end{equation}
where again, we used \eqref{partiepos} in the last line, the fact that $\lambda_1\geq \lambda_2$ and that $(1-\rho_1^2-\rho_3^2)\geq 0$. Finally,
\begin{equation}\label{J1-J2}
   \big(J(R^1)-J(R^2)\big)(R^1-R^2 \big)_+ = -(r+1)(R^1-R^2 \big)_+^2. 
\end{equation}
Collecting \eqref{F1-F2}, \eqref{H1-H2} and \eqref{J1-J2} we are left with
\begin{equation}\label{Computperf}
\begin{split}
    \frac{1}{2}A(t)  
       \leq C\int_0^t\int_B &\Big( \big(\rho_1^1-\rho_1^2 \big)_+^2(r,u)+ \big(T^1-T^2 \big)_+^2(r,u)+ \big(R^1-R^2 \big)_+^2(r,u)\Big){\mathrm{d}u}{\mathrm{d}r} \\
       &= C \int_0^t A(r){\mathrm{d}r},
\end{split}
\end{equation}
where $C$ is a constant which depends on $\lambda_1,\lambda_2, r, d$ and by Grönwall's lemma, $A(t)=0$.
\end{proof}

\begin{corollary}\label{C1}
Denote by $\whr^0=(\rho_1^0,T^0,R^0)$, resp. $\whr^1=(\rho_1^1,T^1,R^1)$, the weak solution of \eqref{Dirichletmodifie} with (Dirichlet ; Robin) boundary conditions and initial data $\rho_1^0=T^0=R^0=0$, resp. $\rho_1^1=T^1=R^1=1$. Then for every $t\geq s$, $\rho_{1}^0(s,.)\leq \rho_{1}^0(t,.)$, $T^0(s,.)\leq T^0(t,.)$ and $R^0(s,.)\leq R^0(t,.)$, resp. $\rho_{1}^1(s,.)\geq \rho_{1}^1(t,.)$, $T^1(s,.)\geq T^1(t,.) $ and $R^1(s,.)\geq R^1(t,.) $ almost surely. Furthermore, any other solution $(\rho_1,T,R)$ of \eqref{Dirichletmodifie} with (Dirichlet ; Robin) boundary conditions satisfies: $\rho_1^0\leq \rho_1\leq \rho_1^1$, $T^0\leq T\leq T^1$ and $R^0\leq R \leq R^1$ almost surely.

The same result holds for $\whr^0=(\rho_1^0,T^0,R^0)$, resp. $\whr^1=(\rho_1^1,T^1,R^1)$, the weak solution of \eqref{Dirichletmodifie} with (Neumann ; Robin) boundary conditions and initial data $\rho_1^0=T^0=R^0=0$, resp. $\rho_1^1=T^1=R^1=1$. 
\end{corollary}
\begin{proof}
We prove the result for the (Dirichlet ; Robin) boundary problem and for $\whr^0$. The proof is the same for $\whr^1$ and for the (Neumann ; Robin) case. Fix $s\geq 0$ and consider $\tau_{s} \whr^0:(t,u)\mapsto \whr^0(t+s,u)$. $\tau_{s} \whr^0$ is the solution of \eqref{Dirichletmodifie} with initial condition $u\mapsto \whr^0(s,u)$ and almost surely in $B$, $\tau_{s} \whr^0(0,u) \geq (0,0,0)=\whr^0(0,u)$. Applying Lemma \ref{L Att} to $\tau_{s} \whr^0 $ and $ \whr^0$ with $s=0$ and $t=t-s$ yields
$\whr^0(t,u) \geq \whr^0(s,u)$ almost surely.
\end{proof}
\begin{lemma}\label{L2} Assume conditions $(H_1)$ are satisfied. 
\begin{itemize}
    \item [•]Denote by $\whr^0=(\rho_1^0,T^0,R^0)$, resp. $\whr^1=(\rho_1^1,T^1,R^1)$, the weak solution of \eqref{Dirichletmodifie} with (Dirichlet ; Robin) boundary conditions and with initial data $(0,0,0)$, resp. $(1,1,1)$. Then,
\begin{equation}\label{AttracDR}
    \underset{t \rightarrow \infty}{\lim} \sum_{i=1}^3 \int_{B}\Big(\big| \rho_{1}^1(t,u) - \rho_{1}^0(t,u)|+| T^1(t,u) - T^0(t,u)|+| R^1(t,u) - R^0(t,u)|\Big)du= 0.
\end{equation}
\item[•]Denote by $\whr^0=(\rho_1^0,T^0,R^0)$, resp. $\whr^1=(\rho_1^1,T^1,R^1)$, the weak solution of \eqref{Dirichletmodifie} with (Neumann ; Robin) boundary conditions and with initial data $(0,0,0)$, resp. $(1,1,1)$. Then,
\begin{equation}\label{AttracNR}
    \underset{t \rightarrow \infty}{\lim} \sum_{i=1}^3 \int_{B}\Big(\big| \rho_{1}^1(t,u) - \rho_{1}^0(t,u)|+| T^1(t,u) - T^0(t,u)|+| R^1(t,u) - R^0(t,u)|\Big)du= 0.
\end{equation}
\end{itemize}

\end{lemma}
\begin{proof}
We start with the proof of the (Dirichlet ; Robin) regime. It is enough to show that 
\begin{equation}\label{sum}
    \underset{t\rightarrow\infty}{\lim}~ \Big( \|\rho_{1}^1(t,.)- \rho_{1}^0(t,.)\|_{2}^2+\|T^1(t,.)- T^0(t,.)\|_{2}^2+\|R^1(t,.)- R^0(t,.)\|_{2}^2\Big)=0.
\end{equation}
Consider the eigenvalue problem for the Laplacian \eqref{eigenprobD+N} and the countable system $\{W_{n},\gamma_{n},~ n\geq 1\}$ of eigensolutions for that problem. For $n\geq 1$ introduce 
\begin{equation}\label{Gnt}
\begin{split}
    K_{n}(t) =A_n(t)+ B_n(t)+&C_n(t):=\sum_{k=1}^n  |\langle R^1(t,.)- R^0(t,.),W_{k}\rangle|^2\\
    & + \sum_{k=1}^n  |\langle\rho_1^1(t,.)- \rho_1^0(t,.),W_{k}\rangle|^2+ \sum_{k=1}^n  |\langle T^1(t,.)- T^0(t,.),W_{k}\rangle|^2.
\end{split}
\end{equation}
Recall that by \eqref{Norme2}, one has
\[A_n(t) \underset{n\rightarrow \infty}{\longrightarrow}\|R^1(t,.)- R^0(t,.)\|_{2}^2=:A(t),~~ ~   B_n(t) \underset{n\rightarrow \infty}{\longrightarrow}\|\rho_1^1(t,.)- \rho_1^0(t,.)\|_{2}^2=:B(t),\]
and
\[C_n(t) \underset{n\rightarrow \infty}{\longrightarrow}\|T^1(t,.)- T^0(t,.)\|_{2}^2=:C(t).\]
Let us first prove that $ \underset{t\rightarrow \infty}{\lim}~\underset{n\rightarrow \infty}{\lim}A_n(t) =0$. $A_n$ is time differentiable and the weak formulation of a solution of \eqref{Dirichletmodifie} with (Dirichlet ; Robin) boundary conditions yields,
\begin{equation}\label{derivv}
\begin{split}
   A'_n(t) &=-2\sum_{k=1}^n(D\gamma_k+r+1) |\langle R_{t}^1-R_{t}^0,W_{k}\rangle|^2\\
   &- 2 \sum_{k=1}^n\langle R_{t}^1-R_{t}^0,W_{k}\rangle\int_{\Gamma_+}(R_t^1-R_t^0)(r)W_k(r)n_1(r)dS(r).
\end{split}
\end{equation}
Integrating this between $0$ and $T$ and using the Cauchy-Schwarz inequality twice yields
\begin{equation*}
\begin{split}
        A_{n}(0)-A_{n}(T) \geq \int_{0}^T\sum_{k=1}^n2(D\gamma_k + r+1)\big|\langle R_{t}^1-R_{t}^0,W_{k}\rangle\big|^2{\mathrm{d}t} \\
        - 2\sqrt{\int_{0}^T\sum_{k=1}^n \big|\langle R_{t}^1-R_{t}^0,W_{k}\rangle\big|^2{\mathrm{d}t} }\sqrt{\int_{0}^T\sum_{k=1}^n \Big(\int_{\Gamma_+}(R_t^1-R_t^0)(r)\Check{W}_k(r)n_1(r)dS(r)\Big)^2 {\mathrm{d}t}}. 
\end{split}
\end{equation*}
Taking $n \rightarrow \infty$ and using \eqref{Norme3} and \eqref{Norme4} using the $W_k's$ and $\Check{W}_k's $ instead of the $V_k's$ and $\Check{V}_k's $ we get
\begin{equation*}
\begin{split}
        A(0)&\geq 2(r+1) \int_0^TA(t)dt + 2D \int_0^T\Tilde{A}(t)dt-2\sqrt{\int_0^TA(t)dt}\sqrt{\int_0^T\|R_t^1-R_t^0\|_{L^2(\Gamma)}^2dt}\\
        &\geq 2(r+1) \int_0^TA(t)dt  + 2D \int_0^T\Tilde{A}(t)dt-2\sqrt{\int_0^TA(t)dt}\sqrt{\int_0^TA(t)dt+\int_0^T\Tilde{A}(t)dt}\\
        &\geq 2(r+1) \int_0^TA(t)dt + 2D \int_0^T\Tilde{A}(t)dt-2\Big(\int_0^TA(t)dt+\int_0^T\Tilde{A}(t)dt\Big)\\
        &\geq 2r \int_0^TA(t)dt+ 2(D-1)\int_0^T\Tilde{A}(t)dt
\end{split}
\end{equation*}
where $\Tilde{A}(t)=\|\nabla(R_t^1-R_t^0)\|_{L^2}^2 $ and where we used the trace inequality \eqref{Trace2} in the second inequality.
Taking $T \rightarrow \infty$, and using that $D\geq 1$ we get that
\[\int_{0}^\infty\|R_{t}^1-R^0_{t}\|_{2}^2{\mathrm{d}t}<\infty.\]
By Corollary \ref{C1}, $R^1$ is almost surely decreasing and $R^0$ increasing therefore $R_{t}^1-R^0_{t} $ is almost surely decreasing and the above inequality implies 
\begin{equation*}\label{ER}
    \|R_{t}^1-R^0_{t}\|_{2}^2 \underset{t \rightarrow \infty}{\longrightarrow} 0.
\end{equation*}
We are now left to show that
\begin{equation}\label{rho1}
\underset{t\rightarrow\infty}{\lim}~ \underset{n\rightarrow\infty}{\lim}~ \big[B_n(t)+ C_n(t)\big]=0. 
\end{equation}
We proceed following the same steps as for $A_n$.
\begin{equation} 
    \begin{split}
        B'_n(t) &= -2D\sum_{k=1}^n\gamma_k |\langle\rho_{1,t}^1-\rho_{1,t}^0,W_{k}\rangle|^2 + 2\sum_{k=1}^n\langle F_1(\whr_t^1)-F_1(\whr_t^0),W_k\rangle \langle \rho_{1,t}^1-\rho_{1,t}^0,W_k\rangle\\
        &- 2\sum_{k=1}^n\langle\rho_{1,t}^1-\rho_{1,t}^0,W_k\rangle\int_{\Gamma_+}(\rho_{1,t}^1-\rho_{1,t}^0)(r)\Check{W}_k(r)n_1(r).dS(r).
    \end{split}
\end{equation}
To lighten notations we will not write the subscript $t$ in the computations, when there is no confusion. Let us compute the second term.
\begin{equation}\label{Comput1}
    \begin{split}
        &\sum_{k=1}^n\langle F_1(\whr^1)-F_1(\whr^0),W_k\rangle \langle\rho_1^1-\rho_1^0,W_k\rangle = \ddd\sum_{k=1}^n\langle\rho_1^1(R^1-R^0),W_k\rangle \langle\rho_1^1-\rho_1^0,W_k\rangle\\
        &+\ddd\sum_{k=1}^n\langle R^0(\rho_1^1-\rho_1^0),W_k\rangle \langle\rho_1^1-\rho_1^0,W_k\rangle\\
        &-\ddd\sum_{k=1}^n\langle (\rho_1^1)^2+(\rho_1^0)^2,W_k\rangle \langle \rho_1^1-\rho_1^0,W_k\rangle\\
        &+2d\lambda_2\sum_{k=1}^n\langle R^1(T^1-T^0),W_k\rangle \langle\rho_1^1-\rho_1^0,W_k\rangle +2d\lambda_2\sum_{k=1}^n\langle T^0(R^1-R^0),W_k\rangle \langle\rho_1^1-\rho_1^0,W_k\rangle\\
        &-2d\lambda_2\sum_{k=1}^n\langle T^1(\rho_1^1-\rho_1^0),W_k\rangle \langle\rho_1^1-\rho_1^0,W_k\rangle -2d\lambda_2\sum_{k=1}^n\langle \rho_1^0(T^1-T^0),W_k\rangle \langle \rho_1^1-\rho_1^0,W_k\rangle\\
        &+\sum_{k=1}^n\langle T^1-T^0,W_k\rangle \langle \rho_1^1-\rho_1^0,W_k\rangle-(r+2)\sum_{k=1}^n|\langle\rho_1^1-\rho_1^0,W_k\rangle|^2.
    \end{split}
\end{equation}
Using Lemma \ref{L Att} and the Cauchy-Schwarz inequality, we get:
\begin{equation*}
    \begin{split}
        -\frac{1}{2}B'_n(t) &\geq \sum_{k=1}^n\big[D\gamma_k+r+2-\ddd\big] \big|\langle\rho_{1,t}^1-\rho_{1,t}^0,W_k\rangle\big|^2\\
        &-(1+2d\lambda_2)\sum_{k=1}^n\langle T^1_t-T^0_t,W_k\rangle \langle\rho_1^1-\rho_{1,t}^0,W_k\rangle \\
        &- 2d\lambda_1\sqrt{A_n(t)}\sqrt{B_n(t)} - \sqrt{B_n(t)} \sqrt{\sum_{k=1}^n\Big(\int_{\Gamma_+}(\rho_{1,t}^1-\rho_{1,t}^0)(r)\Check{W}_k(r)n_1(r)dS(r)\Big)^2}.
    \end{split}
\end{equation*}
Integrating this between $0$ and $T$ and using the Cauchy-Scwharz inequality we are left with
\begin{equation}\label{PourB}
    \begin{split}
        \frac{1}{2}\big(B_n(0)-B_n(T) \big) &\geq \int_0^T\sum_{k=1}^n\big[D\gamma_k+r+2-\ddd\big] \big|<\rho_{1,t}^1-\rho_{1,t}^0,W_k>\big|^2{\mathrm{d}t} \\
        &-(1+2d\lambda_2) \int_0^T\sqrt{B_n(t)}\sqrt{C_n(t)}{\mathrm{d}t}-2d\lambda_1 \int_0^T\sqrt{A_n(t)}\sqrt{B_n(t)}{\mathrm{d}t}\\
        &-\sqrt{\int_0^TB_n(t){\mathrm{d}t}}\sqrt{\int_0^T {\sum_{k=1}^n\Big(\int_{\Gamma_+}(\rho_{1,t}^1-\rho_{1,t}^0)(r)\Check{W}_k(r)}n_1(r)dS(r)\Big)^2\mathrm{d}t}.
    \end{split}
\end{equation}
Now
\begin{equation}
    \begin{split}
        C'_n(t)&= -2D\sum_{k=1}^n\gamma_k |\langle T^1_t-T^0_t,W_{k}\rangle|^2 + 2\sum_{k=1}^n\langle H(\whr^1_t)-H(\whr^0_t),W_k\rangle \langle T^1_t-T^0_t,W_k\rangle\\
        &- 2\sum_{k=1}^n\langle T^1_t-T^0_t,W_k\rangle \int_{\Gamma_+}(T^1_t-T^0_t)(r)\Check{W}_k(r)n_1(r).dS(r).
    \end{split}
\end{equation}
Again, we compute the second term using the explicit expression of $H$:
\begin{equation}\label{Comput2}
\begin{split}
    &\sum_{k=1}^n\langle H(\whr^1)-H(\whr^0),W_k\rangle \langle\rho_1^1-\rho_1^0,W_k\rangle = \ddd \sum_{k=1}^n\langle \rho_1^1-\rho_1^0,W_k\rangle \langle T^1-T^0,W_k\rangle\\
    &+(2d\lambda_2-1)\sum_{k=1}^n|\langle T^1-T^0,W_k\rangle|^2 - \ddd\sum_{k=1}^n\langle\rho_1^1(T^1-T^0),W_k\rangle \langle T^1-T^0,W_k\rangle\\
    &-\ddd\sum_{k=1}^n\langle T^0(\rho_1^1-\rho_1^0),W_k\rangle \langle T^1-T^0,W_k\rangle\\
    &-2d\lambda_2\sum_{k=1}^n\langle (T^1)^2-(T^0)^2,W_k\rangle \langle T^1-T^0,W_k\rangle.
\end{split}
\end{equation}
Using Lemma \ref{L Att} and the Cauchy-Schwarz inequality, we get:
\begin{equation*}
\begin{split}
        &-\frac{1}{2}C'_n(t)\geq \sum_{k=1}^n\big[ D\gamma_k+1-2d\lambda_2\big]\big|\langle T^1_t-T^0_t,W_k\rangle \big|^2\\
        &-\sqrt{C_n(t)}\sqrt{\sum_{k=1}^n\Big(\int_{\Gamma_+}(T^1_t-T^0_t)(r)\Check{W}_k(r)n_1(r)dS(r)\Big)^2}-\ddd\sqrt{C_n(t)}\sqrt{B_n(t)}.
\end{split}
\end{equation*}
Integrating this between $0$ and $T$ and using the Cauchy-Scwharz inequality we are left with:
\begin{equation}\label{PourC}
    \begin{split}
        \frac{1}{2}\big(C_n(0)-C_n(T)\big)&\geq \int_0^T\sum_{k=1}^n\big[D\gamma_k+1-2d\lambda_2 \big]\big|\langle T^1_t-T^0_t,W_k\rangle \big|^2{\mathrm{d}t}\\
        &-\ddd\int_0^T \sqrt{C_n(t)}\sqrt{B_n(t)}{\mathrm{d}t}\\
        &-\sqrt{\int_0^TC_n(t){\mathrm{d}t}}\sqrt{\int_0^T\sum_{k=1}^n\Big(\int_{\Gamma_+}(T^1_t-T^0_t)(r)\Check{W}_k(r)n_1(r)dS(r)\Big)^2{\mathrm{d}t}} .
    \end{split}
\end{equation}
Summing inequalities \eqref{PourB} and \eqref{PourC}, using that $B_n$ is uniformly bounded by a constant $K_1$ and $C_n$ by a constant $K_2$ and that for any $a,b>0$, $-\sqrt{a}\sqrt{b}\geq -\frac{1}{2}(a+b)$, we obtain
\begin{equation}
    \begin{split}\label{grosterm}
        &\frac{1}{2}\big(B_n(0)-B_n(T)+ C_n(0)-C_n(T) \big) \geq\\
        &\int_0^T\sum_{k=1}^n\big[D\gamma_k+r+2-\ddd\big] \big|\langle\rho_{1,t}^1-\rho_{1,t}^0,W_k\rangle\big|^2{\mathrm{d}t}\\
        &-\frac{1}{2}\int_0^TB_n(t){\mathrm{d}t} -\frac{1}{2} \int_0^T\sum_{k=1}^n\Big(\int_{\Gamma_+}(\rho_{1,t}^1-\rho_{1,t}^0)(r)\Check{W}_k(r)n_1(r)dS(r)\Big)^2{\mathrm{d}t}\\
        &+ \int_0^T\sum_{k=1}^n\big[D\gamma_k+1-2d\lambda_2\big] \big|\langle T^1_t-T^0_t,W_k\rangle\big|^2{\mathrm{d}t}\\
        &-\frac{1}{2} \int_0^TC_n(t){\mathrm{d}t} -\frac{1}{2} \int_0^T\sum_{k=1}^n\Big(\int_{\Gamma_+}(T^1_t-T^0_t)(r)\Check{W}_k(r)n_1(r)dS(r)\Big)^2{\mathrm{d}t}\\
        &-2dK_1\lambda_1\sqrt{\int_0^TA_n(t){\mathrm{d}t}}-(1+2d\lambda_1)K_2\sqrt{\int_0^TB_n(t){\mathrm{d}t}}.
    \end{split}
\end{equation}
Now, split the first and fourth terms into two parts. Lower bound the first one using that $\gamma_1 < \gamma_k$ for any $k\geq 2$ and we shall then use \eqref{Norme3} to deal with the second term. The left-hand side in \eqref{grosterm} is lower bounded by
\begin{equation}
    \begin{split}
        &\frac{1}{2}\Big[\int_0^T\big(D\gamma_1 + r+2 - 2d(\lambda_1-\lambda_2) \big)\sum_{k=1}^n \big|\langle\rho_{1,t}^1-\rho_{1,t}^0,W_k\rangle\big|^2{\mathrm{d}t} - \int_0^T B_n(t) {\mathrm{d}t} \Big]\\
        & + \frac{1}{2}\Big[\int_0^T\sum_{k=1}^n\big(D\gamma_k+r+2-\ddd\big) \big|\langle\rho_{1,t}^1-\rho_{1,t}^0,W_k\rangle\big|^2{\mathrm{d}t}\\
        &- \int_0^T\sum_{k=1}^n\Big(\int_{\Gamma_+}(\rho_{1,t}^1-\rho_{1,t}^0)(r)\Check{W}_k(r)n_1(r)dS(r)\Big)^2{\mathrm{d}t}  \Big]\\
        & + \frac{1}{2}\Big[\int_0^T \big(D\gamma_1+1-2d\lambda_2\big) \sum_{k=1}^n\big|\langle T^1_t-T^0_t,W_k\rangle\big|^2{\mathrm{d}t} - \int_0^T C_n(t) {\mathrm{d}t}\Big]\\
        & + \frac{1}{2}\Big[\int_0^T  \sum_{k=1}^n\big(D\gamma_k+1-2d\lambda_2\big)\big|\langle T^1_t-T^0_t,W_k\rangle\big|^2{\mathrm{d}t}\\
        &- \int_0^T\sum_{k=1}^n\Big(\int_{\Gamma_+}(T^1_t-T^0_t)(r)\Check{W}_k(r)n_1(r)dS(r)\Big)^2{\mathrm{d}t} \Big]\\
        &- 2d\lambda_1K_1 \sqrt{\int_0^T A_n(t){\mathrm{d}t}} - (1 + 2d\lambda_1)K_2 \sqrt{\int_0^T B_n(t){\mathrm{d}t}}.
    \end{split}
\end{equation}
Taking $n$ to infinity, using \eqref{ordrevp2} and the dominated convergence theorem as well as the trace inequality stated in Theorem \ref{Trace} and using that $D\geq 1$ and $\gamma_1>1$ (see \eqref{gama})  we get:
\begin{equation}
    \begin{split}\label{timeint}
       &\big(B(0)-B(T)+ C(0)-C(T) \big) \geq\\
       &2\big(r+1-\ddd \big) \int_0^T B(t){\mathrm{d}t} + (D-1) \int_0^T\|\nabla(\rho_{1,t}^1-\rho_{1,t}^0)\|^2{\mathrm{d}t} \\
       & + 2(1-2d\lambda_2) \int_0^TC(t){\mathrm{d}t} + (D-1)\int_0^T\|\nabla(T^1_t-T^0_t)\|^2{\mathrm{d}t} \\
       &- 4d\lambda_1 K_1 \sqrt{\int_0^T A(t){\mathrm{d}t}} - (1+2d\lambda_1)K_2\sqrt{\int_0^T B(t){\mathrm{d}t}}. 
    \end{split}
\end{equation}
Since conditions $(H_1)$ hold, all the factors between the time integrals $\int_0^TB(t){\mathrm{d}t}$ and $\int_0^TC(t){\mathrm{d}t}$ are strictly positive and inequality \eqref{timeint} implies that
\[\int_0^{\infty}B(t){\mathrm{d}t}<\infty,~ ~ \text{and}~ \int_0^{\infty}C(t){\mathrm{d}t}<\infty . \]
Again, by Corollary \ref{C1}, $\rho_1^1$ and
$T^1$ are almost surely decreasing and $\rho_1^0$ and $T^0$ increasing, therefore $\rho_{1}^1-\rho_1^0$ and $T^1-T^0$ are almost surely decreasing and the above inequalities imply
\begin{equation*}
    \|\rho_{1,t}^1-\rho_{1,t}^0\|_{2}^2 \underset{t \rightarrow \infty}{\longrightarrow} 0,~ ~ \text{and}~ ~ \|T_{t}^1-T_{t}^0\|_{2}^2 \underset{t \rightarrow \infty}{\longrightarrow} 0.
\end{equation*}
For the proof in the (Neumann ; Robin) regime, one proceeds in the same way, but decomposing the difference between $\whr^1$ and $\whr^0$ on the basis $(V_k)_{k\geq 1}$ and using conditions $(H_2)$.
\end{proof}
Now, we are able to prove Theorem \ref{T3}.
\begin{proof}[Proof of Theorem \ref{T3}]
    Again, we focus on the (Dirichlet ; Robin) regime and the proof is the same for all the others. As said before, it is enough to prove uniqueness of a solution of 
\begin{equation} \label{Dirichletmodifiestat}
    \left\{
    \begin{array}{ll}
        D\Delta \rho_1 + F_1(\rho_1,T,R)=0,~~ ~  \rho_{1_{|\Gamma^-}} = b_1(.),~~~\partial_{e_{1}}\rho_1(t,.)_{|\Gamma^{+}}=\frac{1}{D}(b_1-\rho_1)_{|\Gamma^{+}}\\ 
        D\Delta T + H(\rho_1,T,R)=0,~T_{|\Gamma^-}= b_1(.)+b_3(.),~\partial_{e_{1}}T(t,.)_{|\Gamma^{+}}=\frac{1}{D}(b_1+b_3-\rho_1-\rho_3)_{|\Gamma^{+}}\\
        D\Delta R + J(R)=0,~  R_{|\Gamma^-}=1-b_2(.)-b_3(.),~~~\partial_{e_{1}}R(t,.)_{|\Gamma^{+}}=\frac{1}{D}(\rho_2+\rho_3-b_1-b_3)_{|\Gamma^{+}}.
    \end{array}
\right.
\end{equation}
\begin{itemize}
    \item [(i)] \textit{Existence:} For $n\in \N$, define
    \begin{equation}
        U^0_{n}= \{u\in B,~ \rho_{1}^0(n,.)\leq \rho_{1}^0(n+1,.),~ T^0(n,.)\leq T^0(n+1,.),~ R^0(n,.)\leq R^0(n+1,.)\}
    \end{equation}
    and
    \begin{equation}
        U^1_{n}=\{u\in B,~ \rho_{1}^1(n,.)\leq \rho_{1}^1(n+1,.),~ T^1(n,.)\leq T^1(n+1,.),~ R^1(n,.)\leq R^1(n+1,.)\}.
    \end{equation}
    By Corollary \ref{C1}, the above sets are almost sure and so is $U:=\underset{n\geq 0}{\cap}(U_{n}^0\cap U_{n}^1)$. On $U$, the sequence of profiles $\{\whr^1(n,.),~ n \geq 1\}$ (resp.$\{\whr^0(n,.),~ n \geq 1\}$) decreases (resp. increases) to a limit that we denote by $\whr^+(.) = (\rho^+_{1}(.),T^+(.),R^+(.))$ (resp. $\whr^-(.) = (\rho^-_{1}(.),T^{-}(.),R^-(.))$). By Lemma \ref{L2}, $\whr^+=\whr^-$ everywhere on $U$ so almost surely on $B$ . Denote this profile by $\underline{\rho}$ and consider $\underline{\rho}(t,.)$ the solution to \eqref{Dirichletmodifiestat} with initial condition $\underline{\rho}$. Since for all $t\geq 0$, $\whr^0(t,.) \leq \underline{\rho}(.) \leq \whr^1(t,.)$ almost surely, by Lemma \ref{L Att} we have that for every $s,t\geq 0$, $\whr^0(t+s,.) \leq \underline{\rho}(s,.) \leq \whr^1(t+s,.)$ almost surely and letting $t \rightarrow \infty$ we get that $\underline{\rho}(s,.) = \underline{\rho}(.)$ for all $s$ so $\underline{\rho}$ is a solution of \eqref{Dirichletmodifiestat}.
    \item[(ii)] \textit{Uniqueness:} Note that by Lemma \ref{L Att} and Corollary \ref{C1}, for any profiles $\whr^a=(\rho_1^a,T^a,R^a)$ and $\whr^b=(\rho_1^b,T^b,R^b)$ satisfying \eqref{Dirichletmodifie} with any initial condition, for every $t>0$
    \begin{equation}\label{Uniq}
    \begin{split}
        & \int_{B}\Big(\big| \rho_{1}^a(t,u) - \rho_{1}^b(t,u)|+\big| T^a(t,u) - T^b(t,u)|+\big| R^a(t,u) - R^b(t,u)|\Big){\mathrm{d}u}\\
        &\leq \int_{B}\Big(\big| \rho_{1}^1(t,u) - \rho_{1}^0(t,u)|+\big| T^1(t,u) - T^0(t,u)|+\big| R^1(t,u) - R^0(t,u)|\Big){\mathrm{d}u}.
    \end{split}
    \end{equation}
    Applying \eqref{Uniq} to two stationary solutions and using Lemma \ref{L2}, one gets uniqueness.
\end{itemize}
\noindent As said before, existence and uniqueness of a solution $\underline{\rho}$ of \eqref{Dirichletmodifiestat} yields existence and uniqueness of the stationary solution of \eqref{D+R}. Similarly, the proof of \eqref{P1} comes from the fact that 
\begin{equation}
    \begin{split}
        &\int_B\Big(|\rho_1(t,u)-\underline{\rho_1}(u)|+ |T(t,u)-\underline{T}(u)|+ |R(t,u)-\underline{R}(u)| \Big)du\\
        &\leq \int_{B}\Big(\big| \rho_{1}^1(t,u) - \rho_{1}^0(t,u)|+\big| T^1(t,u) - T^0(t,u)|+\big| R^1(t,u) - R^0(t,u)|\Big)du
    \end{split}
\end{equation}
where again, we applied \eqref{Uniq} and the fact that the right-hand side term converges to $0$.
\end{proof}

\appendix
\section{Change of variable formulas}
The following change of variable formulas have been established in \cite[Section 5.2]{KMS}. Recall that for $i,j\in \{0,1,2,3\}$ and $x,y\in B_N$, $v_{j}(x/N) = \log(\alpha_{j}(x/N))$, and 
\begin{equation*}
    R_{i,j}^{x,y}(\widehat{\alpha})=\exp\Big(\big(v_j(y/N)-v_{j}(x/N)\big)-\big(v_i(y/N)-v_{i}(x/N)\big)\Big)-1.
\end{equation*}
Note that $ R_{i,j}^{x,y}(\widehat{\alpha})= O(N^{-1})$. Consider $f:\whsN\rightarrow \R$ and $x,y\in B_{N}$.
\begin{itemize}
    \item [(i)] For $(i,j)\in \{0,1,2,3\}^2$ such that $i\neq j$,
    \begin{equation}\label{chtva1}
        \int_{\whsN}\e_i(x)\e_j(y)f(\xi^{x,y},\omega^{x,y})d\nualph\xiw = \int_{\whsN}\e_j(x)\e_{i}(y)(R_{i,j}^{x,y}(\widehat{\alpha})+1)f\xiw d\nualph\xiw.
    \end{equation}
    \item [(ii)] For $(i,j)\in \{0,1,2,3\}^2$ such that $i\neq j$,
\begin{equation}
\begin{split}\label{chtva2}
    \int\e_{i}(x)b_{j}(x/N)f\xiw d\nualph\xiw = \int\e_j(x) b_i(x/N) f(\sigma_{i,x}\xiw) d\nualph\xiw,
\end{split}
\end{equation}
\end{itemize}
To prove both points, we use the explicit expression of $\nualph$. Let us give the details for (ii). Take $(i,j)\in \{0,1,2,3\}^2$ with $i\neq j$, then,
\begin{equation*}
\begin{split}
    \int\e_{i}(x)b_{j}(x/N)f\xiw d\nualph\xiw&= \int_{\xiw,~ \e_i(x)=1}b_{j}(x/N)f(\sigma_{i,x}\xiw) d\nualph\xiw\\
    &= \int_{(\check{\xi},\check{\omega})\in \widehat{\Sigma}_{N-1}} b_j(x/N) f(\sigma_{i,x}\xiw) \frac{b_i(x/N)}{b_0(x/N)}d\nu^{N-1}_{\widehat{\alpha}}(\check{\xi},\check{\omega})\\
    &= \int\e_j(x) b_i(x/N) f(\sigma_{i,x}\xiw) d\nualph\xiw,
\end{split}
\end{equation*}
because \[\frac{b_j(x/N)}{b_0(x/N)}\nu^{N-1}_{\widehat{\alpha}}(\check{\xi},\check{\omega}) = \nu^N_{\widehat{\alpha}}\big\{\e_j(x)=1,\xiw_{|B_N\setminus\{x\}}=(\check{\xi},\check{\omega})\big\}.\] 

\section{Simulations} \label{Simuu}
The hydrodynamic equations with Neumann boundary conditions 
\begin{equation}\label{N+N}
    \left\{
    \begin{array}{ll}
        \partial_{t}\whr = D\Delta \whr + \widehat{F}(\whr)~ \text{in}~ B\times (0,T),\\\\
        \partial_{e_{1}}\whr(t,.)_{|\Gamma}=0~ \text{for}~ 0<t\leq T,
    \end{array}
\right.
\end{equation}
have been simulated in dimension $1$ with $B=[0,1]$. For that, we used an Euler explicit scheme and chose the following parameters:
\begin{itemize}
\item Time horizon: T=100
\item Time subdivision: $\delta_T= 5.10^5$
\item Space subdivision: $\delta_x= 100$
\item $r=1$ and $D=1$.
\end{itemize}
In Figures \ref{fig:figuresAB} and \ref{fig:figuresCD}, the $x$ axis corresponds to the one dimensional space $B=[0,1]$ and the $y$ axis is the space of values of the density profiles $\rho_1$, $\rho_1+\rho_3$, $1-(\rho_2+\rho_3)$.

In the first simulation, we took $\lambda_1=0.75$ and $\lambda_2=0.25$, so the conditions $(H_1)$ are satisfied. In Figure \ref{fig : SimuA} we presented the solution of \eqref{N+N} at time $T=100$, with initial condition $(\rho_1,\rho_1+\rho_3, 1-\rho_2-\rho_3)(0,.)=(0,0,0)$  and in Figure \ref{fig: SimuB} the solution of \eqref{N+N} at time $T=100$, with initial condition $(\rho_1,\rho_1+\rho_3, 1-\rho_2-\rho_3)(0,.)=(1,1,1)$ . As expected (see Theorem \ref{T3}), both profiles in Figure \ref{fig:figuresAB} coincide.

In the second simulation, we took $\lambda_1=1$ and $\lambda_2=0.75$, so the conditions $(H_1)$ are not satisfied. In Figure \ref{fig:simuC} we presented the solution of \eqref{N+N} at time $T=100$, with initial condition $(\rho_1,\rho_1+\rho_3, 1-\rho_2-\rho_3)(0,.)=(0,0,0)$ and in Figure \ref{fig:simuD} the solution of \eqref{N+N} at time $T=100$, with initial condition $(\rho_1,\rho_1+\rho_3, 1-\rho_2-\rho_3)(0,.)=(1,1,1)$ . The profiles in Figure \ref{fig : SimuA} and \ref{fig: SimuB} do not coincide, which proves, numerically that conditions on the parameters are needed for both these limiting profiles to coincide.\medskip

\begin{figure}
    \begin{subfigure}[b]{0.5\linewidth}
        \centering
        \includegraphics[width=\linewidth]{MSVfig1.png}
        \caption{$\big(\rho_1,\rho_1+\rho_3, 1-\rho_2-\rho_3\big)(T,.)$, solution of \eqref{N+N}\\ with initial condition $(0,0,0)$.}
        \label{fig : SimuA}
    \end{subfigure}%
    \hfill
    \begin{subfigure}[b]{0.5\linewidth}
        \centering
        \includegraphics[width=\linewidth]{MSVfig2.png}
        \caption{$\big(\rho_1,\rho_1+\rho_3, 1-\rho_2-\rho_3\big)(T,.)$, solution of \eqref{N+N}\\ with initial condition $(1,1,1)$.}
        \label{fig: SimuB}
    \end{subfigure}
    \caption{$T=100$, $\lambda_1=0.75$, $\lambda_2=0.25$. Conditions $(H_1)$ are satisfied.}
    \label{fig:figuresAB}
\end{figure}
\begin{figure}
    \begin{subfigure}[b]{0.5\linewidth}
        \centering
        \includegraphics[width=\linewidth]{MSVfig3.png}
        \caption{$\big(\rho_1,\rho_1+\rho_3, 1-\rho_2-\rho_3\big)(T,.)$, solution of \eqref{N+N}\\ with initial condition $(0,0,0)$.}
       \label{fig:simuC}
    \end{subfigure}%
    \hfill
    \begin{subfigure}[b]{0.5\linewidth}
        \centering
        \includegraphics[width=\linewidth]{MSVfig4.png}
        \caption{$\big(\rho_1,\rho_1+\rho_3, 1-\rho_2-\rho_3\big)(T,.)$, solution of \eqref{N+N}\\ with initial condition $(1,1,1)$.}
        \label{fig:simuD}
    \end{subfigure}
    \caption{$T=100$, $\lambda_1=1$, $\lambda_2=0.75$. Conditions $(H_1)$ are not satisfied.}
    \label{fig:figuresCD}
\end{figure}

\newpage
\noindent \textbf{Acknowledgements:} We thank the anonymous referees for their careful reading of the paper and comments which allowed to notably improve it. The authors would also like to thank Frank Redig for an interesting discussion as well as Camille Pouchol for some useful advice regarding numerical simulations. This work has been conducted within the FP2M federation (CNRS FR 2036).

\bibliographystyle{plain}
\bibliography{MSVbiblio.bib}

\begin{thebibliography}{10}

\bibitem{Vauchelet}
Luis Almeida, Michel Duprez, Yannick Privat, and Nicolas Vauchelet.
\newblock Optimal control strategies for the sterile mosquitoes technique.
\newblock {\em J. Differential Equations}, 311:229--266, 2022.

\bibitem{baldasso_exclusion_2017}
Rangel Baldasso, Ot\'{a}vio Menezes, Adriana Neumann, and Rafael~R. Souza.
\newblock Exclusion process with slow boundary.
\newblock {\em J. Stat. Phys.}, 167(5):1112--1142, 2017.

\bibitem{billingsley_convergence_1999}
Patrick Billingsley.
\newblock {\em Convergence of probability measures}.
\newblock Wiley Series in Probability and Statistics: Probability and
  Statistics. John Wiley \& Sons, Inc., New York, second edition, 1999.
\newblock A Wiley-Interscience Publication.

\bibitem{Borrello}
Davide Borrello.
\newblock Stochastic order and attractiveness for particle systems with
  multiple births, deaths and jumps.
\newblock {\em Electron. J. Probab.}, 16:no. 4, 106--151, 2011.

\bibitem{sau}
Lorenzo Dello~Schiavo, Lorenzo Portinale, and Federico Sau.
\newblock Scaling limits of random walks, harmonic profiles, and stationary
  non-equilibrium states in lipschitz domains, (2022), arXiv: 2112.14196 (to
  appear in \textit{Annales of Applied Probability}).

\bibitem{Derrida}
Bernard Derrida.
\newblock Non-equilibrium steady states: fluctuations and large deviations of
  the density and of the current.
\newblock {\em J. Stat. Mech. Theory Exp.}, (7):P07023, 45 pp., 2007.

\bibitem{TIS}
V.A. Dyck, J.~Hendrichs, and A.S. Robinson.
\newblock {\em Sterile Insect Technique Principles and Practice in Area-Wide
  Integrated Pest Management}.
\newblock Springer, 2005.

\bibitem{evans_partial_2010}
Lawrence~C. Evans.
\newblock {\em Partial differential equations}, volume~19 of {\em Graduate
  Studies in Mathematics}.
\newblock American Mathematical Society, Providence, RI, second edition, 2010.

\bibitem{Fine}
Lawrence~C. Evans and Ronald~F. Gariepy.
\newblock {\em Measure theory and fine properties of functions}.
\newblock Textbooks in Mathematics. CRC Press, Boca Raton, FL, revised edition,
  2015.

\bibitem{Lebo}
Gregory Eyink, Joel~L. Lebowitz, and Herbert Spohn.
\newblock Hydrodynamics of stationary nonequilibrium states for some stochastic
  lattice gas models.
\newblock {\em Comm. Math. Phys.}, 132(1):253--283, 1990.

\bibitem{farfan_hydrostatics_2011}
Jonathan Farfan, Claudio Landim, and Mustapha Mourragui.
\newblock Hydrostatics and dynamical large deviations of boundary driven
  gradient symmetric exclusion processes.
\newblock {\em Stochastic Process. Appl.}, 121(4):725--758, 2011.

\bibitem{franco_hydrodynamical_2013}
Tertuliano Franco, Patr\'{\i}cia Gon\c{c}alves, and Adriana Neumann.
\newblock Hydrodynamical behavior of symmetric exclusion with slow bonds.
\newblock {\em Ann. Inst. Henri Poincar\'{e} Probab. Stat.}, 49(2):402--427,
  2013.

\bibitem{TT}
Tertuliano Franco, Patr\'{\i}cia Gon\c{c}alves, and Adriana Neumann.
\newblock Phase transition in equilibrium fluctuations of symmetric slowed
  exclusion.
\newblock {\em Stochastic Process. Appl.}, 123(12):4156--4185, 2013.

\bibitem{Patricia}
Patr\'{\i}cia Gon\c{c}alves.
\newblock Hydrodynamics for symmetric exclusion in contact with reservoirs.
\newblock In {\em Stochastic dynamics out of equilibrium}, volume 282 of {\em
  Springer Proc. Math. Stat.}, pages 137--205. Springer, Cham, 2019.

\bibitem{cmp/1104161907}
M.~Z. Guo, G.~C. Papanicolaou, and S.~R.~S. Varadhan.
\newblock Nonlinear diffusion limit for a system with nearest neighbor
  interactions.
\newblock {\em Comm. Math. Phys.}, 118(1):31--59, 1988.

\bibitem{Durrett}
Xiangying Huang and Rick Durrett.
\newblock A stochastic spatial model for the sterile insect control strategy.
\newblock {\em Stochastic Process. Appl.}, 157:249--278, 2023.

\bibitem{kipnis_scaling_1999}
Claude Kipnis and Claudio Landim.
\newblock {\em Scaling limits of interacting particle systems}, volume 320 of
  {\em Grundlehren der mathematischen Wissenschaften [Fundamental Principles of
  Mathematical Sciences]}.
\newblock Springer-Verlag, Berlin, 1999.

\bibitem{knipling_possibilities_1955}
E.~F. Knipling.
\newblock Possibilities of {Insect} {Control} or {Eradication} {Through} the
  {Use} of {Sexually} {Sterile} {Males}.
\newblock {\em Journal of Economic Entomology}, 48(4):459--462, August 1955.

\bibitem{kuoch:hal-01100145}
Kevin Kuoch.
\newblock Phase transition for a contact process with random slowdowns.
\newblock {\em Markov Process. Related Fields}, 22(1):53--85, 2016.

\bibitem{KMS}
Kevin Kuoch, Mustapha Mourragui, and Ellen Saada.
\newblock A boundary driven generalized contact process with exchange of
  particles: hydrodynamics in infinite volume.
\newblock {\em Stochastic Process. Appl.}, 127(1):135--178, 2017.

\bibitem{LMS}
C.~Landim, M.~Mourragui, and S.~Sellami.
\newblock Hydrodynamic limit for a nongradient interacting particle system with
  stochastic reservoirs.
\newblock {\em Teor. Veroyatnost. i Primenen.}, 45(4):694--717, 2000.

\bibitem{article}
C.~Landim and K.~Tsunoda.
\newblock Hydrostatics and dynamical large deviations for a reaction-diffusion
  model.
\newblock {\em Ann. Inst. Henri Poincar\'{e} Probab. Stat.}, 54(1):51--74,
  2018.

\bibitem{liggett_interacting_2005}
Thomas~M. Liggett.
\newblock {\em Interacting particle systems}.
\newblock Classics in Mathematics. Springer-Verlag, Berlin, 2005.
\newblock Reprint of the 1985 original.

\bibitem{Roubi}
Tom\'{a}\v{s} Roub\'{\i}\v{c}ek.
\newblock {\em Nonlinear partial differential equations with applications},
  volume 153 of {\em International Series of Numerical Mathematics}.
\newblock Birkh\"{a}user/Springer Basel AG, Basel, second edition, 2013.

\bibitem{Manifold}
Bart van Ginkel and Frank Redig.
\newblock Hydrodynamic limit of the symmetric exclusion process on a compact
  {R}iemannian manifold.
\newblock {\em J. Stat. Phys.}, 178(1):75--116, 2020.

\end{thebibliography}

\end{document}